\documentclass[11pt]{amsart}
\usepackage[T1]{fontenc}
\usepackage{lmodern}
\usepackage[margin=1in]{geometry}
\usepackage{amsmath,amssymb,mathtools,mathrsfs}
\usepackage{microtype}
\usepackage{hyperref}
\hypersetup{hidelinks,pdftitle={Fourier quasicrystals with positive integer masses are sections of Lee--Yang cycles},pdfauthor={Dongsheng Wei}}
\newcommand{\R}{\mathbb R}
\newcommand{\C}{\mathbb C}
\newcommand{\Z}{\mathbb Z}
\newcommand{\Q}{\mathbb Q}
\newcommand{\N}{\mathbb N}
\newcommand{\T}{\mathbb T}
\newcommand{\PP}{\mathbb P}
\newcommand{\cA}{\mathcal A}
\newcommand{\varn}{\operatorname{var}}
\newcommand{\vbar}{\overline{\operatorname{var}}}
\DeclareMathOperator{\supp}{supp}
\DeclareMathOperator{\Gr}{Gr}

\DeclareMathOperator{\Log}{Log}

\DeclareMathOperator{\Frac}{Frac}

\newtheorem{maintheorem}{Theorem}

\newtheorem{theorem}{Theorem}[section]
\newtheorem{lemma}[theorem]{Lemma}
\newtheorem{proposition}[theorem]{Proposition}
\newtheorem{corollary}[theorem]{Corollary}
\theoremstyle{definition}
\newtheorem{definition}[theorem]{Definition}
\theoremstyle{remark}
\newtheorem{remark}[theorem]{Remark}
\newtheorem{example}[theorem]{Example}
\newtheorem{question}[theorem]{Question}
\numberwithin{equation}{section}
\title[Integer-mass Fourier quasicrystals]{Fourier quasicrystals with positive integer masses are sections of Lee--Yang cycles}
\author{Dongsheng Wei}
\address{Shenzhen, China}
\date{}
\begin{document}
\begin{abstract}
We prove that Fourier quasicrystal measures with positive integer masses
are precisely the nonzero real sections of effective Lee--Yang cycles
by matrices with positive maximal minors, with masses given by weighted
local intersection degrees. For Delone support, the reduced support of
the cycle can be chosen strict. For unit masses, the cycle can be chosen
reduced, with local intersection degree one at every root. Every
integer-mass FQ measure has support equal to the common complex zero
set of finitely many real-valued trigonometric polynomials. These results
answer Questions (1), (2), and (4) of Alon,
Kummer, Kurasov, and Vinzant. We recover a finite-dimensional torus from
the spectrum and prove that its canonical analytic set is algebraic.
A support-height theorem for holonomic sequences provides a linear
spectral bound, from which we reconstruct the complex branches and their
integer coefficients using fiber moments. Singular sections carry their
natural local intersection degrees; discarding these can destroy the
Fourier quasicrystal property.
\end{abstract}
\maketitle
\setcounter{tocdepth}{1}
\tableofcontents

\section{Introduction}
\label{sec:introduction}

Poisson summation exhibits discrete measures with discrete Fourier
transforms. We ask what geometric structure this simultaneous
discreteness imposes when the masses are positive integers.

We characterize integer-mass Fourier quasicrystal measures as the
nonzero real sections of effective Lee--Yang cycles by matrices with
positive maximal minors. Each mass is a weighted local intersection
degree; for Delone support, the cycle can be chosen with strict
Lee--Yang reduced support.

Limits of translates of a unit-mass measure may have atoms of higher
mass. These multiplicities are recorded by local intersection degrees.
The reduced support of the representing cycle determines the equations
for the point set.

\begin{definition}[Fourier quasicrystals]\label{def:fourier-quasicrystal}
Let $B_R$ be the Euclidean ball of radius $R$ centered at zero.
A Fourier quasicrystal measure on $\R^d$ is a measure
$\mu=\sum_{x\in\Lambda}a_x\delta_x$ whose Fourier transform is a measure
$\widehat\mu=\sum_{\xi\in\Lambda'}c_\xi\delta_\xi$, where $\Lambda$ and
$\Lambda'$ are locally finite, the displayed coefficients are nonzero, and
\begin{equation}\label{eq:fq-measure-growth}
 \sum_{x\in\Lambda\cap B_R}|a_x|+
 \sum_{\xi\in\Lambda'\cap B_R}|c_\xi|\le CR^{N_0}\qquad(R>1)
\end{equation}
for some $C>0$ and $N_0\in\N$. As in
\cite[Definition~1.1]{AKKV}, a Fourier quasicrystal set, or FQ-set, is a
set $\Lambda$ whose counting measure $\mu_\Lambda=\sum_{x\in\Lambda}\delta_x$
is a Fourier quasicrystal measure. Its growth condition reads
\begin{equation}
 \#(\Lambda\cap B_R)+\sum_{\xi\in\Lambda'\cap B_R}|c_\xi|
 \le CR^{N_0}\qquad(R>1).
 \label{eq:fq-growth}
\end{equation}
\end{definition}

\begin{definition}[Integer-mass measures]\label{def:integer-mass-measure}
An integer-mass measure on $\R^d$ is a nonzero measure
\[
 \mu=\sum_{x\in\Lambda}a_x\delta_x,
 \qquad a_x=\mu(\{x\})\in\Z_{>0},
\]
where $\Lambda=\supp\mu$ is locally finite. It is an integer-mass Fourier
quasicrystal measure if it satisfies
Definition~\ref{def:fourier-quasicrystal}, with the masses $a_x$.
An FQ-set is defined by its \emph{counting measure}, so the support of
an integer-mass FQ measure need not be
an FQ-set; see Example~\ref{ex:two-circles}.
\end{definition}

We use the Fourier kernel $e^{-2\pi i\langle x,\xi\rangle}$, so the
Fourier identity for an FQ measure reads
\begin{equation}\label{eq:intro-fourier-identity}
 \sum_{x\in\Lambda}a_x\widehat f(x)
   =\sum_{\xi\in\Lambda'}c_\xi f(\xi),\qquad f\in\mathcal S(\R^d).
\end{equation}
The growth condition ensures absolute convergence
(Remark~\ref{rem:weighted-schwartz}); thereafter the proofs use only
local finiteness of $\Lambda'$. The polynomial bound on the number of
spectral points follows from Theorem~\ref{thm:main-height}.

\subsection{Background}
For the lattice examples, Poisson summation states that if $\mathcal L\subset\R^d$
is a lattice and $\mathcal L^*$ is its dual, then
\[
 \widehat{\sum_{x\in\mathcal L}\delta_x}
   =\operatorname{covol}(\mathcal L)^{-1}
       \sum_{\xi\in\mathcal L^*}\delta_\xi.
\]
Crystalline measures are tempered measures with locally finite support
whose Fourier transforms are measures with locally finite support;
the weighted growth bound is not imposed. Meyer constructed examples
beyond finite combinations of lattice combs \cite{Meyer}.
For general coefficients the weighted growth condition is not automatic: there are crystalline measures on $\R$ that are not Fourier quasicrystal measures \cite{Favorov2024,MRS2026}.

The inverse problem depends on the distinction between local finiteness
and uniform discreteness of the spectrum. Lev and Olevskii proved
that a measure with uniformly discrete support and spectrum is a finite
sum of translated lattice combs multiplied by trigonometric polynomials,
for complex measures in one dimension and for positive measures in every
dimension \cite[Theorems~1--3]{LevOlevskii}. In the present setting the
spectrum is only locally finite. Kurasov and Sarnak \cite{KurasovSarnak}
constructed explicit positive crystalline measures and Fourier
quasicrystal measures from stable polynomials. The one-dimensional
description by real-rooted exponential polynomials is due to
Olevskii--Ulanovskii \cite{OU2020}; Alon--Cohen--Vinzant \cite{ACV}
then obtained a Lee--Yang polynomial representation. In higher
dimensions, Lawton and Tsikh \cite{LawtonTsikh} study Fourier quasicrystals
in $\R^n$ defined by systems of real-rooted trigonometric polynomials.
Earlier, Lawton and Tsikh
\cite{LawtonTsikhSufficient} gave sufficient
conditions for measures with positive integer weights to be Fourier
quasicrystals and suggested that these conditions might also be necessary.
Our inverse theorems represent every such Fourier quasicrystal as a section
of an effective Lee--Yang cycle, the mass of each point being a weighted
local intersection degree; Proposition~\ref{prop:integer-equation-obstruction}
shows that these masses need not be local colengths of any finite system
of trigonometric equations vanishing on the support.

Here toral type means that the Fourier spectrum generates a finitely
generated group \cite[\S3.2]{LawtonTsikh}. Theorem~\ref{thm:all-fq-toral-cycles}
proves this for integer-mass FQ measures in every dimension and identifies
their hulls with finite-dimensional tori, answering the question of
Lawton and Tsikh within this class \cite[\S6, Question~2]{LawtonTsikh}.
At his KTH seminar of 24 March 2026, Kummer announced the joint result
with Alon that every uniformly discrete Fourier quasicrystal is almost
periodic of toral type \cite{KTH}. A broader statement appears in
\cite[\S4.2, Theorem~3]{Lawton2026}: for Fourier quasicrystals with
positive integer coefficients and a uniform bound on the number of
points in unit balls, toral type is attributed there to a forthcoming
paper of Alon and Kummer (personal communication). The total-mass bound
in Lemma~\ref{lem:ac-translation-bound} places every integer-mass FQ
measure under this announcement, which covers the toral-type conclusion
of Theorem~\ref{thm:all-fq-toral-cycles}. Our independent proof also
describes the analytic hull through collisions, including its local
cycles and their multiplicities. For crystalline measures, strong
almost periodicity is characterized by translation boundedness
\cite[Theorem~3.9]{MRS2026}; the latter is proved here in
Lemma~\ref{lem:ac-translation-bound}
(cf.\ \cite[Lemma~3.10]{MRS2026}).

Alon--Kummer \cite{AKHypersurfaces} relate measures on periodic
hypersurfaces whose Fourier transforms are supported in cones to
Lee--Yang polynomials. Their result connects Fourier cones to algebraic
geometry; our reconstruction treats arbitrary codimension and branched
fibers.

\subsection{Lee--Yang sections and the questions of AKKV}\label{subsec:intro-lee-yang}
For a variety $X\subset(\PP^1)^n$ and a real $n\times d$ matrix $L$, set
\[
 \Lambda(X,L)=\{x\in\C^d:\exp(2\pi iLx)\in X\}.
\]
The Lee--Yang condition requires reciprocal conjugation invariance and
excludes logarithmic modulus vectors in positive planes. On the
algebraic torus, the amoeba of $X$ must avoid every positive $d$-plane
away from the origin. Such a plane is represented by a matrix with
positive maximal minors; positive lines are spanned by vectors with
positive coordinates. Section~\ref{sec:preliminaries} gives the full
definition, including the projective boundary and strictness. We use
reduced varieties throughout.

If $L$ has positive maximal minors and $X$ is Lee--Yang, every
point of $\Lambda(X,L)$ is real, since
\[
 \Log\exp(2\pi iLx)=-2\pi L\operatorname{Im}x\in L\R^d,
\]
and amoeba avoidance forces this vector, hence $\operatorname{Im}x$,
to vanish. For strict Lee--Yang varieties the construction gives
Delone FQ-sets \cite[Theorem~2.3]{AKKV}.

\begin{theorem}[Alon--Kummer--Kurasov--Vinzant]
\label{thm:akkv-construction}
If $X\subset(\PP^1)^n$ is strict Lee--Yang of codimension $d$ and
$L\in\R^{n\times d}$ has positive maximal minors, then $\Lambda(X,L)$
is a real Delone FQ-set. Its spectrum is contained in
$\{L^tk:k\in\Z^n,\ \varn(k)<d\}$.
\end{theorem}

We prove converses and extend the construction to singular sections,
counted with local intersection degrees. We address the following
questions from \cite[\S2.4]{AKKV}, quoted with their original numbering:
\begin{quote}\small
(1) Can all Delone FQ-sets in every dimension be obtained via our
construction, namely as $\Lambda(X,L)$ with strict Lee--Yang $X$ and $L$
with positive minors?

(2) If $X$ is Lee--Yang but $X(\T)$ is not smooth, is it still true that
$\Lambda(X,L)$ is an FQ-set, possibly with multiplicities? If so, can every
FQ-set be obtained in this way?

(4) Can all FQ-sets in every dimension be obtained as the common zeros
of a real-rooted system of trigonometric equations?
\end{quote}

As in \cite[Remark~2.4]{AKKV}, a real-rooted system may have more
than $d$ equations in $d$ variables; real-rooted means that every common
complex zero is real. AKKV note that their questions have positive
answers in dimension one, and prove Question~(1) for the Delone sets
arising from the Lawton--Tsikh construction
\cite[\S2.4 and Corollary~8.6]{AKKV}.

\subsection{Main results}
An effective Lee--Yang cycle assigns positive integer coefficients to
its irreducible components; its section masses are the corresponding
weighted local intersection degrees. The cycles $\mathfrak X$ and
section measures $\nu^{\mathfrak X}_\theta$ are defined in
Definition~\ref{def:effective-lee-yang-cycle}. We begin with Delone
support (Theorem~\ref{thm:main-delone}), then remove separation and
include the unit-mass case in Theorem~\ref{thm:main-equations}. The
spectral group is the group generated by the Fourier spectrum.

\begin{maintheorem}[Delone classification]
\label{thm:main-delone}
An integer-mass measure on $\R^d$ is an FQ measure with Delone support
if and only if it is a section of an effective Lee--Yang cycle with
strict Lee--Yang reduced support, by a matrix with positive maximal
minors, counted with weighted local intersection degrees. Here Delone
support may be replaced by uniformly discrete support: every nonzero
integer-mass FQ measure has relatively dense support.
\end{maintheorem}

\begin{maintheorem}[Classification and support equations]
\label{thm:main-equations}
An integer-mass measure $\mu$ is an FQ measure if and only if
$\mu=\nu^{\mathfrak X,L}_0$ for an effective Lee--Yang cycle $\mathfrak X$
of codimension $d$ in $(\PP^1)^n$ for some $n\ge d$, and a matrix
$L\in\R^{n\times d}$ with positive maximal minors. The support of the cycle can be required
to satisfy ordinary sign variation off the torus. The support of $\mu$
is the common complex zero set of finitely many trigonometric polynomials
real-valued on $\R^d$, with frequencies in its spectral group.

For unit masses the cycle can be chosen reduced, with local intersection
degree one at every physical root, equivalently smoothness and
transversality there.
\end{maintheorem}

\begin{remark}\label{rem:smooth-cycle-coefficients}
Smoothness alone does not force a weighted cycle to have coefficient one.
For a strict Lee--Yang variety $X$ and a matrix with positive maximal
minors, the reduced support of the cycle $2X$ is smooth and transverse
at every physical root, but each root has mass two.
In general, section masses need not be local colengths of equations for
the support; see Proposition~\ref{prop:integer-equation-obstruction}.
\end{remark}

\begin{maintheorem}[Singular weighted sections]
\label{thm:main-singular}
Let $\mathfrak X$ be an effective Lee--Yang cycle of codimension $d$ in
$(\PP^1)^n$ for some $n\ge d$, and let $L\in\R^{n\times d}$ have positive maximal minors.
For every real phase its section is real and locally finite. The measures
$\nu^{\mathfrak X,L}_\theta$, possibly zero for an empty section, form a
vaguely continuous, uniformly translation-bounded family of FQ measures.
Their spectra are contained in $L^t\Z^n$ and have spectral total variation
$O((1+R)^n)$ in $B_R$. Conversely, every integer-mass FQ measure is such a section.
\end{maintheorem}

\begin{remark}\label{rem:unweighted-sections}
Counting distinct points once can fail to give an FQ measure.
Example~\ref{ex:two-circles} and, in Section~\ref{sec:singular},
Proposition~\ref{prop:singular-example} exhibit this failure of unweighted counting.
\end{remark}

The spectral coordinates also satisfy a linear bound.

\begin{maintheorem}[Linear spectral height]
\label{thm:main-height}
The spectrum of every integer-mass FQ measure generates a free abelian
group of finite rank $n$. For every basis matrix $L$ of this group there
is $C_L<\infty$ such that
\begin{equation}\label{eq:main-height}
 \|k\|\le C_L(1+\|L^tk\|)
 \quad\text{whenever }L^tk\in\Lambda',\quad k\in\Z^n.
\end{equation}
Consequently $\#(\Lambda'\cap B_R)=O((1+R)^n)$.
\end{maintheorem}

The relation to the questions of AKKV is as follows.
\begin{enumerate}
\item The unit-mass specialization of Theorem~\ref{thm:main-delone}
answers Question~(1); the strict forward construction is
\cite[Theorem~2.3]{AKKV}, and the inverse is proved in
Section~\ref{sec:classification}.
\item Theorem~\ref{thm:main-singular}, together with the unit-mass part of
Theorem~\ref{thm:main-equations},
answers both parts of Question~(2) in the natural generality of positive
integer masses.
\item The equations in Theorem~\ref{thm:main-equations}
answer Question~(4); the unit-mass classification is also stated in
Corollary~\ref{cor:geometric-classification}.
\end{enumerate}
Question~(3) asks for another construction if one of the first two
answers is negative, so it does not arise. Question~(5), concerning
arbitrary real-rooted trigonometric zero sets, is not addressed here.

\subsection{A section with collisions}
\begin{example}\label{ex:two-circles}
Let $X$ be the union of the projective closures of
$z_1z_2=1$ and $z_1^2z_2=1$ in $(\PP^1)^2$, and put
$L=(1,\sqrt2)^t$. Both components are strict Lee--Yang curves: on them
the logarithmic moduli satisfy $u_2=-u_1$ and $u_2=-2u_1$, respectively,
with the same limiting signs at the projective boundary. Both are
preserved by reciprocal conjugation and are smooth, being graphs of
the morphisms $t\mapsto t^{-1}$ and $t\mapsto t^{-2}$ on $\PP^1$.
Their torus loci are circles meeting only at $(1,1)$, with distinct
tangent lines, so $X$ is Lee--Yang but not strict.

For $\theta\in\R^2$, direct substitution gives
\begin{equation}\label{eq:two-circles-section}
 \Lambda_\theta=
 \frac{\Z-\theta_1-\theta_2}{1+\sqrt2}
 \ \cup\
 \frac{\Z-2\theta_1-\theta_2}{2+\sqrt2}.
\end{equation}
Write $a=(1+\sqrt2)^{-1}$ and $b=(2+\sqrt2)^{-1}$, so $a/b=\sqrt2$.
Whenever the two cosets are disjoint, the counting measure of their union is the sum
of two lattice combs and hence is an FQ measure by Poisson summation.
For example, at $\theta=(0,1/2)$ a common point would give
$\sqrt2(m-1/2)=n-1/2$ with $m,n\in\Z$, which is impossible.
Since $a/b$ is irrational, $a\Z+b\Z$ is dense in $\R$, so the
distance between the cosets has infimum zero. Their union is relatively
dense but not uniformly discrete, and hence is not Delone.

Along the physical line the defining equations have derivatives
$2\pi i(1+\sqrt2)$ and $2\pi i(2+\sqrt2)$ at their roots. Both are
nonzero, so each component contributes local intersection degree one. At $\theta=0$
the two cosets meet only at zero. The measure with its local intersection
degrees is
\[
 \nu_0=\sum_{j\in\Z}\delta_{aj}+\sum_{j\in\Z}\delta_{bj},
 \qquad
 \widehat\nu_0=a^{-1}\sum_{j\in\Z}\delta_{j/a}
              +b^{-1}\sum_{j\in\Z}\delta_{j/b}.
\]
This measure has mass two at zero and unit masses elsewhere. The
unweighted counting measure $\nu_0-\delta_0$ has Fourier transform
$\widehat\nu_0-d\xi$, with a nonzero absolutely continuous part, so it
is not an FQ measure. Each lattice comb depends vaguely continuously
on the phase, and $L\R$ is dense in $\T^2$ since $1$ and $\sqrt2$
are rationally independent.
Translates of the disjoint section at $\theta=(0,1/2)$ along phases
approaching zero therefore converge vaguely to $\nu_0$. Two unit atoms
have collided in the hull, the vague closure of the translates.
Proposition~\ref{prop:singular-example} gives a
two-dimensional example of the same obstruction to unweighted counting.
\end{example}

\subsection{Two general tools}
Two results apply independently of quasicrystals: the support-height
theorem and a smooth case of automatic algebraicity, stated as
Theorems~\ref{thm:main-support} and~\ref{thm:main-algebraicity}.
The proof of the former in Section~\ref{sec:holonomic-height} applies
Sabbah's localization and restriction results
\cite[Theorem~3.1 and Lemmas~3.2, 3.4]{Sabbah1992} to the cyclic module
itself, including the elements annihilated by nonzero polynomials.

\begin{maintheorem}[Support height]\label{thm:main-support}
Let $f:\Z^n\to\C$ generate a holonomic module for the algebra of
polynomial coordinates $(s_jf)(k)=k_jf(k)$ and invertible shifts
$(\tau_jf)(k)=f(k+e_j)$, identified with the
algebra of differential operators on $(\C^*)^n$ by
$\tau_j=t_j$ and $s_j=-t_j\partial_{t_j}$.
If $T:\R^n\to\R^d$ is linear and
$\{k:f(k)\ne0,\ \|Tk\|\le R\}$ is finite for every $R$, then there is
a constant $C<\infty$ such that $\|k\|\le C(1+\|Tk\|)$ on the
entire support of $f$.
\end{maintheorem}

Theorem~\ref{thm:main-support} applies without positivity or a lower
bound on the nonzero values of $f$, and without regular holonomicity.
Applied to the holonomic flux measure, it proves
Theorem~\ref{thm:main-height}, as shown in
Corollary~\ref{cor:universal-linear-height}.

Bell--Smertnig \cite[Theorem~1.3]{BellSmertnig} give finite unions of
simple linear sets as supports of D-finite series whose nonzero
coefficients lie in a finitely generated multiplicative group.
Garoufalidis \cite[Theorem~1.1]{GaroufalidisNewton} obtains eventually
quasi-linear Newton polytopes for nonzero recurrent Laurent-polynomial
sequences over $\Q$ with index-independent recurrence coefficients.
Chen \cite[slide~5]{ChenDegree} announces, in joint work with Bell, Wan,
Wang and Yin, a linear degree bound for holonomic polynomial sequences.
Theorem~\ref{thm:main-support} allows arbitrary complex coefficients on
the full lattice and arbitrary proper real projections of the support.

The second result is the smooth transverse case of
Theorem~\ref{thm:automatic-algebraicity}, with $\T^n$ identified
with $(S^1)^n$.

\begin{maintheorem}[Algebraicity from Fourier support]
\label{thm:main-algebraicity}
Let $T:\R^n\to\R^d$ have rank $d$, and put $c=n-d$ and
$V=(\ker T)^\perp$. Let $M\subset\T^n$ be a nonempty compact
embedded real-analytic $c$-manifold with $T_pM\oplus V=\R^n$ at
every point. Let $m=a\,d\operatorname{vol}_M$, where $a$ is continuous
and strictly positive and the volume is induced by the flat torus metric.
If the Fourier support of $m$ is $T$-proper, that is,
\[
 \#\{k\in\Z^n:\widehat m(k)\ne0,\ \|Tk\|\le R\}<\infty
 \qquad(R>0),
\]
then the Zariski closure of $M$ in $(\C^*)^n$ is pure of
complex dimension $c$.
\end{maintheorem}

\subsection{Strategy of the proof}
The construction begins with finite-band observations of the measure
(Section~\ref{sec:analytic-cycles}). Positivity bounds the local mass,
and integrality allows nearby atoms to be grouped into finite cycles.
Local finiteness of the spectrum makes each bandlimited observation a
finite exponential sum. A holomorphic local inverse recovers the
cluster moments from these observations, including through collisions.
The resulting finite-band rigidity implies that the spectral group
$\Gamma$ is finitely generated. Choosing its basis as the rows of $L$
parametrizes the hull by $\T^n$. The canonical set $M$ consists of
phases whose measure has an atom at the origin, and the flux measure
$m$ on $M$ relates the support $\Lambda=\supp\mu$ to the spectrum:
\begin{equation}\label{eq:intro-bridge}
 \Lambda=\{x\in\R^d:Lx\in M\},\qquad
 \widehat m(k)=
 \begin{cases}
 c_{L^tk},&L^tk\in\Lambda',\\
 0,&L^tk\notin\Lambda'.
 \end{cases}
\end{equation}
Pulling back the condition that a cycle contain the origin gives the
analytic germ of $M$ and its component multiplicities.

The next step is to prove algebraicity before using a spectral height
bound. In Section~\ref{sec:algebraicity}, Fourier vanishing gives an
upper bound for character spaces on $M$. If its Zariski closure had
dimension greater than $c=n-d$, logarithmic forms and their residues
would detect a transverse tropical cone and give a contradictory lower
bound. The algebraic description then makes the flux measure holonomic
(Section~\ref{sec:holonomic-flux}). Its Fourier support is proper under
the physical projection, so the support-height theorem in
Section~\ref{sec:holonomic-height} yields the linear spectral bound.

With this bound, the moments along nearby rational fibers are Laurent
polynomials. In Section~\ref{sec:non-delone-lh}, their Hermitian
matrices recover the complex branches, including exceptional fibers,
and their growth gives amoeba avoidance. Rational conditional weights
identify the local masses with constant integer coefficients on the
global irreducible components. The compression of
Section~\ref{sec:compression}, following \cite[\S8]{AKKV}, takes a
preimage under a finite torus isogeny to obtain a matrix with positive
maximal minors while preserving the sections and their weighted local
intersection degrees. This completes the inverse classification in
Section~\ref{sec:classification}.

The forward construction in Section~\ref{sec:singular} uses the
Lee--Yang inequality to force all roots to be real, and conservation of
local degree to pass through singular fibers. The canonical variety
also gives real-valued equations for the support.
Section~\ref{sec:q4-criterion} identifies its vanishing ideal when
arbitrary real frequencies are allowed and exhibits the obstruction
to realizing masses as local colengths.

Section~\ref{sec:singular} can be read independently of the inverse
construction, using the preliminaries, the compression proposition and
the cited results of AKKV. Section~\ref{sec:main-proofs} collects the
Delone and unit-mass specializations.

\section{Preliminaries}
\label{sec:preliminaries}

We fix the conventions for torus sections and holonomic sequences,
and recall the local degrees used to count singular intersections.
We distinguish spectral basis matrices from matrices with positive
maximal minors, which describe the resulting Lee--Yang sections.

\subsection{Tori, sections, and Fourier conventions}
We write $\T^n=\R^n/\Z^n$ and identify it with $(S^1)^n$ by
$\theta\mapsto\exp(2\pi i\theta)$. Exponentials, absolute values,
and logarithms of vectors are coordinatewise. For
$X\subset(\PP^1)^n$, put $X(\T)=X\cap(S^1)^n$. At the projective
boundary we use $\log0=-\infty$ and $\log\infty=+\infty$.
For $Y\subset(\C^*)^n$, its amoeba is
$\cA(Y)=\{\Log z:z\in Y\}$, where $\Log z=(\log|z_j|)_{j=1}^n$.
Unless specified otherwise, norms are Euclidean and $B_R$ is the
ball of radius $R$ centered at zero. If $W$ is a real linear subspace,
then $P_W$ is orthogonal projection onto $W$, and $W_\C$ is its
complexification. We put $c=n-d$.

A set is Delone if it is uniformly discrete and relatively dense:
distinct points have a common positive minimum separation, and every
ball of some fixed radius meets the set. A positive measure is
translation bounded if its mass on translates of the unit ball is
uniformly bounded. We use the translation convention
$\tau_y\delta_x=\delta_{x+y}$ and the FQ-set convention of
Definition~\ref{def:fourier-quasicrystal}.

\begin{remark}[Absolute convergence]\label{rem:weighted-schwartz}
The growth condition in Definition~\ref{def:fourier-quasicrystal}
ensures absolute convergence against Schwartz functions. For
$\phi\in\mathcal S(\R^d)$, if
$\mathcal V(R)=\sum_{|\xi|\le R}|c_\xi|\le C(1+R)^{N_0}$ and
$N_1>N_0$, summing over dyadic shells gives
\begin{equation}
 \sum_{\xi\in\Lambda'}|c_\xi|\,|\widehat\phi(\xi)|
 \le C_{\phi,N_1}\left(\mathcal V(1)+
 \sum_{j\ge0}2^{-jN_1}\mathcal V(2^{j+1})\right)<\infty.
 \label{eq:weighted-schwartz}
\end{equation}
The same bound holds with $c_\xi$ replaced by $c_\xi\omega_\xi$,
$|\omega_\xi|=1$. The physical half of the growth bound similarly implies
$\sum_x|a_x||\widehat f(x)|<\infty$ for every
$f\in\mathcal S(\R^d)$. For test functions with compact Fourier support,
the spectral sums are finite by local finiteness of the spectrum.
\end{remark}

\begin{definition}\label{def:physical-section}
For a real matrix $L$ of rank $d$, write
$E_L(x)=\exp(2\pi iLx)$ and call $V=L\R^d$ its physical plane.
A point of the section
$\{x:\exp(2\pi i\theta)E_L(x)\in X\}$ is a physical root at phase
$\theta$. A real $c$-dimensional submanifold $S\subset\R^n$ is
$V$-transverse at $p$ if $T_pS\oplus V=\R^n$.
For a linear map $T:\R^n\to\R^d$, a set $E\subset\Z^n$ is
$T$-proper if $\{k\in E:\|Tk\|\le R\}$ is finite for every $R$.
\end{definition}

\subsection{Sign variation and Lee--Yang varieties}
For $v\in(\R\cup\{\pm\infty\})^n$, let $\varn(v)$ be the number
of sign changes after removing zero entries, and let $\vbar(v)$ be the
largest number obtained by assigning signs to the zero entries.
For example, $\varn(1,0,1)=0$ whereas $\vbar(1,0,1)=2$.
The positive and nonnegative Grassmannians are denoted by $\Gr_+(d,n)$
and $\Gr_{\ge0}(d,n)$: their elements admit a basis matrix with positive,
respectively nonnegative, maximal minors.

\begin{definition}[Lee--Yang varieties]\label{def:lee-yang}
Following \cite[Definitions~2.2 and~4.1]{AKKV}, a reduced equidimensional
closed variety $X\subset(\PP^1)^n$ of codimension $d$ is Lee--Yang if it is invariant under reciprocal
conjugation $z\mapsto\overline z^{-1}$ and
$\vbar(\log|z|)\ge d$ for every $z\in X\setminus\T^n$.
It is strict Lee--Yang if $X$ is smooth at every point of $X(\T)$ and
the stronger inequality $\varn(\log|z|)\ge d$ holds off the torus.
\end{definition}

By the sign-variation criterion \cite[Lemma~3.2]{AKKV}, a
nonzero $v\in\R^n$ belongs to some $W\in\Gr_+(d,n)$ exactly when
$\vbar(v)<d$, and to some $W\in\Gr_{\ge0}(d,n)$ exactly when
$\varn(v)<d$. On $(\C^*)^n$, the Lee--Yang inequality therefore expresses
amoeba avoidance for all positive $d$-planes; the stronger inequality
in Definition~\ref{def:lee-yang} expresses avoidance for all nonnegative
$d$-planes. Both inequalities also hold at the projective boundary.
Sections by positive matrices are real by the calculation preceding
Theorem~\ref{thm:akkv-construction}.

\subsection{Local intersection degrees}
The weights in a singular section count sheets of finite analytic maps.
We recall conservation of degree before comparing these weights with
lengths of fibers.

\begin{lemma}[Local degree of a finite analytic map]
\label{lem:local-degree-general}
Let $A$ be a reduced pure $m$-dimensional analytic space, and let
$\pi:A\to B$ be a proper holomorphic map with finite fibers, where
$B$ is a connected complex manifold of dimension $m$.
\begin{enumerate}
\item[(i)] There is a nowhere dense analytic subset $D\subset B$ such
that $\pi$ is an unramified covering of constant degree $\deg\pi$ over
$B\setminus D$.
\item[(ii)] Every irreducible component of $A$ maps onto $B$.
\item[(iii)] Let $a\in A$, and let $U_a$ be a relatively compact open
neighborhood of $a$ with
$\pi^{-1}(\pi(a))\cap\overline{U_a}=\{a\}$.
There is a connected neighborhood $W$ of $\pi(a)$ such that
$U_a\cap\pi^{-1}(W)\to W$ is proper. Its degree, denoted by
$\deg_a\pi$, is independent of these choices.
\item[(iv)] For every $w\in B$,
\[
 \sum_{a\in\pi^{-1}(w)}\deg_a\pi=\deg\pi.
\]
\end{enumerate}
\end{lemma}
\begin{proof}
By the proper mapping theorem \cite[II, Theorem~8.8]{Demailly}, the
image of each irreducible component is analytic. Its dimension is $m$
because the map is finite, so it equals the connected manifold $B$,
proving (ii). Components are locally finite
\cite[II, Theorem~5.3]{Demailly}; each meets a fixed finite fiber,
and hence there are only finitely many.

The images of the singular and critical loci form a nowhere dense
analytic subset $D$; outside it the map is an unramified covering
\cite[II, Theorem~4.19]{Demailly}. The complement is connected by
\cite[II, Remark~4.2]{Demailly}, so the number of sheets is constant.
This proves (i), with degree interpreted as in
\cite[III, Example~9.2]{Demailly}.

Put $w_0=\pi(a)$. Since the compact set $\partial U_a$ misses the fiber
over $w_0$, we can choose a connected neighborhood $W$ of $w_0$
disjoint from $\pi(\partial U_a)$. The inverse image of a compact
subset of $W$ in $U_a$ is closed in the compact set $\overline{U_a}$;
thus the restricted map is proper, and its generic sheet count defines
$\deg_a\pi$. To compare two choices, take an isolating neighborhood
contained in both and shrink the base. Both localized fibers then lie
in the smaller neighborhood, since otherwise compactness would produce
another point in the central fiber. Their generic sheet counts agree,
proving (iii).

For (iv), choose disjoint isolating neighborhoods of the points of
$\pi^{-1}(w)$. By properness, after shrinking the base its full inverse
image lies in their union. At a common regular base point, the sheets
in each neighborhood are counted by its local degree. Their total is
$\deg\pi$, proving the identity at $w$.
\end{proof}

\begin{definition}\label{def:local-intersection-degree}
Let $A\subset\C^n$ be a reduced pure $c$-dimensional analytic germ at
$y$, where $c+d=n$, and let $\pi:(A,y)\to(\C^c,0)$ be finite.
Its local degree $\deg_y\pi$ is defined by
Lemma~\ref{lem:local-degree-general}(iii): it counts the points of a
nearby regular fiber in an isolating neighborhood of $y$, over all
local irreducible components.
For an isolated section of $A$ by a complex $d$-plane in $\C^n$, its
local intersection degree is the local degree of projection along that
plane. To obtain a finite proper local representative, choose an
isolating neighborhood whose boundary misses the central fiber and
shrink the base to avoid its boundary image. The fibers are then compact
analytic subsets of a domain in $\C^n$, hence finite. The degree is
independent of the linear coordinates on the quotient.
For a subvariety of $(\PP^1)^n$, we compute the section degree in
logarithmic coordinates $z=\exp(2\pi iw)$, where
$\exp(2\pi i(\theta+Lx))$ becomes the affine plane $\theta+L\C^d$.
\end{definition}

\begin{definition}[Effective Lee--Yang cycles]
\label{def:effective-lee-yang-cycle}
Let $1\le d\le n$ and $c=n-d$. An effective Lee--Yang cycle of
codimension $d$ is a finite nonempty formal sum
\[
 \mathfrak X=\sum_i\varkappa_iX_i,
 \qquad \varkappa_i\in\Z_{>0},
\]
where the $X_i\subset(\PP^1)^n$ are distinct irreducible reduced closed
subvarieties of codimension $d$, the reduced support
$|\mathfrak X|=\bigcup_iX_i$ is Lee--Yang in
Definition~\ref{def:lee-yang}, and reciprocal conjugation preserves the
cycle: exchanged components have equal coefficients. Its support is
\emph{strict} if $|\mathfrak X|$ is strict Lee--Yang.

For a real phase $\theta$ and a positive-minor matrix $L$, define
\[
 \deg^{\mathfrak X}_\theta(x)
   =\sum_i\varkappa_i\deg^{X_i}_\theta(x),
 \qquad
 \nu^{\mathfrak X}_\theta
   =\sum_{\substack{x\in\C^d:\,\exp(2\pi i(\theta+Lx))\in|\mathfrak X|}}
          \deg^{\mathfrak X}_\theta(x)\delta_x.
\]
When the matrix has to be indicated we write $\nu^{\mathfrak X,L}_\theta$.
\end{definition}

\begin{remark}\label{rem:effective-lee-yang-sections}
The realness calculation in Section~\ref{subsec:intro-lee-yang} uses only
$\vbar(\log|z|)\ge d$ and thus applies to each component $X_i$ and each
local branch. Since a complex analytic subset of $\C^d$ contained in
$\R^d$ is discrete, every section root is isolated. Projection along
the section plane has a finite proper local representative by
Definition~\ref{def:local-intersection-degree}, with degree defined by
Lemma~\ref{lem:local-degree-general}.
Here $\deg^{X_i}_\theta(x)$ is the local intersection degree
of Definition~\ref{def:local-intersection-degree}, and is zero when $X_i$
does not contain the image of $x$.
These degrees can differ from lengths of ideals
(Remark~\ref{rem:degree-length}). We also use this weighted notation for
finite local sections of effective analytic or algebraic cycles in the
canonical torus before compression.

A component is called \emph{active} when $X_i(\T)\ne\varnothing$.
Deleting inactive components leaves every real-phase section measure
unchanged; if all components are inactive, these sections are empty.
By Lemma~\ref{lem:integer-active-components}, each active component is
preserved by reciprocal conjugation and has a Zariski-dense torus locus
of real dimension $c$. Thus every nonempty torus locus has dimension $c$.
\end{remark}

\begin{remark}[Degree and length]\label{rem:degree-length}
Put $B=\mathcal O_{A,y}$ and $R=\C\{w_1,\ldots,w_c\}$, the local
ring of the smooth base. The finite projection makes $B$ a finite
$R$-module of generic rank $\deg_y\pi$. Because $B$ is reduced and
pure-dimensional, with every local component dominating the base,
$B$ is torsion-free over $R$. Put $\mathfrak q=(w_1,\ldots,w_c)B$.
Both residue fields are $\C$, so the lengths of $B/\mathfrak q^N$
over $R$ and $B$ agree. The multiplicity formula
\cite[\S14]{Matsumura} and $e_{(w)}(R)=1$ give
\[
 \deg_y\pi=e_{\mathfrak q}(B)
       \le \operatorname{length}_B(B/\mathfrak q).
\]
The inequality follows from \cite[\S14]{Matsumura}, with equality
precisely when $B$ is Cohen--Macaulay \cite[\S17]{Matsumura}.
Writing $\mathfrak m_R=(w_1,\ldots,w_c)$, the number
$\operatorname{length}_B(B/\mathfrak q)
=\dim_\C(B/\mathfrak m_RB)$ is the minimum number of generators of
$B$ as an $R$-module, so it is at least its rank $r=\deg_y\pi$.
If equality holds, the kernel of a surjection $R^r\to B$ has rank zero.
As a submodule of a free module it must vanish, so $B$ is free.
Over the regular local ring $R$, the finite module $B$ is free if and
only if $B$ is Cohen--Macaulay. In either case, $B/\mathfrak q$ is Artinian.
\end{remark}

The next lemma identifies a reduced germ of local degree one with its
smooth base. We use it for the unit-mass specialization in
Section~\ref{sec:main-proofs}.

\begin{lemma}[Local degree one]\label{lem:local-degree-one}
Let $A$ be a reduced analytic germ of pure dimension $c$ at $y$, and let
$\pi$ map $(A,y)$ finitely to a smooth $c$-dimensional germ.
If $\pi|_A$ has local degree one, then it is a local
biholomorphism; in particular $A$ is smooth at $y$.
\end{lemma}

\begin{proof}
Each local irreducible component has dimension $c$ and maps finitely
onto a base neighborhood, contributing at least one to the degree.
There is therefore just one component. Let $\mathcal O_0$ be the local
ring of the smooth base and $\mathcal O_1$ that of $A$. Since degree
one identifies their fraction fields,
$\mathcal O_0\subset\mathcal O_1\subset\Frac(\mathcal O_0)$, with
$\mathcal O_1$ integral over $\mathcal O_0$. The base is normal, so
$\mathcal O_1=\mathcal O_0$.
\end{proof}

\subsection{Holonomic distributions and sequences}
A distribution on a smooth real algebraic manifold is algebraically
holonomic if the cyclic module it generates over the differential
operators of its complexification is holonomic. Here holonomic means
that the characteristic variety of a nonzero module has dimension equal
to that of the underlying smooth variety \cite[\S3]{GSSauthor}.
The zero module is holonomic by convention. We use this definition for
the flux measure in Section~\ref{sec:holonomic-flux}.

For sequences, the algebra $\mathscr D_n$ generated by coordinate multiplication
$s_j$ and invertible shifts $\tau_j$ acts by
$(s_jf)(k)=k_jf(k)$ and $(\tau_jf)(k)=f(k+e_j)$, with
$\tau_i s_j=(s_j+\delta_{ij})\tau_i$.
The sequence shifts $\tau_j$ are distinct from the translations
$\tau_y$ of measures.
The identification $\tau_j=t_j$, $s_j=-t_j\partial_{t_j}$ with differential
operators on $(\C^*)^n$ defines holonomicity of the cyclic sequence
module. This module retains its polynomial torsion, the elements
annihilated by nonzero polynomials, in the support-height theorem of
Section~\ref{sec:holonomic-height}.

\subsection{Notation}
Before compression we use a spectral basis matrix; afterwards we use
the resulting matrix with positive maximal minors. Our notation is
summarized below.
\begin{center}
\begin{tabular}{@{}p{0.25\textwidth}p{0.69\textwidth}@{}}
$\Gamma$, $L$ & Spectral group and its basis matrix.\\
$\widetilde L$ & Matrix with positive maximal minors obtained by compression.\\
$S$ & Nonsingular integer matrix with $\widetilde L=S^{-1}L$.\\
$E_L$ & Exponential map $E_L(x)=\exp(2\pi iLx)$.\\
$V$, $c$ & Physical plane $L\R^d$ and dimension $n-d$ of the canonical set.\\
$\mu_t$ & Hull measure at phase $t$.\\
$M$, $Y$ & Canonical set and its Zariski closure in the algebraic torus.\\
$m$ & Canonical flux measure on $M$.\\
$M^\C$, $\mathfrak M^\C$ & Reduced canonical analytic germ and its effective analytic cycle.\\
$\mathfrak Y$, $\varkappa_i$ & Effective cycle on $Y$ and its positive integer component coefficients.\\
$\mathfrak X$ & Effective Lee--Yang cycle; in Sections~\ref{sec:compression} and~\ref{sec:classification}, the one obtained by compression.\\
$\nu^{\mathfrak X}_\theta$, $\deg^{\mathfrak X}_\theta$ & Section measure and weighted local intersection degree at phase $\theta$.
\end{tabular}
\end{center}

\section{Finite toral type and analytic zero-cycles}
\label{toral-type-and-analytic-zero-cycles-without-uniform-discreteness}

\label{sec:analytic-cycles}

We recover a finite-dimensional torus from bandlimited observations
and use cluster moments to describe its sections through collisions.
A local inverse proves analytic dependence under a temporary
finite-generation hypothesis and then yields finite generation itself.
The resulting analytic set and its flux measure supply the hypotheses
for automatic algebraicity in the next section.

\paragraph{Standing notation.}
In Sections~\ref{sec:analytic-cycles}--\ref{sec:classification}, $L$
is the matrix whose rows form a $\Z$-basis of the spectral group
$\Gamma=\langle\Lambda'\rangle_\Z$, once its finite generation has
been established. It has rank $d$, with
\[
 T=L^t,\qquad V=L\R^d,\qquad \ker T=V^\perp;
\]
$T$ is injective on $\Z^n$. The hull is the vague closure of translates of the integer-mass FQ
measure $\mu$, parametrized on the torus by $\mu_t$, with $\mu_{Lx}=\tau_{-x}\mu$. The canonical set,
its Zariski closure, and its flux measure are denoted by
\[
 M=\{t\in\T^n:0\in\supp\mu_t\},\qquad
 Y=\overline M^{\mathrm{Zar}}\subset(\C^*)^n,\qquad m,
\]
respectively. The definition of $m$ in \eqref{eq:HF-4.1} uses the Fourier
convention of
\eqref{eq:intro-fourier-identity}. In Theorems~\ref{thm:main-delone}--\ref{thm:main-singular}
and Section~\ref{sec:singular}, $L$ denotes an arbitrary matrix with
positive maximal minors. Compression of a spectral basis matrix yields $\widetilde L=S^{-1}L$, where $S$ is
the nonsingular integer matrix of Proposition~\ref{prop:compression}.

\begin{theorem}[Finite toral type and analytic zero-cycles]
\label{thm:all-fq-toral-cycles}
Let $\mu$ be an integer-mass FQ measure on $\R^d$, with spectral
group $\Gamma=\langle\Lambda'\rangle_{\Z}$. Then this group is finitely
generated. Choose a basis $\gamma_1,\ldots,\gamma_n$, and let
$L\in\R^{n\times d}$ have these basis vectors as its rows. The
following statements hold.
\begin{enumerate}
\item[(a)] The hull is faithfully parametrized by $\mu_t$, $t\in\T^n$,
with positive integer masses, a common local mass bound, and
$\mu_{t+Lx}=\tau_{-x}\mu_t$.
\item[(b)] Locally in phase, all atoms can be grouped into uniformly
separated clusters of bounded total multiplicity. Their unordered
zero-cycles form holomorphic families in complex phase neighborhoods,
including collisions and irreducible branching. The simultaneous
multisymmetric moments form an $\ell^\infty$-valued holomorphic map,
each block taking values in the corresponding symmetric-product
variety $V_{s_b}$, where $s_b$ is the total multiplicity of cluster $b$.
All cluster counts include multiplicities.
\item[(c)] The canonical set $M=\{t:0\in\supp\mu_t\}$ is compact and
real analytic. Its reduced complex germ $M^\C$ is pure of codimension
$d$. There is, in addition, a well-defined effective analytic cycle
$\mathfrak M^\C$ supported on this germ. Its coefficients are positive
integers on local irreducible components and agree on overlaps.
\item[(d)] At a phase where the atom at zero has mass one, $M^\C$ is
smooth and transverse to $L\C^d$.
\item[(e)] For each real phase $t$ there is a neighborhood $U_t$ of $0$
in $\C^d$ such that
\[
 \{x\in U_t:t+Lx\in M^\C\}
       =\supp\mu_t\cap U_t\subset\R^d.
\]
If $t\in M$, then $\deg^{\mathfrak M^\C}_t(0)=\mu_t(\{0\})$;
by covariance, for every $x\in\supp\mu_t$,
\[
 \deg^{\mathfrak M^\C}_t(x)=\mu_t(\{x\}).
\]
\end{enumerate}
\end{theorem}

\begin{remark}\label{rem:analytic-cycle-degree}
The reduced degree of $M^\C$ may differ from the mass. For example,
when $d=1$ and $\mu=2\sum_{k\in\Z}\delta_k$, the canonical set $M$
is one point and $\mathfrak M^\C$ is that point with coefficient two.
\end{remark}

Finite generation is assumed through
Section~\ref{the-canonical-analytic-toral-set} and proved in
Section~\ref{local-finite-band-rigidity}. The mass bound in
Lemma~\ref{lem:ac-translation-bound} holds without this assumption.
In Lemma~\ref{lem:ac-persistence}, phase dependence enters through
\eqref{eq:AC-1.3}; Lemma~\ref{lem:ac-rigidity} replaces this estimate
by uniform weak closeness. The cluster, quotient and kernel
constructions apply to a single positive integer atomic measure with
a uniform local mass bound.

In this section, \(t\) denotes a real phase and \(\theta\) its complex extension; later \(\theta_\ell\) will denote logarithmic lifts of sheets.
A cluster is a finite effective zero-cycle, with coincident points
counted according to their integer masses. We describe it by
multisymmetric moments.

With the chosen spectral basis, write
\[
 \widehat\mu=\sum_{\xi\in\Lambda'}c_\xi\delta_\xi,\qquad
 \xi=L^tk(\xi),\quad k(\xi)\in\mathbb Z^n.
\]

\subsection{The hull and its uniform topology}\label{the-toral-family-and-its-uniform-topology}

The coordinate \(k(\xi)\) is unique, since the rows of \(L\) form a
basis of the spectral group.

\subsubsection{Translation boundedness}\label{translation-boundedness}

\begin{lemma}[Mass bounds and relative density]\label{lem:ac-translation-bound}
There is an integer \(K\) such that every vague limit of translates of
\(\mu\) is a positive integer atomic measure with mass at most \(K\)
in every unit ball. Every measure in the hull is nonzero and has
relatively dense support, with a common relative-density radius.
\end{lemma}
\begin{proof}
Choose a nonnegative, nonzero smooth function \(\chi_0\) with compact
support sufficiently close to zero that \(\phi=|\check\chi_0|^2\)
has no nonzero point of \(\Lambda'\) in its Fourier support. Then
\(\phi\) is a nonnegative Schwartz function, and the positive integral
of \(\chi_0\) makes it positive at zero and hence near zero.
Fourier transformation gives
\[
 \mu*\phi=C.
\]
Evaluating the convolution at an atom shows that the constant is positive. If \(\phi\ge a_0>0\) on a small ball \(B_r\), positivity gives
\[
 a_0\mu(y+B_r)\le C\qquad(y\in\mathbb R^d).
\]
This bounds the total mass. Finite coverings extend the bound uniformly
to every fixed-radius ball.

The translates of \(\mu\) are vaguely relatively compact. On each
compact ball, extract the finitely many particle positions in a slightly
larger ball, repeated according to mass. Limits then have nonnegative
integer masses and are positive integer atomic measures. Choose an
integer \(K\) bounding every limit by \(K\) on unit balls, using
the original bound on balls of radius two.

Translation boundedness and Schwartz decay imply, as \(R\to\infty\),
\[
 \sup_{\nu\text{ in the hull}}\sup_{y\in\mathbb R^d}
 \int_{|x-y|\ge R}\phi(y-x)\,d\nu(x)\longrightarrow0.
\]
To obtain this estimate, cover by unit cubes and sum their Schwartz
suprema, using the uniform translated mass bounds. The same estimate
allows passage to the vague limit in
\(\nu*\phi=C\). Choose \(R\) so that the displayed tail is less than \(C/2\).
If \(\nu\) had no support point in \(y+B_R\), its convolution at \(y\) would
be less than \(C/2\), a contradiction. This proves both nonvanishing and
uniform relative density.
\end{proof}

\subsubsection{Construction from finite-band observations}\label{construction-from-finite-band-observations}

The flow \(x\mapsto Lx\bmod\mathbb Z^n\), which we call the physical flow, is dense in \(\mathbb T^n\): an integer character vanishing on its image would give a nontrivial integer relation among the basis vectors \(\gamma_j\).

For a Schwartz function \(h\) with compactly supported Fourier transform, define the trigonometric polynomial
\begin{equation}
 F_h(t)=
  \sum_{\xi\in\Lambda'}
   c_\xi\,\widehat h(-\xi)\,e^{2\pi ik(\xi)\cdot t}.               \label{eq:AC-1.1}
\end{equation}
The sum is finite, and covariance at a phase in the physical flow gives
\[
 F_h(Lx)=\int h(y)\,d(\tau_{-x}\mu)(y).
\]
Given \(Lx_j\to t\), choose a vaguely convergent subsequence of
\(\tau_{-x_j}\mu\). Testing its limit against any bandlimited Schwartz
function gives \(F_h(t)\), for every such \(h\).

To prove uniqueness, let \(g\in C_c^\infty\) and use smooth frequency
truncation to obtain bandlimited Schwartz functions \(g_N\to g\) in
the Schwartz topology. The common unit-ball mass bound gives
\[
 \left|\int(g_N-g)\,d\nu\right|
 \le C_K\sup_y(1+|y|)^{d+1}|g_N(y)-g(y)|
\]
for every translated hull measure \(\nu\). Agreement on bandlimited tests therefore extends to \(C_c^\infty\).

The limit is independent of the approximating sequence.
\begin{definition}\label{def:ac-hull-measures}
The limit \(\mu_t\) is called the hull measure at phase \(t\).
\end{definition}
Conversely, if \(\tau_{-x_j}\mu\to\nu\) vaguely, compactness of the
torus gives a subsequence with \(Lx_j\to t\). The finite-band
observations then identify \(\nu=\mu_t\). Thus the family
\((\mu_t)_{t\in\mathbb T^n}\) is exactly the hull of \(\mu\).
The construction gives
\begin{equation}\label{eq:ac-fourier-observation}
 \int h\,d\mu_t=F_h(t)
\end{equation}
for every bandlimited Schwartz function \(h\). The same approximation
argument proves vague continuity and covariance
\begin{equation}
 \mu_{t+Lx}=\tau_{-x}\mu_t.                                    \label{eq:AC-1.2}
\end{equation}
Each \(\mu_t\) has positive integer masses and the common local mass
bound. Equality of \(\mu_t,\mu_{t'}\) implies equality of their nonzero
spectral characters. Since these generate \(\mathbb Z^n\), the
parametrization is injective.

For any fixed compactly supported smooth \(g\), put
\[
 F_g(t)=\int g\,d\mu_t.
\]
The approximation argument proves continuity on the compact torus.
Combining this with covariance gives the uniform estimate
\begin{equation}
 \sup_x\left|
  \int g(y-x)\,d(\mu_t-\mu_{t'})(y)\right|
  =\sup_x|F_g(t+Lx)-F_g(t'+Lx)|
  \longrightarrow0\quad(t\to t').                             \label{eq:AC-1.3}
\end{equation}
The estimate \eqref{eq:AC-1.3} also holds uniformly for a compact family
of test functions modulo translation: approximate the family uniformly
on a common compact support by finitely many tests, then use the common
local mass bound.

Finally \(L\) has rank \(d\). Otherwise a nonzero vector \(v\) with \(Lv=0\), together with \eqref{eq:AC-1.2}, would make a nonzero discrete atomic measure invariant under all translations \(av\), which is impossible.

\subsection{Uniform clusters and continuity of their moments}\label{uniform-clusters-and-continuity-of-their-moments}

We group nearby atoms into clusters whose masses persist under small
phase changes. The local mass bound controls their diameters and the
number of moments needed for the holomorphic construction.

Fix a reference phase and translate the torus coordinate so that \(t_0=0\); write \(\nu=\mu_0\). Choose
\begin{equation}
 0<\delta<\frac1{10K}.
 \label{eq:AC-delta-choice}
\end{equation}
Join distinct support points at distance less than \(\delta\), and let \(A_b\) be the connected components, each carrying the masses of its points.

Each component has at most \(K\) distinct locations. Otherwise, a
connected subset of \(K+1\) locations could be built by adding one
adjacent point at a time. All would lie within \(K\delta<1\) of the
first, contradicting the local mass bound. The component has diameter at most \(K\delta\) and lies in a unit
ball, so its mass
\[
 s_b=\nu(A_b)
\]
also satisfies \(1\le s_b\le K\).

Choose a base point \(y_b\in A_b\). Then
\begin{equation}
 R=K\delta,\qquad
 \operatorname{diam}A_b\le R,\qquad
 \operatorname{dist}(A_b,A_{b'})\ge\delta\quad(b\ne b').             \label{eq:AC-2.1}
\end{equation}
The base points \(y_b\) are therefore \(\delta\)-separated and form a countable set.

Fix a small number
\begin{equation}
 0<\eta<\delta/100,\qquad \sigma=\eta/4.                           \label{eq:AC-2.2}
\end{equation}

\begin{lemma}[Uniform cluster persistence]
\label{lem:ac-persistence} For \(t\) sufficiently close to zero, assign
each atom of \(\mu_t\) to the unique reference component within
\(\eta/4\) of it. Every atom has such a component, the total mass
assigned to \(A_b\) is \(s_b\), and the centered power sums
\begin{equation}
 p_{b,\alpha}(t)=
   \sum_{\text{atoms assigned to }A_b}(x-y_b)^\alpha,
 \qquad 1\le|\alpha|\le s_b,                                  \label{eq:AC-2.3}
\end{equation}
are continuous at zero in the joint \(\ell^\infty\) norm, with all
sums counting multiplicities. We write
\(p_b^0=(p_{b,\alpha}(0))_{1\le|\alpha|\le s_b}\) for the centered
reference moment vector.

\end{lemma}

\begin{proof}
Choose a nonnegative bump \(\zeta\), supported in \(B(0,\eta/4)\), with
\(\zeta(0)=1\). By \eqref{eq:AC-1.3}, its convolutions with
\(\mu_t-\nu\) have supremum norm less than \(1/2\) when \(t\) is close
to zero. An atom of \(\mu_t\) at distance at least \(\eta/4\) from
\(\operatorname{supp}\nu\) would contribute at least one at its location,
where the reference convolution vanishes. This contradiction places
every atom in the required neighborhood, and the separation in
\eqref{eq:AC-2.1} makes its assignment unique.

For each component choose a smooth cutoff \(\chi_b\), equal to one on its \(\eta/2\)-neighborhood and zero outside its \(\eta\)-neighborhood. To choose these uniformly, fix \(\psi\) equal to one on \(B(0,1/2)\) and zero outside
\(B(0,1)\), and use \(\psi((x-a)/\eta)\) at each atom \(a\).
Set \(\chi_b\) equal to one minus the product of one minus these
rescaled bumps.

After translation by \(y_b\), the family lies in a compact subset of
\(C_c^\infty(B(0,R+\eta))\), being a finite union of continuous
images of compact spaces parametrizing at most \(K\) positions.
Multiplication by centered monomials of degree at most \(K\) preserves
compactness. Thus \eqref{eq:AC-1.3} and finite approximation of this compact family imply
\[
 \sup_b\left|\int\chi_b\,d(\mu_t-\nu)\right|\to0
\]
and the same limit for the moment tests. For \(t\) close enough, the
first integrals are integers of modulus less than \(1/2\), hence zero;
each cluster count is \(s_b\). The moment limits give the claimed
\(\ell^\infty\) continuity. \end{proof}

The moments therefore provide continuous coordinates through
collisions. Their holomorphic extension will be obtained from the
finite symmetric product.

\subsection{Finite zero-cycle varieties and uniform local extension}\label{finite-zero-cycle-varieties-and-uniform-local-extension}

The symmetric product supplies algebraic coordinates for each cluster.
To apply a local inverse in a Banach space, we extend the observation
functions to ambient moment coordinates with uniform bounds.

\subsubsection{A concrete affine quotient}\label{a-concrete-affine-quotient}

For fixed \(s\le K\), let
\[
 \mathcal A_s=\{\alpha\in\mathbb Z_{\ge0}^d:1\le|\alpha|\le s\},
 \qquad N_s=|\mathcal A_s|.
\]
The map
\begin{equation}
 \pi_s:(\mathbb C^d)^s\longrightarrow\mathbb C^{N_s},
 \qquad
 (x_1,\ldots,x_s)\longmapsto
 \left(\sum_{j=1}^s x_j^\alpha\right)_{\alpha\in\mathcal A_s}     \label{eq:AC-3.1}
\end{equation}
identifies unordered effective \(s\)-tuples with the affine quotient
\[
 V_s=\operatorname{Sym}^s(\mathbb C^d).
\]

The power sums in \eqref{eq:AC-3.1} generate the invariant ring.
Indeed, invariant polynomials are spanned by orbit sums of products of
monomials assigned to distinct point indices. Inclusion-exclusion on set partitions expresses these sums in power
sums \(p_\alpha\) of initially arbitrary degree. For \(s+1\) nonzero multiindices \(\beta_1,\ldots,\beta_{s+1}\), the orbit sum using \(s+1\) distinct indices is zero. The partition formula isolates
\[
 (-1)^s s!\,p_{\beta_1+\cdots+\beta_{s+1}}
\]
plus products of at least two lower-degree power sums. Split a multiindex
of degree greater than \(s\) into \(s+1\) nonzero parts; induction
reduces every generator to degree at most \(s\).

Each coordinate \(x_{j,\ell}\) is integral over this invariant ring: it satisfies the monic polynomial
\[
 \prod_{i=1}^s(\mathsf{x}-x_{i,\ell}),
\]
whose coefficients, by Newton's identities, are polynomials in
\(p_{ke_\ell}\), \(1\le k\le s\); hence \(\pi_s\) is finite and its
image \(V_s\) is closed. Interpolating on two distinct finite permutation
orbits and averaging over the group produces an invariant separating
them, so the fibers are precisely the unordered tuples.

To see analytic properness, bounded quotient coordinates bound the
monic coefficients and therefore every \(x_{j,\ell}\). Inverse images
of compact sets are thus closed and bounded. It follows that convergence
of moment vectors is equivalent, locally and uniformly on compact sets,
to convergence of unordered tuples in the matching metric
\begin{equation}
 d_{\mathrm{cyc}}(\{x_j\},\{y_j\})
   =\min_{\pi\in S_s}\max_j|x_j-y_{\pi(j)}|.               \label{eq:AC-3.2}
\end{equation}

The finite-group quotient is normal. In fact, invariants of a finite
group acting on a normal domain form a normal ring: an invariant
fraction integral over the invariant ring is integral over the full
polynomial ring, so it lies in that ring and remains invariant. Holomorphic symmetric functions on \((\mathbb C^d)^s\) descend
holomorphically to the analytic quotient
\cite[\S3, Theorem~2 and Proposition~4; \S4, Theorem~4]{CartanQuotient}. On the unramified locus this is the quotient property; normality
extends the locally bounded function across the branch locus.

More generally, let \(\Omega\subset\mathbb C^d\) be open and
\(h\in\mathcal O(\Omega)\). On the open subset
\(\operatorname{Sym}^s(\Omega)\subset V_s\), the trace
\begin{equation}
 \mathcal T_{s,h}(p)=\sum_{j=1}^s h(x_j),\qquad p=\pi_s(x_1,\ldots,x_s), \label{eq:AC-3.3}
\end{equation}
is holomorphic: the symmetric holomorphic function on \(\Omega^s\)
descends by the same quotient argument.

\begin{lemma}[Splitting finite cycles]\label{lem:ac-splitting}
Let \(\Theta\) be a complex domain, let
\(p:\Theta\to\operatorname{Sym}^s(\mathbb C^d)\) be holomorphic,
let \(\theta_0\in\Theta\), and let \(U\subset\mathbb C^d\) be
bounded and open. If
\(\operatorname{supp}p(\theta_0)\cap\partial U=\varnothing\),
then, near \(\theta_0\), the part of \(p(\theta)\) in \(U\)
is a holomorphic family in \(\operatorname{Sym}^{s'}(\mathbb C^d)\),
where \(s'\) is its degree at \(\theta_0\).
\end{lemma}
\begin{proof}
Choose disjoint small neighborhoods of the support points of
\(p(\theta_0)\), avoiding \(\partial U\), and denote their union
by \(\Omega\). Continuity in the matching metric makes
\(\operatorname{supp}p(\theta)\) stay in \(\Omega\) after shrinking
\(\Theta\) around \(\theta_0\); the degree in \(U\) is then
constant. Each function \(x^\alpha\mathbf1_U\) is holomorphic on
\(\Omega\). Its trace in \eqref{eq:AC-3.3} is the corresponding
power sum of the restricted cycle. These power sums are its holomorphic
coordinates in the closed moment variety
\(V_{s'}\).
For \(s'=0\) the restricted cycle is the constant empty cycle.
\end{proof}

\subsubsection{The local extension bound}\label{the-local-extension-bound}

\begin{lemma}[Bounded extension on a smaller compact set]
\label{lem:ac-extension} Let \(W\subset\mathbb C^N\) be a reduced closed analytic subset,
let \(U\subset\mathbb C^N\) be Stein and open, and let
\(K_0\Subset U\) be compact. There are a compact \(K_1\subset W\cap U\) and a constant \(C\) with the following property: whenever \(g\in\mathcal O(W\cap U)\) and \(b_0>0\) satisfy \(\sup_{K_1}|g|\le b_0\), there exists \(\widetilde g\in\mathcal O(U)\), restricting to \(g\), with
\begin{equation}
 \sup_{K_0}|\widetilde g|\le Cb_0.                                         \label{eq:AC-3.4}
\end{equation}

\end{lemma}

\begin{proof}
By Cartan's theorem B \cite[IX, Corollary~4.11]{Demailly}, the restriction map
\[
 \mathcal R:\mathcal O(U)\longrightarrow\mathcal O(W\cap U)
\]
is a continuous surjection of Fréchet spaces. The open mapping theorem
therefore implies that the image of \(\{\widetilde f:\sup_{K_0}|\widetilde f|<1\}\) contains a neighborhood of zero defined by finitely many compact seminorms in the target. Taking their union gives a compact \(K_1\), and some \(\varepsilon>0\), such that
\[
 \sup_{K_1}|g|<\varepsilon
 \ \Longrightarrow\
 g=\mathcal R \widetilde g,\quad \sup_{K_0}|\widetilde g|<1.
\]
Apply this to \(\varepsilon g/(2b_0)\). Then \eqref{eq:AC-3.4} holds with \(C=2/\varepsilon\). The scaling is valid even if the seminorm of \(g\) on \(K_1\) is zero, because \(b_0>0\). \end{proof}

We apply this estimate chart by chart to the trace functions below.
Their extensions define the map from moments to observations.

\subsubsection{Uniform charts for all reference clusters}\label{uniform-charts-for-all-reference-clusters}

For each \(s\le K\), the set \(K_s\subset V_s\) of cycles with all
\(s\) points real and in \(\overline B(0,R)\) is compact. Properness of
the quotient and \eqref{eq:AC-3.2} allow us to cover it by finitely many
charts with uniform extension bounds.

\begin{lemma}[Uniform quotient charts]\label{lem:ac-uniform-charts}

There exist finitely many conjugation-invariant Stein polydisks \(U_{s,j}\), centered at real moment vectors, smaller polydisks, called inner charts, covering \(K_s\), and constants \(r_*>0\), \(C_{\mathrm{ext}}<\infty\), common to all \(s,j\), with the following properties.

Every centered reference moment vector \(p_b^0\) belongs to one of the
inner charts, and the corresponding chart contains the ambient
\(\ell^\infty\)-ball of radius \(4r_*\) about \(p_b^0\). Every cycle in
\(V_{s_b}\cap U_{s_b,j}\), translated by \(y_b\), has its points in
\begin{equation}
 D_b=\{x:\operatorname{dist}(\operatorname{Re}x,A_b)<\eta,
                 \ |\operatorname{Im}x|<\sigma\}.
 \label{eq:AC-3.5}
\end{equation}
Every bounded holomorphic function \(g\) on
\(V_{s_b}\cap U_{s_b,j}\) has a holomorphic extension
\(G\in\mathcal O(U_{s_b,j})\) satisfying
\[
 \sup_{\|p-p_b^0\|_\infty\le4r_*}|G(p)|
 \le C_{\mathrm{ext}}\sup_{V_{s_b}\cap U_{s_b,j}}|g|.
\]

\end{lemma}

\begin{proof}

Choose \(U\) so small that its cycles are close, in the matching
metric, to the real cycle at its center. Shrink the inner chart to
ensure the same for every reference cycle it contains. The triangle
inequality in \eqref{eq:AC-3.2} then implies \eqref{eq:AC-3.5}.
Choose a compact intermediate polydisk \(K_0\Subset U\) containing a
\(4r_{*,j}\)-ball about every point of the inner chart, and apply
Lemma~\ref{lem:ac-extension}. Take a finite subcover of each \(K_s\),
then the minimum of the \(r_{*,j}\) and maximum of the extension
constants. The compact set controlling the original function in
Lemma~\ref{lem:ac-extension} lies in \(V_s\cap U\), where its bound
is available.

\end{proof}

\subsection{Bandlimited kernels with complex cluster control}\label{bandlimited-kernels-with-complex-cluster-control}

We construct a bandlimited test for each cluster moment, so that its
observation has only finitely many frequencies. Gaussian approximation
and Fourier truncation give errors controlled on complex neighborhoods
and summable over all clusters.

\begin{lemma}[Uniform cluster kernels]
\label{lem:ac-kernels} Fix the geometric constants in \eqref{eq:AC-2.1}--\eqref{eq:AC-2.2}, an integer \(\kappa>d\), and \(\varepsilon>0\). For each cluster \(A_a\) and each multiindex \(1\le|\alpha|\le s_a\), there is a bandlimited Schwartz function \(h_{a,\alpha}\), real-valued
on \(\mathbb R^d\) and extending to an entire function on
\(\mathbb C^d\), such that
\begin{equation}
 \sup_{x\in D_b}
 \left|h_{a,\alpha}(x)
       -\mathbf1_{\{a=b\}}(x-y_a)^\alpha\right|
 \le \varepsilon(1+|y_a-y_b|)^{-\kappa}.                            \label{eq:AC-4.1}
\end{equation}
All functions have Fourier support in one fixed ball. After translation by \(y_a\), they have common bounds in every Schwartz seminorm. The bandwidth and seminorm bounds depend on the displayed uniform constants and \(\varepsilon\), but not on the cluster configuration or indices.

\end{lemma}

\begin{proof}
We first approximate the moment by Gaussian smoothing and then truncate
its Fourier transform. Both steps are uniform in the at most \(K\)
point positions, their multiplicities, and \(|\alpha|\le K\).
\smallskip\noindent\emph{Step 1: Gaussian approximation.}

Choose a fixed smooth bump equal to one on \(B(0,2)\), zero outside \(B(0,3)\), and lying in \([0,1]\). Define \(\chi^{\mathrm{ker}}_a\) as one minus the product of one minus its translates and rescalings at the points of \(A_a\), with scale \(\eta\). Thus
\[
 \chi^{\mathrm{ker}}_a=1\text{ on the }2\eta\text{-neighborhood of }A_a,\qquad
 \operatorname{supp}\chi^{\mathrm{ker}}_a\subset
       \text{the }3\eta\text{-neighborhood of }A_a.
\]
All translated derivatives are bounded uniformly, and the support lies in \(B(y_a,R+3\eta)\).

Let \(P_{a,\alpha}(x)=(x-y_a)^\alpha\). For \(0<\upsilon\le1\), set
\[
 Q_{a,\alpha,\upsilon}=e^{-\upsilon\Delta}P_{a,\alpha}
   =\sum_{j=0}^{\lfloor|\alpha|/2\rfloor}
             \frac{(-\upsilon)^j}{j!}\Delta^j P_{a,\alpha},
 \qquad f_{a,\alpha,\upsilon}=\chi^{\mathrm{ker}}_a Q_{a,\alpha,\upsilon}.
\]
This is a finite inverse-heat correction. Let
\[
 G_\upsilon(x)=(4\pi \upsilon)^{-d/2}\exp\!\left(-\frac{x\cdot x}{4\upsilon}\right),
 \qquad H_{a,\alpha,\upsilon}=G_\upsilon*f_{a,\alpha,\upsilon}.
\]
Compact support makes the convolution entire. The inverse-heat
correction ensures that convolving the Gaussian with the polynomial
\(Q_{a,\alpha,\upsilon}\) recovers \(P_{a,\alpha}\).

For \(x=u+iv\),
\begin{equation}
 |G_\upsilon(x-w)|=(4\pi \upsilon)^{-d/2}
       \exp\!\left(\frac{|v|^2-|u-w|^2}{4\upsilon}\right).             \label{eq:AC-4.2}
\end{equation}
If \(x\in D_a\), every point where \(\chi^{\mathrm{ker}}_a-1\ne0\) is at real distance at least \(\eta\) from \(u\). Hence
\[
 H_{a,\alpha,\upsilon}(x)-P_{a,\alpha}(x)
  =\int G_\upsilon(x-w)(\chi^{\mathrm{ker}}_a(w)-1)Q_{a,\alpha,\upsilon}(w)\,dw.
\]
The polynomial has a bound \(C(1+|w-u|)^K\), because \(u\) remains in a fixed-radius ball about \(y_a\). Splitting the Gaussian exponent as
\[
 e^{-|u-w|^2/(4\upsilon)}
 \le e^{-\eta^2/(8\upsilon)}e^{-|u-w|^2/(8\upsilon)}
\]
and using \(|v|<\sigma=\eta/4\) gives
\begin{equation}
 \sup_{D_a}|H_{a,\alpha,\upsilon}-P_{a,\alpha}|
        \le C e^{-a_1/\upsilon},                                      \label{eq:AC-4.3}
\end{equation}
for positive constants \(C,a_1\) independent of \(a,\alpha,\upsilon\). For the remaining Gaussian integral, rescaling by \(\sqrt \upsilon\)
cancels the normalization; the resulting polynomial moments are
uniformly bounded.

Now let \(b\ne a\), \(x=u+iv\in D_b\), and write \(d_{ab}=|y_a-y_b|\). The real distance \(\varrho\) from \(u\) to \(\operatorname{supp}\chi^{\mathrm{ker}}_a\) satisfies
\[
 \varrho\ge\delta-4\eta>\delta/2,\qquad
 d_{ab}\le \varrho+2R+4\eta.
\]
Formula \eqref{eq:AC-4.2}, the bounded support and amplitude of \(f\), and a second splitting of the Gaussian show
\begin{equation}
 \begin{split}
 (1+d_{ab})^\kappa|H_{a,\alpha,\upsilon}(x)|
 &\le C_\kappa \upsilon^{-d/2} e^{-(\delta-4\eta)^2/(8\upsilon)}
       e^{\sigma^2/(4\upsilon)}
       \sup_{\varrho\ge0}(1+\varrho)^\kappa e^{-\varrho^2/(8\upsilon)}\\
 &\le C'_\kappa e^{-a_2/\upsilon}.
 \end{split}                                                  \label{eq:AC-4.4}
\end{equation}
The supremum is uniformly bounded for \(\upsilon\le1\). Absorbing
the factor \(\upsilon^{-d/2}\) into a slightly weaker exponential,
choose \(\upsilon\) small enough that \eqref{eq:AC-4.3}--\eqref{eq:AC-4.4} are less than half of the required bounds.

\smallskip\noindent\emph{Step 2: Fourier truncation.}

We now truncate the Fourier transform while retaining the complex-strip
estimate. The transform is
\[
 \widehat H(\xi)=e^{-4\pi^2\upsilon|\xi|^2}\widehat f(\xi).
\]
Choose a real even smooth frequency cutoff \(\vartheta\), equal to one on the unit ball and zero outside the ball of radius two, and define
\[
 \widehat h_\rho(\xi)=
   \vartheta(\xi/\rho)e^{-4\pi^2\upsilon|\xi|^2}\widehat f(\xi).
\]
For fixed \(\upsilon\), we claim
\begin{equation}
 \sup_{\substack{u\in\mathbb R^d\\|v|\le\sigma}}
  (1+|u-y_a|)^\kappa
    |H(u+iv)-h_\rho(u+iv)|\longrightarrow0
 \quad(\rho\to\infty),                                           \label{eq:AC-4.5}
\end{equation}
uniformly over the cluster family.

Translate \(y_a\) to zero. The functions \(f(\,\cdot+y_a)\) are
bounded in \(C_c^\infty\) with common compact support, so their Fourier
transforms and the required derivatives are uniformly bounded. The
tail integrand is
\[
 (1-\vartheta(\xi/\rho))e^{-4\pi^2\upsilon|\xi|^2}
       \widehat f(\xi)e^{-2\pi v\cdot\xi}.
\]
For \(|v|\le\sigma\), every fixed finite number of its \(\xi\)-derivatives is bounded, on its support \(|\xi|\ge \rho\), by
\[
 C_{\upsilon,\sigma,N}(1+|\xi|)^{C_N}e^{-2\pi^2\upsilon|\xi|^2}
\]
once \(\rho\) is large enough. Derivatives of the cutoff contribute only bounded inverse powers of \(\rho\); the linear exponential from \(v\) is absorbed into half of the Gaussian exponent. These derivative bounds tend to zero in \(L^1\) on \(|\xi|\ge \rho\). Applying \((1-\Delta_\xi)^N\) under the Fourier integral, with \(2N\ge \kappa\), proves \eqref{eq:AC-4.5}.

For \(x=u+iv\in D_b\), the relation
\[
 |y_a-y_b|\le|u-y_a|+R+\eta
\]
converts the weight in \eqref{eq:AC-4.5} to that in
\eqref{eq:AC-4.1}. Choose \(\rho\) large enough that truncation
contributes at most the remaining half of the error; the Fourier
support lies in \(B(0,2\rho)\).

For fixed \(\upsilon,\rho\), the compactly supported Fourier functions
and their translated derivatives are uniformly bounded. Integration by
parts therefore bounds the translated Schwartz seminorms of \(h_\rho\).
Since \(f\) and the even cutoff are real, the kernels are real on
\(\mathbb R^d\). \end{proof}

\subsection{The map from moments to observations}\label{a-holomorphic-banach-map}

The kernel errors are summable over clusters. Using the extension bound,
we obtain a holomorphic map from moments to observations on an ambient
\(\ell^\infty\) ball, where we can apply a contraction argument.

Let
\[
 \mathcal J=\{(b,\alpha):\alpha\in\mathcal A_{s_b}\},
 \qquad \mathcal B=\ell^\infty(\mathcal J;\mathbb C),
\]
and let \(p^0=(p_b^0)\) be the reference moment vector.

Since the base points are \(\delta\)-separated and \(\kappa>d\),
\begin{equation}
 S_\kappa=\sup_a\sum_b(1+|y_a-y_b|)^{-\kappa}
       \le S^*(d,\delta,\kappa)<\infty.                     \label{eq:AC-5.1}
\end{equation}
Packing bounds give \(O(\delta^{-d}2^{jd})\) base points in a dyadic annulus of radius \(2^j\). Summing against \(2^{-j\kappa}\) proves \eqref{eq:AC-5.1}.

On the quotient chart belonging to cluster \(b\), define
\begin{equation}
 g_{a,\alpha;b}(p_b)
   =\sum_{j=1}^{s_b}h_{a,\alpha}(y_b+x_j)
       -\mathbf1_{\{a=b\}}p_{b,\alpha},                        \label{eq:AC-5.2}
\end{equation}
where \(p_b=\pi_{s_b}(x_1,\ldots,x_{s_b})\). When \(a=b\), the last term is a coordinate of this moment vector; otherwise it is omitted.

By the quotient construction in Section~\ref{a-concrete-affine-quotient},
the trace is holomorphic. Every chart cycle lies in \(D_b\), so
\eqref{eq:AC-4.1} gives
\begin{equation}
 |g_{a,\alpha;b}|
     \le K\varepsilon(1+|y_a-y_b|)^{-\kappa}.                       \label{eq:AC-5.3}
\end{equation}
Lemma~\ref{lem:ac-uniform-charts}, applied to each function
\eqref{eq:AC-5.2}, gives an ambient holomorphic extension \(G_{a,\alpha;b}\) on a neighborhood of the \(4r_*\)-ball about \(p_b^0\), with
\begin{equation}
 |G_{a,\alpha;b}|
    \le C_{\mathrm{ext}}K\varepsilon
             (1+|y_a-y_b|)^{-\kappa}.                              \label{eq:AC-5.4}
\end{equation}
Choose the extensions independently of phase. Since the kernels are real
on \(\mathbb R^d\), the traces
satisfy \(g_{a,\alpha;b}(\overline p_b)=\overline{g_{a,\alpha;b}(p_b)}\).
The polydisks and moment varieties are conjugation-invariant. Averaging
with the complex-conjugate reflection therefore preserves the extension
of \(g_{a,\alpha;b}\) and its bound while making it real on real
moment coordinates.

Set
\begin{equation}
 \mathcal E(p)_{a,\alpha}
     =\sum_bG_{a,\alpha;b}(p_b),
 \qquad
 \Psi(p)=p+\mathcal E(p).                                \label{eq:AC-5.5}
\end{equation}
Each row sum converges absolutely, with uniform row bounds on the
\(\ell^\infty\) ball of radius \(4r_*\).

\begin{lemma}[Holomorphy of the error map]
\label{lem:ac-banach} The map \(\mathcal E\) is holomorphic on the \(\mathcal B\)-ball of radius \(2r_*\) about \(p^0\). On the smaller ball of radius \(r_*\),
\begin{equation}
 \|D\mathcal E(p)\|
    \le C_1K\varepsilon S_\kappa/r_*,                                \label{eq:AC-5.6}
\end{equation}
where \(C_1\) depends only on the fixed extension constant.

\end{lemma}

\begin{proof}
We combine Cauchy estimates on each row with the uniform summability
in \eqref{eq:AC-5.1}.

For \(p\) in the \(2r_*\)-ball, apply Cauchy's estimate on each
finite-dimensional block, using the distance to the \(4r_*\)-boundary:
\[
 \|DG_{a,\alpha;b}(p_b)\|
   \le C_1K\varepsilon r_*^{-1}
                    (1+|y_a-y_b|)^{-\kappa}.
\]
The same one-variable Cauchy estimate along the complex line
\(p_b+z\dot p_b\) gives, uniformly for \(\|\dot p\|_\infty<r_*/2\),
\[
 \begin{split}
 &|G_{a,\alpha;b}(p_b+\dot p_b)-G_{a,\alpha;b}(p_b)
                     -DG_{a,\alpha;b}(p_b)\dot p_b|\\
 &\qquad\le
 C_2K\varepsilon r_*^{-2}
       (1+|y_a-y_b|)^{-\kappa}\|\dot p\|_\infty^2.
 \end{split}
\]
By \eqref{eq:AC-5.1}, the derivative row sums define a bounded
complex-linear operator on \(\mathcal B\). Summing the remainders and
taking the supremum gives a uniform \(O(\|\dot p\|_\infty^2)\) remainder
in \(\mathcal B\), proving complex Fr\'echet differentiability and
\eqref{eq:AC-5.6}. The same Cauchy estimate on a smaller ball proves
continuity of the derivative and hence holomorphy. \end{proof}

\paragraph{Choice of constants.}
Fix \(\delta,\eta,\sigma\) as functions of \((d,K)\), satisfying
\eqref{eq:AC-delta-choice} and \eqref{eq:AC-2.2}, and take
\(\kappa=d+1\). Lemma~\ref{lem:ac-uniform-charts} then fixes \(r_*\)
and \(C_{\mathrm{ext}}\), while \eqref{eq:AC-5.1} fixes the packing
bound \(S^*\). Let \(C_1\) be the Cauchy-estimate constant for these
extension charts, and choose \(\varepsilon>0\) so that
\(C_1K\varepsilon S^*/r_*<1/8\). In particular,
\begin{equation}
 \varepsilon<\frac{r_*}{8C_1KS_\kappa}.
 \label{eq:AC-5.7}
\end{equation}
In the proof of Lemma~\ref{lem:ac-kernels}, choose the heat time
\(\upsilon\) and then the frequency cutoff \(\rho\) for this
\(\varepsilon\). Finally choose \(\rho_*>2\rho\), so the open
ball \(B(0,\rho_*)\) contains all kernel Fourier supports. These
constants depend only on \((d,K)\).

\subsection{Holomorphic extension of cluster cycles}\label{finite-band-data-recover-the-entire-cycle-family}

The kernel observations are finite exponential sums in the phase.
Inverting the map from moments to observations extends all cluster
moments holomorphically; the identity theorem ensures that they remain
in the symmetric products.

\begin{proposition}[Holomorphic cluster cycles]\label{prop:ac-holomorphic-clusters}
Under the temporary finite-generation hypothesis, the cluster moments
extend holomorphically to a complex phase neighborhood, with values
in the corresponding symmetric products. For a fixed measure and
spectral basis, the radius can be chosen uniformly over all real
reference phases.
\end{proposition}
\begin{proof}
\emph{Step 1: observations as trigonometric polynomials.}

Let
\[
 \Phi(t)_{a,\alpha}=\int h_{a,\alpha}(x)\,d\mu_t(x).
\]
The kernels have Fourier support in a common compact ball, so only
one finite subset \(\Lambda'_\rho\subset\Lambda'\) occurs. By
\eqref{eq:ac-fourier-observation},
\begin{equation}
 \Phi(t)=
   \sum_{\xi\in\Lambda'_\rho}
       v_\xi e^{2\pi ik(\xi)\cdot t},
 \qquad
 (v_\xi)_{a,\alpha}
       =c_\xi\widehat h_{a,\alpha}(-\xi).                     \label{eq:AC-6.1}
\end{equation}
The translated Schwartz bounds give a uniform \(L^1\) bound on the
kernels and hence \(v_\xi\in\mathcal B\). Thus \eqref{eq:AC-6.1}
extends to an entire \(\mathcal B\)-valued finite exponential sum in
the complex phase variables.

For real phases near zero, the moment vector \(p(t)\) of
Lemma~\ref{lem:ac-persistence} is continuous at \(t=0\). The ambient
extensions restrict to their traces on the cycle varieties, so summing
over clusters gives
\begin{equation}
 \Psi(p(t))=\Phi(t).                              \label{eq:AC-6.2}
\end{equation}
Translation boundedness and Schwartz decay give absolute convergence
of the kernel sum. The row estimate \eqref{eq:AC-5.4} gives absolute
convergence of the error sum, justifying the regrouping by clusters.

\smallskip\noindent\emph{Step 2: inversion and cycle relations.}

To invert the map from moments to observations, write
\(\mathfrak a_0=\Phi(0)=\Psi(p^0)\).
For \(\|\mathfrak a-\mathfrak a_0\|<r_*/2\), seek \(p=p^0+u\), \(\|u\|\le r_*\), from
\begin{equation}
 u=(\mathfrak a-\mathfrak a_0)-\bigl(\mathcal E(p^0+u)-\mathcal E(p^0)\bigr).       \label{eq:AC-6.3}
\end{equation}
By \eqref{eq:AC-5.6}--\eqref{eq:AC-5.7}, the right side maps the ball
into itself with contraction constant at most \(1/8\). The iterates are
holomorphic in \(\mathfrak a\) and converge uniformly on smaller
\(\mathfrak a\)-balls to the unique holomorphic inverse.

For complex \(\theta\) sufficiently close to zero, the extension
\(\Phi(\theta)\) of \eqref{eq:AC-6.1} lies in this
\(\mathfrak a\)-ball, so \eqref{eq:AC-6.3} defines a holomorphic
vector \(p^*(\theta)\). For real \(t\) close to zero, Lemma~\ref{lem:ac-persistence} places
the moment vector \(p(t)\) in the same neighborhood. It solves
\eqref{eq:AC-6.2}, so uniqueness gives
\[
 p^*(t)=p(t)\quad\text{for real }t.
\]

Every polynomial relation defining \(V_{s_b}\) vanishes on \(p_b^*(t)\) for
real \(t\) in an open phase neighborhood. By the identity theorem in
the complex phase variables, it vanishes on the holomorphic extension:

\[
 p_b^*(\theta)\in V_{s_b}
\]
throughout a sufficiently small complex neighborhood.

The centered cluster tests in Lemma~\ref{lem:ac-persistence}, over all
reference phases, lie in one compact family. Uniform continuity on
the torus therefore makes persistence uniform in the reference phase.
The radius \(r_*\), contraction bound and kernel \(L^1\) bounds are
uniform as well. For the fixed spectral basis, \eqref{eq:AC-6.1}
involves one finite set of indices \(k(\xi)\); the kernel bounds
control its coefficient norms uniformly, while real phase translations
contribute only factors of modulus one. Thus the complex phase
derivative of \(\Phi\) is uniformly bounded on a fixed tube about the
real torus. A sufficiently small common radius keeps \(\Phi\) in the
\(r_*/2\)-ball required by \eqref{eq:AC-6.3}, so the inverse and
identity-theorem arguments apply at every real reference phase. This
radius may depend on the spectral data.
\end{proof}

\subsection{The canonical set and its analytic germ}\label{the-canonical-analytic-toral-set}

We first record how the coefficients of a holomorphic finite cycle
are carried by its reduced incidence.

\begin{lemma}[Coefficients of a holomorphic finite cycle]
\label{lem:integer-incidence-coefficients}
Let $p:U\to\operatorname{Sym}^s(\C^d)$ be a holomorphic family on a
connected complex phase neighborhood, with its particles confined to a
fixed relatively compact target neighborhood. Let
\[
 \mathcal I=\{(\theta,x):x\in\supp p(\theta)\}_{\mathrm{red}}.
\]
After shrinking $U$ to a sufficiently small connected neighborhood of a
given point, the projection $\mathcal I\to U$ is finite and
$\mathcal I$ is pure of dimension $\dim U$. It carries a unique effective analytic cycle
$\mathfrak I$ with constant positive integer component coefficients
and weighted fibers $p(\theta)$. Namely, if $\mathcal I_j$ are the components,
with coefficients $\varkappa_j$, then for every $(\theta,x)\in U\times\C^d$,
a term being zero when $(\theta,x)\notin\mathcal I_j$,
\[
 \operatorname{mult}_x p(\theta)
   =\sum_j\varkappa_j
       \deg_{(\theta,x)}(\mathcal I_j\longrightarrow U).
\]
Here $\operatorname{mult}_x p(\theta)$ denotes the coefficient of $x$
in the finite cycle $p(\theta)$. The coefficients and the identity are
compatible with restrictions to isolating neighborhoods and with
identities of holomorphic cycle families.
\end{lemma}
\begin{proof}
The universal incidence map
\[
 (\mathbb C^d)^s/S_{s-1}\longrightarrow(\mathbb C^d)^s/S_s
\]
is finite. Its reduced base change is the incidence, and the particles
remain in a fixed relatively compact target neighborhood.
Shrink the connected base around a reference point so that the finitely
many local components meeting its finite fiber exhaust the incidence.
Let $\mathcal D_{\mathrm{exc}}\subset U$ be the union of the images of
components not dominating $U$. Since the map is finite and proper,
these images are proper analytic subsets
\cite[II, Theorem~8.8]{Demailly}. For any $(\theta_0,x_0)\in\mathcal I$, choose
$\theta_j\notin\mathcal D_{\mathrm{exc}}$ tending to $\theta_0$.
Continuity in $\operatorname{Sym}^s$ supplies particles $x_j\to x_0$.
The incidence is therefore the closure of its part over
$U\setminus\mathcal D_{\mathrm{exc}}$. The finitely many dominating
components are closed and contain this closure, excluding vertical
components. Each finite dominating component has dimension
$\dim U$.

Remove from $U$ the images of the singular, ramification and
component-intersection loci. Their union is proper analytic, and over
its complement the distinct particles form local holomorphic graphs. Split the cycle in disjoint neighborhoods
of these particles, using Lemma~\ref{lem:ac-splitting}. Each split cycle has locally constant positive integer degree. Its
support is one graph, whose coefficient is that degree.

The regular locus of each irreducible incidence component is connected
\cite[II, Proposition~4.24]{Demailly} and remains so after removing the
proper analytic exceptional preimage \cite[II, Remark~4.2]{Demailly}.
The coefficient consequently has one value $\varkappa_j$ on that
component of the reduced support.

At an arbitrary $(\theta,x)$, split the family in an isolating neighborhood
of $x$. Its total degree is the fixed integer
$\operatorname{mult}_x p(\theta)$. At a nearby phase outside the
exceptional set its particles on component $\mathcal I_j$ have coefficient
$\varkappa_j$, and the number of these particles is the local reduced
degree of that component, by conservation of number
(Lemma~\ref{lem:local-degree-general}). Summation proves the identity, and the same
argument proves independence of the isolating neighborhood. Since the
generic graphs determine the coefficients uniquely, equal holomorphic
families induce equal cycles on their common domain.
\end{proof}

Pulling back the condition that a cluster contain the origin constructs
the canonical germ. Complex covariance identifies the incidence locally
with a product, determining the codimension, transporting component
coefficients and identifying the real fibers.

\begin{proposition}[Canonical analytic germ]\label{prop:ac-canonical-germ}
Under the temporary finite-generation hypothesis, the conclusions in
Theorem~\ref{thm:all-fq-toral-cycles}(c)--(e) hold.
Moreover, the local projection of \(M^{\mathbb C}\) along
\(L\mathbb C^d\) is finite at every point of \(M\) and has the following real-fiber property:
in real complementary logarithmic coordinates its full inverse image
of the real base is its real locus.
\end{proposition}

\label{analyticity-and-purity}

\begin{proof}\emph{Analyticity and purity.}
The hull, persistence, quotient charts, kernels and local inverse apply
to positive integer cluster masses bounded by \(K\), as provided by
Lemma~\ref{lem:ac-translation-bound}. Proposition~\ref{prop:ac-holomorphic-clusters}
therefore gives holomorphic symmetric-product families with these
multiplicities.

Let
\[
 M=\{t\in\mathbb T^n:0\in\operatorname{supp}\mu_t\}.
\]
An atom of mass at least one at zero persists in a vague limit, so this
set is closed and therefore compact.

At a phase in \(M\), take the cluster containing the origin and write
\[
 Z_b(\theta)=y_b+\operatorname{supp}p_b^*(\theta).
\]
The set \(Z_b(\theta)\) records positions in physical space; \(p_b^*\)
records them relative to \(y_b\). The condition that a cycle
contain \(-y_b\) is algebraic on \(V_{s_b}\): it is the image under
the finite quotient of the union of the coordinate subspaces
\(x_j=-y_b\). Define \(M^{\mathbb C}\) to be the reduction of the
analytic pullback of this incidence condition by the holomorphic moment
map. At nearby real phases, Lemma~\ref{lem:ac-persistence} and the
separation in \eqref{eq:AC-2.1} keep all atoms from the other clusters
at distance greater than \(\delta/2\) from the origin. Hence
\(0\in\operatorname{supp}\mu_t\) if and only if \(0\in Z_b(t)\),
so its real locus is \(M\). The incidence condition has equations
with real coefficients, and the moment map commutes with conjugation
by the identity theorem, since its real-phase values are real.
Their pullbacks therefore give real-analytic defining functions for
\(M\); reduction does not change this zero set. At a phase outside
\(M\), uniform support continuity makes \(M\) locally empty.

Let \(U_0\subset\mathbb C^n\) be a small connected phase ball. For a cluster of degree \(s=s_b\), its reduced support incidence
\[
 \mathcal I=\{(\theta,x):x\in Z_b(\theta)\}
\]
is finite over \(U_0\), since the universal incidence
\[
 (\mathbb C^d)^s/S_{s-1}
       \longrightarrow(\mathbb C^d)^s/S_s
\]
is finite and our incidence is its reduced base change. The particles
remain in a fixed compact target neighborhood.

By Lemma~\ref{lem:integer-incidence-coefficients}, the reduced incidence
is pure of dimension \(n\) and carries an effective analytic cycle
\(\mathfrak I\) with positive integer component coefficients.

\smallskip\noindent\emph{Complex covariance and codimension.}\label{complex-covariance-and-codimension}

Real covariance \eqref{eq:AC-1.2} implies, for small real phase and physical increments,
\[
 p_b(t+Lu)=p_b(t)-u
\]
as unordered cycles, on neighborhoods small enough to retain the same
cluster. Translation is polynomial in the moment coordinates, so both
sides extend holomorphically in phase and physical increment. The
identity theorem extends the relation to small complex
\((\theta,u)\).

The local incidence is therefore invariant under
\[
 (\theta,x)\longmapsto(\theta+Lu,x-u).
\]
The biholomorphic coordinate change
\[
 (\theta,x)\longmapsto(\theta+Lx,x)
\]
therefore identifies its germ at a point with \(x=0\) with
\[
 M^{\mathbb C}\times\mathbb C^d.
\]
Purity of the incidence proves that \(M^{\mathbb C}\) is pure of codimension \(d\).
Transport the component coefficients of \(\mathfrak I\) through this
identification. The components of the product are exactly the products
of the local components of \(M^{\mathbb C}\) with \(\mathbb C^d\);
the transported coefficients define \(\mathfrak M^{\mathbb C}\).

To glue germs arising from different reference phases, fix a real
phase in their overlap and choose a physical neighborhood \(W\) of the
origin, contained in \(B(0,\delta/2)\), whose boundary misses both
cycle supports at that phase. Translate each centered cycle by its own
center to put both families in physical coordinates. Lemma~\ref{lem:ac-splitting} restricts
these families to the common neighborhood. At nearby real phases these
restricted cycles are the same part of \(\mu_t\), so they have the
same degree and moment coordinates. The identity theorem identifies their holomorphic moment maps and
reduced incidences; by the last assertion of
Lemma~\ref{lem:integer-incidence-coefficients}, the effective cycles agree as well. A finite cover of the compact
real torus and shrinking give a common complex neighborhood.

\smallskip\noindent\emph{Unit-mass phases and real sections.}\label{smooth-unit-mass-phases-and-local-real-rootedness}

If the atom at zero has mass one, isolate it from the other atoms in
its cluster. Lemma~\ref{lem:ac-splitting} then gives a degree-one
local cycle, hence a single holomorphic position \(\mathfrak x(\theta)\). Covariance gives
\[
 \mathfrak x(\theta+Lu)=\mathfrak x(\theta)-u,\qquad D\mathfrak x(\theta)L=-I_d.
\]
It follows that \(\mathfrak x\) is a submersion. Its zero set
\(M^{\mathbb C}\) is smooth there and transverse to the physical
\(d\)-plane.

Let \(t_0\in M\) and let \(s_0=\mu_{t_0}(\{0\})\).
At this phase, keep the unordered local cycle.
For small complex \(x\), covariance gives
\[
 \theta+Lx\in M^{\mathbb C}
 \quad\Longleftrightarrow\quad
 x\in Z_b(\theta).
\]
For real \(\theta\), the cycle has only real points, so the small
complex sections along the physical plane have only real zeros.

Isolate the origin from the other atoms of its cluster.
Lemma~\ref{lem:ac-splitting} gives a holomorphic cycle of degree
\(s_0\), with support \(Z_0(\theta)\) in a small ball
\(B\subset\mathbb C^d\), for \(\theta\) in a connected complex
ball \(U\) about \(t_0\). Its reduced incidence
\[
 \mathcal I_0=\{(\theta,x)\in U\times B:x\in Z_0(\theta)\}
\]
has a finite projection \(\rho:\mathcal I_0\to U\), and the
fiber over \(t_0\) consists of the single point \((t_0,0)\).
By Lemma~\ref{lem:integer-incidence-coefficients}, the associated
weighted incidence cycle has weighted local degree over \(t_0\) equal to
\(s_0\).

As germs at \((t_0,0)\), complex covariance gives
\[
 \theta+Lx\in M^{\mathbb C}
 \quad\Longleftrightarrow\quad x\in Z_0(\theta).
\]
To obtain this identity, shrink the phase and physical neighborhoods
so that translation retains the same local cycle. The map
\((\theta,x)\mapsto(\theta+Lx,x)\) identifies \(\mathcal I_0\)
with the germ of \(M^{\mathbb C}\times\mathbb C^d\), under which
\(\rho\) becomes \((p,x)\mapsto p-Lx\).
Choose a real linear complement to \(L\mathbb R^d\) and let \(N\)
be its complexification, a complement to \(L\mathbb C^d\),
and write \(p=t_0+p_N+Ly(p)\). In the coordinates
\((p,x)\mapsto(p,y(p)-x)\), this map is
\(\pi|_{M^{\mathbb C}}\times\mathrm{id}_{\mathbb C^d}\),
where \(\pi(p)=p_N\) is the projection along the physical plane.
Consequently \(\pi|_{M^{\mathbb C}}\) is finite, and its weighted
local degree on \(\mathfrak M^{\mathbb C}\) is \(s_0\): it is the
sum of each positive integer component coefficient times the local
intersection degree of Definition~\ref{def:local-intersection-degree}.
For a real phase \(t\) and a point
\(x\in\operatorname{supp}\mu_t\cap U_t\), apply this argument
at \(t_0=t+Lx\). Covariance \eqref{eq:AC-1.2} gives
\(\mu_{t_0}(\{0\})=\mu_t(\{x\})\).

For \(y\in M^{\mathbb C}\) whose complementary coordinate is real,
write \(y=t_0+u+Lx\), with \(u\) real in the complement. By complex
covariance, \(x\) is a particle of the cycle at the real phase
\(t_0+u\). These particles are real, so \(x\) and \(y\) are real
as required.

If \(t\notin M\), the closed analytic set \(M^\C\), whose real locus
is \(M\), misses \(t\). Choose \(U_t\) small enough that its physical
section is empty. Since \(0\notin\operatorname{supp}\mu_t\), we may
shrink \(U_t\) further so that \(\operatorname{supp}\mu_t\cap U_t\)
is empty as well.
\end{proof}
\begin{remark}\label{rem:ac-cycle-multiplicity}
The universal incidence need not be flat when $d\ge2$. For $s=d=2$,
its scheme-theoretic fiber over $2[0]$ is
\[
 \C[w_1,w_2]/(w_1^2,w_1w_2,w_2^2),
\]
which has length three, whereas the generic degree is two. Reduction
can also lose cycle coefficients when $d=1$: for
$p(\theta)=2[x(\theta)]$, the reduced incidence is a single graph.
Accordingly, the coefficients of $\mathfrak M^\C$ are obtained from
the generic graphs of the reduced incidence; its weighted local degrees
recover the fiber masses.
\end{remark}

\begin{lemma}[Dense real sheets of a local finite projection]
\label{lem:integer-real-sheets}
Let $(A,p)$ be a reduced pure $c$-dimensional analytic germ in logarithmic
complex coordinates, with $p$ real. Let a real linear projection to
$\C^c$ be finite on this germ, and suppose that its full inverse image
of the real base consists of real points. Then, in every local irreducible
branch at every real point, smooth unramified real graphs occur arbitrarily
near that point. These graphs have real dimension $c$ and are maximal
totally real in that branch. In particular the smooth transverse real
locus is dense in the real locus, and a holomorphic function vanishing on
that real locus vanishes on every local branch.
\end{lemma}
\begin{proof}
For $c=0$ this is immediate. Otherwise isolate a point of a finite fiber
and shrink the base and source as in
Lemma~\ref{lem:local-degree-general}(iii).
Each local branch is proper and finite over the smaller base; its image
is analytic of dimension $c$ and therefore fills that connected base.
Lemma~\ref{lem:local-degree-general}(ii)--(iii) applies to each local
branch inside any prescribed neighborhood of the isolated point.

The singular, critical, and pairwise intersection loci have dimension
less than $c$ \cite[II, Proposition~4.26]{Demailly}. Their images form
a proper analytic subset by finiteness and
\cite[II, Theorem~8.8]{Demailly}. Remove this subset from the base.
Its intersection with the real base is nowhere dense: a holomorphic equation vanishing on a real open
ball vanishes identically. Choose real regular values tending to the central
value. The branch has preimages in the prescribed neighborhood, all real
by hypothesis. The inverse function theorem gives a holomorphic graph
near each preimage, with real restriction over a real base ball. The
complexification of its real tangent space is the complex tangent
space of the graph, proving maximal total reality.

The identity theorem in the base variables shows that a holomorphic
function vanishing on the real locus vanishes on each such graph.
Analytic continuation extends this vanishing over the irreducible
branch. Since the prescribed neighborhood was arbitrary, the same
construction proves density.
\end{proof}

\subsection{Finite generation}\label{local-finite-band-rigidity}

Let \(\mathcal C_K\) be the positive integer atomic measures on \(\mathbb R^d\) having mass at most \(K\) in each unit ball. The uniform weak topology is defined by the pseudometrics
\[
 d_g(\mu,\nu)=\sup_x\left|\int g(y-x)\,d(\mu-\nu)(y)\right|,
 \qquad g\in C_c^\infty(\mathbb R^d).
\]

\begin{lemma}[Local finite-band rigidity]
\label{lem:ac-rigidity} There is \(\rho_*=\rho_*(d,K)<\infty\) such that, for every \(\nu\in\mathcal C_K\), some neighborhood \(\mathcal U\subset\mathcal C_K\) of \(\nu\) in the relative uniform-weak topology satisfies
\begin{equation}
 \nu'\in\mathcal U,\quad
 \widehat{\nu'}|_{B(0,\rho_*)}=\widehat\nu|_{B(0,\rho_*)}
 \quad\Longrightarrow\quad \nu'=\nu.                           \label{eq:HF-1.1}
\end{equation}
Fourier equality is distributional; the radius \(\rho_*\) is chosen
so that this open ball contains the closed Fourier supports of the
constructed kernels.

\end{lemma}

\begin{proof}
The preceding construction partitions \(\nu\) into
\(\delta\)-connected clusters with masses \(s_b\le K\), uniformly bounded
diameters and separated centers. Under small uniform weak perturbations,
positivity and integrality preserve every cluster mass, while the
centered moment vectors remain uniformly close to \(p^0\). The quotient
charts, extension constants and kernel bandwidth depend only on \(d,K\) and the fixed geometric constants.

The kernels \(h_{a,\alpha}\) have Fourier support in this ball, so
\eqref{eq:HF-1.1} implies equality of their integrals. Both moment
vectors solve
\[
 \Psi(p)=\Phi,\qquad
 \Psi(p)=p+\mathcal E(p),\quad \|D\mathcal E\|<1/8
\]
in the same small \(\ell^\infty\) ball. By contraction, \(p=p^0\);
since multisymmetric coordinates separate unordered cycles, the clusters
and hence the measures agree.
\end{proof}

The normalized cluster and moment cutoffs form a compact family.
Finite approximation by finitely many tests and their translates
determines a neighborhood preserving integer counts and the required
moment bound.

\subsubsection{The compact dual hull}\label{a-compact-hull-without-assuming-finite-generation}

For an integer-mass FQ measure \(\mu\), put
\[
 \Gamma=\langle\Lambda'\rangle_{\mathbb Z}
       \subset\mathbb R^d,\qquad
 \mathcal K=\operatorname{Hom}(\Gamma,S^1),
\]
where \(S^1\subset\C^*\) is the unit circle and \(\Gamma\) has the
discrete topology. Local finiteness of the spectrum implies countability,
so \(\mathcal K\) is a compact metrizable abelian group, realized as a
closed subgroup of \((S^1)^\Gamma\); write \(e\) for its identity.
The homomorphism
\[
 \iota(x)(\xi)=e^{2\pi i\xi\cdot x}
\]
has dense image by Kronecker's theorem: on each finite set of
frequencies, the characters of the compact dual satisfy the same
integer relations as the image.

By positivity and the spectral gap around zero, the uniform local mass
bound places all vague limits of translates in one
\(\mathcal C_K\). For a bandlimited Schwartz \(h\), define
\begin{equation}
 F_h(\omega)=\sum_{\xi\in\Lambda'}
       c_\xi\widehat h(-\xi)\omega(\xi).                       \label{eq:HF-2.1}
\end{equation}
This finite sum, evaluated at \(\omega=\iota(x)\), is the integral of
\(h\) against \(\tau_{-x}\mu\).

Vague limits of translates define a unique measure
\(\mu_\omega\) for each \(\omega\in\mathcal K\). For uniqueness,
approximate compactly supported smooth tests by smooth frequency
truncations; the common local mass bound controls the error
uniformly.
The family is then vaguely continuous, lies in \(\mathcal C_K\), and
satisfies
\begin{equation}
 \mu_{\omega\iota(x)}=\tau_{-x}\mu_\omega.                      \label{eq:HF-2.2}
\end{equation}
For every \(g\in C_c^\infty\), \(F_g(\omega)=\int g\,d\mu_\omega\)
is continuous by the same approximation. Uniform continuity on the
compact group implies convergence in the topology of
Lemma~\ref{lem:ac-rigidity}:
\begin{equation}
 \begin{split}
 &\sup_x\left|\int g(y-x)\,d(\mu_\omega-\mu_e)(y)\right|\\
 &\qquad\le
 \sup_{\omega'\in\mathcal K}
      |F_g(\omega\omega')-F_g(\omega')|
 \longrightarrow0\quad(\omega\to e).
 \end{split}                                                  \label{eq:HF-2.3}
\end{equation}

To prove injectivity, suppose \(\mu_\omega=\mu_e\) and test the
Fourier transform near an isolated spectral atom \(\xi\). By
\eqref{eq:HF-2.1}, \(c_\xi\omega(\xi)=c_\xi\). Since every
\(c_\xi\ne0\) and the spectrum generates \(\Gamma\), we have \(\omega=e\).

\subsubsection{Finite generation from the local inverse}\label{finite-generation-from-the-local-inverse}

Restriction to a finite band maps the compact dual hull onto a torus.
Its kernel is discrete by local rigidity and finite by compactness.
Duality places the spectral group inside a finite extension of the
group generated by the band.

\begin{proposition}[Finite generation of the spectrum]
\label{prop:finite-spectral-group}
Every integer-mass FQ measure has a finitely generated spectral group.
\end{proposition}

\begin{proof}
Let
\[
 \Gamma_0=\langle\Lambda'\cap B(0,\rho_*)\rangle_{\mathbb Z}.
\]
By local finiteness of the spectrum, \(\Lambda'\cap B(0,\rho_*)\) is
finite, so \(\Gamma_0\) is finitely generated. Since the group is
torsion-free, its compact dual is a torus, and restriction is a
continuous surjection
\[
 \pi:\mathcal K\longrightarrow\operatorname{Hom}(\Gamma_0,S^1):
\]
its compact image contains the dense image of the homomorphism in that torus.

Let \(H_0=\ker\pi\). For \(\omega\in H_0\), the Fourier data of
\(\mu_\omega\) and \(\mu_e\) agree on the chosen band. If \(\omega\)
is sufficiently close to \(e\), the measures lie in the neighborhood
of Lemma~\ref{lem:ac-rigidity} by \eqref{eq:HF-2.3}. It follows that
\(\mu_\omega=\mu_e\), and injectivity implies \(\omega=e\).
The identity \(e\) is therefore isolated in the compact group \(H_0\).
By translation every point is isolated, so \(H_0\) is finite.

Let \(e_0\) be its exponent. For \(\xi\in \Gamma\), the character \(\omega\mapsto\omega(e_0\xi)\) is trivial on \(H_0\), hence descends to
\[
 \mathcal K/H_0\cong\operatorname{Hom}(\Gamma_0,S^1).
\]
Since characters of a torus are integer monomials, some
\(g_0\in \Gamma_0\) satisfies \(\omega(e_0\xi)=\omega(g_0)\) for
all \(\omega\in\mathcal K\). Evaluation on \(\iota(x)\) for all
real \(x\) yields \(e_0\xi=g_0\), and hence
\begin{equation}
 \Gamma_0\subseteq \Gamma\subseteq e_0^{-1}\Gamma_0.                            \label{eq:HF-3.1}
\end{equation}
The quotient on the right by \(\Gamma_0\) is finite, so generators
of \(\Gamma_0\) and finitely many coset representatives generate
\(\Gamma\).
\end{proof}

\begin{proof}[Proof of Theorem~\ref{thm:all-fq-toral-cycles}]
The polynomial growth in Definition~\ref{def:fourier-quasicrystal}
implies temperedness. Proposition~\ref{prop:finite-spectral-group}
removes the temporary finite-generation assumption. The compact-dual
hull construction and the construction in
Section~\ref{construction-from-finite-band-observations} prove (a),
Proposition~\ref{prop:ac-holomorphic-clusters} proves (b), and
Proposition~\ref{prop:ac-canonical-germ} proves (c)--(e).
\end{proof}

\subsection{The flux measure}\label{positive-toral-flux}

We represent the spectral coefficients by a measure supported on
\(M\). Its disintegration in normal coordinates yields the finite-sheet
model for the next section. Here \(\theta\) denotes a real logarithmic
coordinate on the torus.

Choose a basis of the spectral group, finitely generated by
Proposition~\ref{prop:finite-spectral-group}, and let \(L\) be the
matrix whose rows form this basis. Identify the hull with
\(\mathbb T^n\), and put
\[
 M=\{t:0\in\operatorname{supp}\mu_t\}.
\]
\begin{definition}[Flux measure]\label{def:ac-flux}
For \(\varphi\ge0\) Schwartz with \(\int\varphi=1\), the flux
measure is defined by
\begin{equation}
 m(f)=\int_{\mathbb T^n}
       \sum_{x\in\mu_t}\varphi(x)f(t+Lx)\,dt.                    \label{eq:HF-4.1}
\end{equation}
Atoms are counted with their masses. The local mass bound and Schwartz
decay ensure a finite positive measure.

\end{definition}

\begin{proposition}\label{prop:ac-flux-support}
The measure \(m\) is independent of \(\varphi\) and has support \(M\).
For \(k\in\mathbb Z^n\), put \(\xi=L^tk\). Then
\begin{equation}
 \widehat m(k)=
 \begin{cases}
 c_\xi,&\xi\in\Lambda',\\
 0,&\xi\notin\Lambda'.
 \end{cases}                                                 \label{eq:HF-4.2}
\end{equation}
\end{proposition}
\begin{proof}
Take \(f(t)=e^{-2\pi ik\cdot t}\) in \eqref{eq:HF-4.1}.
For a bandlimited \(\varphi\), formula \eqref{eq:HF-4.2} follows from
\eqref{eq:ac-fourier-observation}, applied to
\(\varphi(x)e^{-2\pi i\xi\cdot x}\); its Fourier transform at
\(-\xi\) is \(\int\varphi=1\). Smooth frequency truncation and the
common local mass bound extend the identity to every Schwartz
\(\varphi\) of integral one, proving independence of \(m\) from
\(\varphi\).

We next prove that its support is \(M\). By covariance, each phase
\(t+Lx\) in \eqref{eq:HF-4.1} belongs to \(M\), so
\(\operatorname{supp}m\subseteq M\). Conversely, suppose \(t_0\in M\),
and take \(f\ge0\) positive near \(t_0\) and \(\varphi\) positive
near zero. For every \(t\) in a sufficiently small open neighborhood
of \(t_0\), vague continuity and the atom at zero ensure that there
are particles \(x\) close to zero. Positivity and integrality bound
their contribution to \(\varphi(x)f(t+Lx)\) below by a positive
constant. Integrating over this open set yields \(m(f)>0\).

\end{proof}

\subsubsection{Finite-sheet disintegration}\label{direct-finite-sheet-disintegration}

Write
\[
 V=\operatorname{im}L,\qquad V^\perp=\ker L^t,\qquad c=n-d,
\]
and let \(B:\mathbb R^c\to V^\perp\) have orthonormal columns. In logarithmic coordinates,
\begin{equation}
 \theta=Bu+Lv,\qquad
 J=|\det[B,L]|=\sqrt{\det(L^tL)}.                              \label{eq:HF-5.1}
\end{equation}

\begin{lemma}[Normal-sheet disintegration]
\label{lem:ac-flux} If \(U\subset\mathbb R^n\) is a bounded torus logarithm chart and \(f\) is continuous and compactly supported in it, then
\begin{equation}
 m(f)=J\int_{\mathbb R^c}
        \sum_{\substack{y\in\mu_{Bu}\\Bu+Ly\in U}}
             f(Bu+Ly)\,du.                                   \label{eq:HF-5.2}
\end{equation}
The total number of summands, counted with multiplicity, has a finite bound \(D_U\) independent of \(u\).

\end{lemma}

\begin{proof}
By independence of \(\varphi\), we may choose a nonnegative
\(C_c^\infty\) function of integral one supported in a sufficiently
small physical ball. The support of \(f\), shifted by the vectors
\(-Lx\) with \(x\in\operatorname{supp}\varphi\), then lies in a
slightly larger logarithm chart \(\widetilde U\).

Every nonzero integrand in \eqref{eq:HF-4.1} has a unique lift \(t\in\widetilde U\). Write \(t=Bu+Lv\), so \(dt=J\,du\,dv\). By covariance,
\[
 \mu_{Bu+Lv}=\tau_{-v}\mu_{Bu}.
\]
Its atoms are \(x=y-v\), \(y\in\mu_{Bu}\), and
\[
 \varphi(x)f(t+Lx)=\varphi(y-v)f(Bu+Ly).
\]
Whenever this expression is nonzero, \(t\) lies in the larger chart,
so the \(v\)-integration extends to all of \(\mathbb R^d\). Tonelli's
theorem and \(\int\varphi(y-v)\,dv=1\) prove \eqref{eq:HF-5.2} for
nonnegative \(f\), and linearity completes the proof of the identity.

If \(\theta=Bu+Ly\in U\), then
\[
 u=B^t\theta,\qquad
 y=(L^tL)^{-1}L^t\theta.
\]
Since the chart is bounded, \(y\) lies in a fixed bounded window;
the uniform local mass bound then yields \(D_U\).
\end{proof}

To use \eqref{eq:HF-5.2} globally, we enumerate its finite fibers
and combine a finite collection of logarithm charts.

To enumerate the finite multisets in \eqref{eq:HF-5.2}, use vague
continuity of \(u\mapsto\mu_{Bu}\): restriction to a bounded window
defines a Borel relation with finite fibers. A countable hierarchy of
rational boxes enumerates these fibers measurably. Repeat each atom
according to its integer mass and pad with zero weights to obtain at
most \(D_U\) sheets.

Cover \(M\) by finitely many logarithm charts and choose a subordinate
Borel partition. Restricting \eqref{eq:HF-5.2} to its pieces, we obtain
a representation on a bounded box in \(\mathbb R^c\):
\begin{equation}
 m=\sum_{\ell=1}^{D}w_\ell(u)\,\delta_{\theta_\ell(u)}\,du,              \label{eq:HF-5.3}
\end{equation}
where \(D<\infty\), the logarithmic lifts \(\theta_\ell(u)\) are
bounded, and the weights \(w_\ell\ge0\) are measurable. Since atoms
are repeated according to their masses, the weights are \(J\) times
indicator functions.

It follows that
\[
 f\longmapsto
   \left(\sqrt{w_\ell(u)}\,f(\theta_\ell(u)\bmod\mathbb Z^n)\right)_{\ell=1}^D
\]
is an isometric embedding
\begin{equation}
 L^2(m)\hookrightarrow L^2(\text{box};\mathbb C^D).             \label{eq:HF-5.4}
\end{equation}
Multiplication by the torus character indexed by \(k\in\mathbb Z^n\) becomes multiplication on sheet \(\ell\) by \(e^{2\pi ik\cdot\theta_\ell(u)}\).

Since the support of \(m\) is \(M\), a Laurent polynomial vanishes
in \(L^2(m)\) if and only if it vanishes on \(M\). Character spans
in \(L^2(m)\) therefore have the same dimensions as their images in
the coordinate ring of the Zariski closure.

\subsubsection{Transverse analytic pieces}\label{finitely-many-transverse-real-analytic-pieces}

\begin{lemma}[Transverse analytic pieces]\label{lem:ac-transverse-strata}\label{par:finite-transverse-strata}
The set \(M\) is the closure of finitely many connected real-analytic
manifolds of dimension \(c=n-d\), each transverse to \(V\).
Moreover, \(L\Lambda\bmod\mathbb Z^n\) is dense in \(M\).
\end{lemma}
\begin{proof}

The phases \(L\Lambda\), where
\(\Lambda=\operatorname{supp}\mu\), are dense in \(M\): for
\(t_0\in M\), choose \(Lx_j\to t_0\). Vague continuity of
\(\mu_{Lx_j}=\tau_{-x_j}\mu\) and the positive atom at zero imply
that there are atoms \(y_j=\lambda_j-x_j\to0\), so
\[
 L\lambda_j=Lx_j+Ly_j\longrightarrow t_0.
\]
Proposition~\ref{prop:ac-canonical-germ}
and Lemma~\ref{lem:integer-real-sheets} give density of the smooth
\(V\)-transverse real locus in \(M\).

Let \(M^\circ\) be the set of points near which \(M\) is a
real-analytic submanifold of dimension \(c=n-d\), transverse to
\(V\), or equivalently with projection onto \(V^\perp\) of full rank.
This relatively open subset of \(M\) is dense by the preceding
paragraph. It is subanalytic and relatively compact, since \(M\) is
compact, and hence has finitely many connected components
\cite[Corollary~2.11, Lemma~3.4 and Remark~7.3]{BierstoneMilman}.
These components are the required pieces, and their closures cover
\(M\).

\end{proof}

\section{Automatic algebraicity}
\label{sec:automatic-algebraicity}\label{sec:algebraicity}

We prove that the canonical set and its Zariski closure have the same
dimension. This identifies the algebraic support used to prove
holonomicity of the flux measure and to reconstruct its branches.
We compare an upper bound for character spaces obtained from Fourier
vanishing with a lower bound obtained from a transverse tropical cone.

\subsection{Logarithmic forms and tropical cones}
\label{sec:tropical-detection}

We use tropicalization over the trivially valued field \(\C\).
For a Laurent polynomial \(F\), its initial form
\(\operatorname{in}_w(F)\) is the sum of its terms of least
\(w\)-weight. We have \(w\in\operatorname{Trop}(Y)\) precisely when
no nonzero Laurent polynomial vanishing on \(Y\) has a unique term
of least \(w\)-weight; see \cite[\S\S1--2]{BJST2007}.
An analytic arc lying in \(Y\) away from its central point determines
such a weight by its coordinate orders.

To prove \(\dim Y=c\), we first detect a transverse tropical cone
using a nonzero logarithmic form. Deligne's strictness theorem places
its class in top weight, where a nonzero iterated residue determines
integral boundary valuations with the required rank. This argument
allows irrational coefficients in the form.

\begin{lemma}[Full transverse rank of a tropical cone]
\label{lem:tropical-rank}
Let \(Y\subset(\mathbb C^*)^n\) be irreducible of complex dimension
\(r\), and let \(c\ge1\).  Let
\(\eta_1,\ldots,\eta_c\in\mathbb R^n\) be linearly
independent, and put
\[
 P_\eta(v)=(\eta_1\cdot v,\ldots,\eta_c\cdot v),\qquad
 \beta=\bigwedge_{b=1}^{c}
       \left(\sum_{j=1}^{n}\eta_{bj}\frac{dz_j}{z_j}\right).
\]
Suppose that \(\beta|_{Y_{\mathrm{reg}}}\) is not identically zero.
Then some maximal rational cone \(\sigma\) in a pure tropical fan for
\(Y\) satisfies
\[
 \dim\sigma=r,
 \qquad \operatorname{rank}
       \bigl(P_\eta|_{\operatorname{span}_{\mathbb R}\sigma}\bigr)=c.
\]
Tropicalization is taken over the trivially valued field
\(\mathbb C\), with the minimum-weight convention.
\end{lemma}

\begin{proof}
\emph{Resolution and compactification.}
Choose a smooth quasiprojective algebraic resolution
\(\varrho:U\to Y\), and then a smooth projective compactification
\(\overline U\) with simple normal crossings boundary
\(\mathcal D=\overline U\setminus U\), using projective resolution and
embedded resolution of a projective compactification
\cite[Theorems~1.0.1--1.0.2 and Lemma~2.5.2]{Wlodarczyk2005}.
Each pullback \(\varrho^*z_j\) is a regular unit on \(U\), also along
exceptional divisors over \(Y\). Its rational extension to
\(\overline U\) therefore has divisor supported on \(\mathcal D\).
The pullback, still denoted by \(\beta\), belongs to
\begin{equation}
 H^0\bigl(\overline U,
           \Omega^c_{\overline U}(\log \mathcal D)\bigr)
 \quad\hbox{and is nonzero.}                         \label{eq:tropical-1}
\end{equation}
Nonvanishing follows by restriction to the nonempty smooth open set
where the resolution is an isomorphism.

\emph{Nonvanishing in cohomology.}
Deligne's logarithmic
Hodge-to-de Rham degeneration implies that a global logarithmic
\(c\)-form on a smooth projective normal-crossings compactification
is closed and that its class in \(H^c(U,\mathbb C)\) vanishes only
when the form itself is zero
\cite[Corollaries~3.2.13(ii) and~3.2.14]{DeligneII}. Thus
\begin{equation}
 [\beta]\ne0\quad\hbox{in }H^c(U,\mathbb C).             \label{eq:tropical-2}
\end{equation}

\emph{Top weight.}
Let \(g:U\to(\mathbb C^*)^n\) be the resolution followed by the
inclusion.  The class in \eqref{eq:tropical-2} lies in the image of
\[
 g^*:H^c((\mathbb C^*)^n,\mathbb C)
       \longrightarrow H^c(U,\mathbb C).
\]
The source is pure of weight \(2c\). Algebraic pullback is a morphism
of mixed Hodge structures, strict for the weight filtration
\cite[Theorems~3.2.5(iii) and~2.3.5(iii)]{DeligneII}.
It follows that
\begin{equation}
 \operatorname{im}g^*\cap
       W_{2c-1}H^c(U,\mathbb C)=\{0\}.                  \label{eq:tropical-3}
\end{equation}
Complexifying the rational pullback map proves \eqref{eq:tropical-3}
for the real coefficients \(\eta_{bj}\) and every complex linear
combination in its image. In particular, \([\beta]\) is nonzero in
\(\operatorname{Gr}^{W}_{2c}H^c(U,\mathbb C)\).

\emph{An iterated residue.}
Write \(\mathcal D^{(p)}\) for the disjoint union of all
\(p\)-fold intersections of distinct irreducible components of
\(\mathcal D\), with the alternating orientation convention.  The weight
spectral sequence, in cohomological weight indexing, is
\begin{equation}
 E_1^{-p,s}
   =H^{s-2p}(\mathcal D^{(p)},\mathbb C)(-p)
   \ \Longrightarrow\ H^{s-p}(U,\mathbb C).             \label{eq:tropical-4}
\end{equation}
It degenerates at \(E_2\) and is identified with the logarithmic
complex by Poincar\'e residues
\cite[Equation~(3.1.5.2), Equation~(3.2.4.1), and
Corollary~3.2.13(i)]{DeligneII}. For weight \(2c\)
in \(H^c\), the incoming differential has source
\(H^{-2}(\mathcal D^{(c+1)},\mathbb C)(-c-1)=0\).  Hence
\begin{equation}\begin{aligned}
 \operatorname{Gr}^{W}_{2c}H^c(U,\mathbb C)
 &\simeq
 \ker\bigl(
 H^0(\mathcal D^{(c)},\mathbb C)(-c)
 \longrightarrow H^2(\mathcal D^{(c-1)},\mathbb C)(-c+1)
 \bigr)\\
 &\subset H^0(\mathcal D^{(c)},\mathbb C)(-c).
                                                               \label{eq:tropical-5}
\end{aligned}\end{equation}
The image of a global logarithmic \(c\)-form in \eqref{eq:tropical-5}
consists of its \(c\)-fold residues. By
\eqref{eq:tropical-2}--\eqref{eq:tropical-5}, one residue of \(\beta\)
is nonzero.

\emph{The residue formula.}
Let \(\mathcal D_{i_1},\ldots,\mathcal D_{i_c}\) be components having such a
nonempty intersection.  Put
\[
 v_a=\bigl(\operatorname{ord}_{\mathcal D_{i_a}}z_1,\ldots,
              \operatorname{ord}_{\mathcal D_{i_a}}z_n\bigr)
       \in\mathbb Z^n,\qquad 1\le a\le c.
\]
At a general point of the intersection, only these boundary components
occur. In local coordinates \(t_1,\ldots,t_r\), with
\(\mathcal D_{i_a}=\{t_a=0\}\), one has
\[
 z_j=u_j\prod_{a=1}^{c}t_a^{v_{aj}},
 \qquad u_j\text{ holomorphic and nowhere zero}.
\]
The residue, up to its orientation sign, equals
\begin{equation}
 \operatorname{Res}_{i_1,\ldots,i_c}\beta
      =\det(\eta_b\cdot v_a)_{b,a=1}^{c}\ne0.             \label{eq:tropical-6}
\end{equation}
In particular \(P_\eta(v_1),\ldots,P_\eta(v_c)\) are independent.

\emph{The tropical cone.}
The cone \(\mathscr C=\sum_a\mathbb R_{\ge0}v_a\) lies in
\(\operatorname{Trop}(Y)\).  For positive
integers \(N_1,\ldots,N_c\), take the analytic arc
\(t_a=\zeta^{N_a}\), keeping generic values of the other coordinates fixed.
Its image in \(Y\) has coordinate orders
\(\sum_aN_av_a\).  A Laurent polynomial vanishing on that arc
cannot have a unique monomial of smallest order, so this vector
belongs to \(\operatorname{Trop}(Y)\). Rescaling and taking the closure
allow, successively, positive rational and nonnegative real
coefficients. By \eqref{eq:tropical-6}, \(\mathscr C\) has dimension \(c\).

By the Bieri--Groves purity theorem~\cite[Theorem~A]{BieriGroves},
there is a finite rational fan for
\(\operatorname{Trop}(Y)\), all of whose maximal cones have
dimension \(r\); see also the beginning of
\cite[\S3]{BJST2007}. Its maximal cones intersect \(\mathscr C\) in a
finite cover of \(\mathscr C\), so one intersection has nonempty
relative interior in \(\mathscr C\). The span of that maximal cone
therefore contains \(\operatorname{span}_{\mathbb R}\mathscr C\), on which
\(P_\eta\) has rank \(c\) by \eqref{eq:tropical-6}.
\end{proof}

\begin{corollary}[Detection from a transverse analytic branch]
\label{cor:transverse-tropical-rank}
Let \(V\subset\mathbb R^n\) have dimension \(d<n\), and put
\(c=n-d\).  Let \(S_0\subset\mathbb T^n\) be a connected
real-analytic \(c\)-manifold that is \(V\)-transverse at some point,
and let \(Y\)
be its Zariski closure in the algebraic torus. Then \(Y\) is irreducible and, with
\(r=\dim_\C Y\), there is an \(r\)-dimensional rational cone
\(\sigma\subset\operatorname{Trop}(Y)\) whose span projects with
rank \(c\) onto \(V^\perp\).
In particular \(r\ge c\).
\end{corollary}

\begin{proof}
If a product of two Laurent polynomials vanishes on \(S_0\), one
factor vanishes on a nonempty open subset and hence, by the
real-analytic identity theorem, on the whole connected manifold.
The vanishing ideal of \(S_0\) is prime, so \(Y\) is irreducible.
The nonempty open set of \(V\)-transverse points of \(S_0\)
meets the regular locus of \(Y\):
any algebraic equation vanishing on this real open set vanishes on
all of \(S_0\), and hence on \(Y\). We may thus choose a
\(V\)-transverse point in \(S_0\cap Y_{\mathrm{reg}}\).

At this point, the \(c\) independent real tangent vectors of \(S_0\)
are complex-linearly independent, since angle coordinates identify them
with real vectors in \(\R^n\subset\C^n\). In particular
\(\dim_\C Y\ge c\). For a basis \(\eta_1,\ldots,\eta_c\)
of \(V^\perp\), the restriction of the logarithmic form \(\beta\)
to \(S_0\) is \((2\pi i)^c\) times the corresponding transverse
volume form. It is nonzero at the chosen point, and
Lemma~\ref{lem:tropical-rank} applies.
\end{proof}

\subsection{The upper bound from Fourier vanishing}
\label{sec:fourier-vanishing-bound}

We bound character spaces using gaps in Fourier support and a finite
description of the measure over normal coordinates. Fix a rank-\(d\) linear map \(T:\R^n\to\R^d\), put
\(c=n-d\ge1\), set \(V=(\ker T)^\perp\), and let \(B\in\R^{n\times c}\) have orthonormal
columns spanning \(V^\perp=\ker T\).  For a finite positive measure \(m\)
on \(\T^n\), write
\[
 \mathcal S=\{k\in\Z^n:\widehat m(k)\ne0\}.
\]
The spectral hypothesis is that \(\mathcal S\) is \(T\)-proper, namely
\begin{equation}
 \#\{k\in\mathcal S:\|Tk\|\le R\}<\infty\qquad(R<\infty).
 \label{eq:alg-proper}
\end{equation}

\begin{definition}[Bounded finite-sheet model]\label{def:alg-sheets}
We say that \(m\) has a \emph{bounded finite-sheet model in the normal
coordinates} if there are a positive integer \(D\), a cube
\(I=[-S/2,S/2]^c\) of positive integer side length \(S\), bounded
measurable lifts \(\theta_\ell(t)\in\R^n\) with a common bound
\(\|\theta_\ell(t)\|\le C_0\), and nonnegative measurable
weights \(w_\ell(t)\), such that for every nonnegative Borel function
\(f\) on \(\T^n\),
\begin{equation}
 \begin{split}
 \int_{\T^n} f\,dm
   &=\int_I\sum_{\ell=1}^D
       w_\ell(t)f(\theta_\ell(t)\bmod\Z^n)\,dt,\\
 B^t\theta_\ell(t)&=t\qquad\text{where }w_\ell(t)>0.
 \end{split}
 \label{eq:alg-sheets}
\end{equation}
Missing sheets have zero weight; the weights are integrable but may be
unbounded.
\end{definition}

We use two instances of this model. For a compact real-analytic
transversal with positive continuous density, the proof of
Theorem~\ref{thm:main-algebraicity} uses local inverses of the normal
projection. For the canonical flux measure, we use \eqref{eq:HF-5.3}:
enlarge its bounded
box to a cube of integer side length and pad the sheets with zero weights.
The lifts there have the form $\theta_\ell(u)=Bu+Ly$, so
$B^t\theta_\ell(u)=u$. The measure identity \eqref{eq:HF-5.3} yields
\eqref{eq:alg-sheets} for every nonnegative Borel function.

For \(R>0\), \(h\ge0\), define the finite set
\begin{equation}
 Q(R,h)=\{k\in\Z^n:\|Tk\|\le R,\ B^tk\in[0,h]^c\}
 \label{eq:alg-index-box}
\end{equation}
and the space
\[
 \mathcal H(R,h)=\operatorname{span}_{\C}
       \{\chi_k:k\in Q(R,h)\}\subset L^2(m),
 \qquad \chi_k(\theta)=e^{2\pi ik\cdot\theta}.
\]
Finiteness of \(Q(R,h)\) follows because
\(x\mapsto(Tx,B^tx)\) is a real-linear isomorphism.

Fourier vanishing makes suitably spaced translates of the character
space orthogonal. On the sheets, their images can be approximated in a
single box of Fourier modes. Comparing its dimension with the number
of orthogonal vectors gives the upper bound.

\begin{lemma}[Upper bound for character spaces]
\label{lem:alg-strip-upper}
Assume \eqref{eq:alg-proper} and \eqref{eq:alg-sheets}.  There is a
constant \(C_*\), independent of \(R,h\), and for every \(R\) a
finite constant \(A_R\), such that for all \(R>0\) and \(h\ge0\),
\begin{equation}
 \dim\mathcal H(R,h)\le C_*(h+A_R)^c.
 \label{eq:alg-rank-upper}
\end{equation}
One can take \(C_*=DS^c\), so \(C_*\) depends only on the sheet model.
\end{lemma}

\begin{proof}
\emph{Spacing and rounding.}
Let \(\eta_0=\sqrt n/2\) and \(\rho=\|T\|\eta_0\).  By
\eqref{eq:alg-proper},
\[
 H_R=\max\{\|B^tk\|_\infty:k\in\mathcal S,\
                         \|Tk\|\le2(R+\rho)\}<\infty,
\]
with zero used if the set is empty, and set
\[
 A_R=H_R+\sqrt n+1.
\]
Choose a positive integer \(j_0\)
such that
\begin{equation}
 h+H_R+2\eta_0<j_0\le h+H_R+2\eta_0+1.
 \label{eq:alg-block-spacing}
\end{equation}
For each \(j\in\Z^c\), round \(Bj_0j\) coordinatewise to an integer
vector \(k_j\), and put \(\delta_j=k_j-Bj_0j\).  Then
\begin{equation}
 \|\delta_j\|\le\eta_0,\qquad
 \|Tk_j\|\le\rho,\qquad
 \|B^tk_j-j_0j\|_\infty\le\eta_0.
 \label{eq:alg-rounding}
\end{equation}

\emph{Orthogonality.}
If \(j\ne j'\) and \(a,b\in Q(R,h)\), the difference
\(k=k_j+a-k_{j'}-b\) satisfies \(\|Tk\|\le2(R+\rho)\).
In a coordinate where \(j-j'\) is nonzero, the corresponding normal
coordinate has absolute value at least \(j_0-h-2\eta_0>H_R\).
It follows that \(\pm k\notin\mathcal S\), so the character inner
product vanishes. The subspaces
\begin{equation}
 \chi_{k_j}\mathcal H(R,h),\qquad j\in\Z^c,
 \label{eq:alg-orthogonal-blocks}
\end{equation}
are pairwise orthogonal. Character multiplication is unitary on
\(L^2(m)\), since \(m\) is positive and characters have modulus one.

\emph{Transfer to the sheets.}
Put \(N_*=\dim\mathcal H(R,h)\) and fix an orthonormal basis
\(\psi_1,\ldots,\psi_{N_*}\). The family
\(\{\chi_{k_j}\psi_i\}_{j,i}\) is orthonormal. By the sheet model,
the map
\begin{equation}
 \begin{split}
 \mathcal U:L^2(m)&\longrightarrow L^2(I;\C^D),\\
 (\mathcal U\psi)_\ell(t)&=\sqrt{w_\ell(t)}\,
                         \psi(\theta_\ell(t)\bmod\Z^n).
 \end{split}
 \label{eq:alg-isometry}
\end{equation}
is an isometric embedding. The normal-coordinate identity gives
\begin{equation}
 (\mathcal U(\chi_{k_j}\psi_i))_\ell(t)
   =e^{2\pi ij_0j\cdot t}
      e^{2\pi i\delta_j\cdot\theta_\ell(t)}
      (\mathcal U\psi_i)_\ell(t).
 \label{eq:alg-modulation}
\end{equation}

\emph{Uniform frequency localization.}
After removal of the first exponential factor, the displayed vectors
are contained in a compact subset of \(L^2(I;\C^D)\).
The errors lie in a compact cube, there are finitely many
\(i\), and boundedness of the lifts gives
\[
 \|e^{2\pi i\delta\cdot\theta_\ell}
       -e^{2\pi i\delta'\cdot\theta_\ell}\|_\infty
       \le2\pi C_0\|\delta-\delta'\|,
\]
where \(C_0\) bounds the lifts. The images of the whole error cube
for these finitely many indices form a compact family. The finite
Fourier projections on \(I\) converge uniformly in norm on this family. For each
\(\epsilon>0\), there is a radius \(W_{\mathrm{tail}}\) for which the
projected squared norms are at least \(1-\epsilon\).

\emph{Counting modes.}
Use the Fourier modes \(e^{2\pi iu\cdot t/S}\), \(u\in\Z^c\),
in each of the \(D\) channels.  The side length \(S\) and spacing
\(j_0\) are integers, so the first factor in
\eqref{eq:alg-modulation} shifts the Fourier index by exactly \(Sj_0j\).
For all \(j\in[-N,N]^c\cap\Z^c\), project onto the channel modes
in \([-Sj_0N-W_{\mathrm{tail}},Sj_0N+W_{\mathrm{tail}}]^c\cap\Z^c\).  Each corresponding orthonormal
vector has projected norm squared at least \(1-\epsilon\).  The
projection rank is at most \(D(2Sj_0N+2W_{\mathrm{tail}}+1)^c\).  Bessel's inequality
therefore gives
\[
 (1-\epsilon)N_*(2N+1)^c\le D(2Sj_0N+2W_{\mathrm{tail}}+1)^c.
\]
Letting \(N\to\infty\) and then \(\epsilon\downarrow0\), we obtain
\[
 N_*\le D(Sj_0)^c\le DS^c(h+H_R+2\eta_0+1)^c,
\]
as required.
The compact family and \(W_{\mathrm{tail}}\) depend on \(R,h\) and
the basis \(\psi_i\), which stay fixed as \(N\) tends to infinity.
\end{proof}

\subsection{A lower bound from a transverse tropical cone}

Characters with distinct images under an integer projection associated
to a rational tropical cone are linearly independent on the variety.
We first count these images for the box \(Q(R,h)\), then prove
independence using initial forms.

\begin{lemma}[Lattice quotient count]
\label{lem:alg-lattice-lower}
Let \(P\in\Z^{r\times n}\) have rank \(r\), and assume
\(P|_{V^\perp}\) is injective.  Then \(r\ge c\), and there are constants
\(\kappa_P>0\), \(R_0,h_0<\infty\), such that
\begin{equation}
 \#P(Q(R,h))\ge \kappa_P R^{r-c}h^c
 \qquad(R\ge R_0,\ h\ge h_0).
 \label{eq:alg-quotient-lower}
\end{equation}
\end{lemma}

\begin{proof}
The invertible map \(x\mapsto(Tx,B^tx)\) sends \(\Z^n\) to a full
lattice in \(\R^d\times\R^c\). Choose a bounded fundamental cell.
Translates by lattice points in \(B_R^d\times[0,h]^c\)
cover the region obtained by reducing the ball radius and moving each
face inward by the cell diameter. Dividing its volume by the cell
volume gives, for sufficiently large \(R,h\),
\begin{equation}
 \#Q(R,h)\ge a_1R^dh^c,
 \label{eq:alg-total-count}
\end{equation}
with constants independent of both parameters.

The group \(\Gamma_P=\ker P\cap\Z^n\) is a full lattice in the
\((n-r)\)-dimensional real kernel.  The assumption implies
\(\ker P\cap\ker T=\{0\}\), so \(T\) is injective on
\(\ker P\) and \(T\Gamma_P\) is a rank-\((n-r)\) lattice in
its real span.  Fix one value of \(Pk\) and one preimage \(k_0\)
in \(Q(R,h)\).  All other preimages differ by elements \(\ell\)
of \(\Gamma_P\) with \(\|T\ell\|\le2R\).  There are at most
\(a_2R^{n-r}\) such elements for \(R\ge1\), independently of \(h\).
Divide \eqref{eq:alg-total-count} by this fiber bound; the exponent is
\(d-(n-r)=r-c\).
\end{proof}

\begin{lemma}[Tropical separation]\label{lem:tropical-separation}
Let \(Y\subset(\C^*)^n\) be irreducible, let \(\sigma\) be a
rational cone contained in \(\operatorname{Trop}(Y)\), and let the
rows of an integer matrix \(P\) be a basis of
\(\operatorname{span}_{\R}\sigma\). For every finite
\(K\subset\Z^n\),
\[
 \dim_\C\operatorname{span}\{z^k|_Y:k\in K\}\ge\#P(K).
\]
\end{lemma}

\begin{proof}
Choose \(w\) in the relative interior of \(\sigma\) whose
coefficients in the row basis of \(P\) are linearly independent over
\(\Q\). Such choices are dense, and for integer indices,
\begin{equation}
 w\cdot k=w\cdot k'\quad\Longleftrightarrow\quad Pk=Pk'.
 \label{eq:alg-generic-weight}
\end{equation}
Choose one character for each distinct value of \(Pk\), \(k\in K\).
A nonzero relation among them would have pairwise distinct weights and
hence a unique term of least weight, contradicting
\(w\in\operatorname{Trop}(Y)\) by the initial-form criterion
\cite[\S\S1--2]{BJST2007}.
\end{proof}

\begin{proposition}[Character count from a transverse cone]
\label{prop:alg-character-lower}
Let \(Y\subset(\C^*)^n\) be irreducible of dimension \(r\), and
suppose an \(r\)-dimensional rational cone
\(\sigma\subset\operatorname{Trop}(Y)\) projects with rank
\(c\) onto \(V^\perp\). Then, for some
constant \(\kappa_P>0\) depending only on \(\sigma\), \(T\), and \(B\),
and all sufficiently large \(R,h\),
\[
 \dim_\C\operatorname{span}\{z^k|_Y:k\in Q(R,h)\}
       \ge\kappa_P R^{r-c}h^c.
\]
\end{proposition}

\begin{proof}
Take the rows of \(P\) to be an integral basis of
\(\operatorname{span}_{\R}\sigma\cap\Z^n\). The subscript on the
constant refers to this projection. The matrix has rank
\(r\), and the transverse rank condition says exactly that
\(P|_{V^\perp}\) is injective. Apply
Lemma~\ref{lem:tropical-separation} to \(K=Q(R,h)\), followed by
Lemma~\ref{lem:alg-lattice-lower}.
\end{proof}

\subsection{The algebraicity theorem}

\begin{definition}[Transverse analytic pieces]\label{def:alg-strata}
A compact set \(M\subset\T^n\) has \emph{transverse analytic pieces}
if there are finitely many
connected real-analytic \(c\)-dimensional embedded submanifolds
\(S_\alpha\subset M\) of \(\T^n\) such that
\begin{equation}
 M=\overline{\bigcup_\alpha S_\alpha},
 \label{eq:alg-strata}
\end{equation}
and each \(S_\alpha\) is \(V\)-transverse at at least one point.
The pieces may overlap and need not be closed.
\end{definition}
For a compact embedded real-analytic \(V\)-transverse manifold, its
connected components are such pieces.

On each piece, Corollary~\ref{cor:transverse-tropical-rank} supplies
a transverse tropical cone, and the two estimates bound the same
character space. If the Zariski closure has larger dimension, they
contradict one another when we fix the strip width and let its height
tend to infinity.

\begin{theorem}[Automatic algebraicity]
\label{thm:automatic-algebraicity}
Let \(T:\R^n\to\R^d\) have rank \(d<n\), put \(c=n-d\ge1\), and
let \(B\) be as in Section~\ref{sec:fourier-vanishing-bound}, with
\(V=(\ker T)^\perp\). Suppose \(m\) is a
finite positive measure on \(\T^n\) with support \(M\), satisfying
\begin{enumerate}
\item[(i)] the \(T\)-properness condition \eqref{eq:alg-proper};
\item[(ii)] \(m\) has a bounded finite-sheet model with respect to \(B\)
(Definition~\ref{def:alg-sheets});
\item[(iii)] \(M\) has transverse analytic pieces
(Definition~\ref{def:alg-strata}).
\end{enumerate}
Then the Zariski closure in the algebraic torus
\[
 Y=\overline M^{\mathrm{Zar}}\subset(\C^*)^n
\]
is pure of complex dimension \(c\).

\end{theorem}

\begin{proof}
Let \(Y_\alpha\) be the Zariski closure of a connected
piece \(S_\alpha\), and put \(r=\dim_\C Y_\alpha\).
Corollary~\ref{cor:transverse-tropical-rank} shows that
\(Y_\alpha\) is irreducible, \(r\ge c\), and an \(r\)-dimensional
rational cone \(\sigma\subset\operatorname{Trop}(Y_\alpha)\) projects
with rank \(c\) onto \(V^\perp\).

The character-span dimension on \(Y_\alpha\) is at most
\(\dim\mathcal H(R,h)\). By continuity, a trigonometric polynomial
vanishes in \(L^2(m)\) exactly when it vanishes on the support \(M\);
it then vanishes on \(S_\alpha\) and its Zariski closure. Therefore
Proposition~\ref{prop:alg-character-lower} gives
\[
 \dim\mathcal H(R,h)\ge\kappa_P R^{r-c}h^c
\]
for sufficiently large \(R,h\). On the other hand,
Lemma~\ref{lem:alg-strip-upper} gives
\[
 \dim\mathcal H(R,h)\le C_*(h+A_R)^c.
\]
If \(r>c\), fix \(R\) large enough that
\(\kappa_P R^{r-c}>C_*\). As \(h\to\infty\), the constant \(A_R\)
stays fixed and the two bounds contradict each other. Hence \(r=c\).

The argument applies to every \(Y_\alpha\), and the finite union of
these Zariski closures is exactly \(Y\), by \eqref{eq:alg-strata}.
Its irreducible components all have dimension \(c\).
\end{proof}

\begin{remark}[The hypotheses of the algebraicity theorem]
\label{rem:algebraicity-hypotheses}
Fourier properness alone does not imply algebraicity in the stated
dimension. Haar measure on $\T^n$ has Fourier support $\{0\}$,
which is $T$-proper, but its support has Zariski closure $(\C^*)^n$.
It satisfies neither geometric hypothesis: finitely many graphs over
the normal coordinates have zero ambient Lebesgue measure, and finitely
many closures of embedded $c$-manifolds, each nowhere dense for $c<n$,
cannot cover $\T^n$.
\end{remark}

\begin{proof}[Proof of Theorem~\ref{thm:main-algebraicity}]
If $c=0$, then $M$ is finite and the conclusion follows. Assume $c>0$,
and choose an orthonormal basis matrix $B$ for $V^\perp$.
By transversality, normal projection restricts to a local
diffeomorphism on a real logarithmic lift. Compactness gives finitely
many bounded local inverses $\theta_\ell:\Omega_\ell\to\R^n$ covering $M$, with
$B^t\theta_\ell(t)=t$ and each image $\theta_\ell(\Omega_\ell)$ of
diameter less than $1$, so the charts are injective modulo $\Z^n$.
Choose a Borel partition
$M=\bigsqcup_\ell E_\ell$ subordinate to these charts. Set
$p_\ell(t)=\theta_\ell(t)\bmod\Z^n$ and, on $\Omega_\ell$, put
\[
 w_\ell(t)=\mathbf1_{E_\ell}(p_\ell(t))a(p_\ell(t))
 \sqrt{\det\bigl(D\theta_\ell(t)^tD\theta_\ell(t)\bigr)}.
\]
All chart domains lie in a cube $[-S/2,S/2]^c$ with $S$ a positive
integer. Extend the weights and lifts by zero outside their domains.
Change of variables proves \eqref{eq:alg-sheets} for every nonnegative
Borel function, with bounded lifts. Since $a$ is positive,
$\supp m=M$. The finitely many connected components of $M$ satisfy
\eqref{eq:alg-strata}, and the assumed Fourier support condition is
\eqref{eq:alg-proper}, so Theorem~\ref{thm:automatic-algebraicity} applies.
\end{proof}

\begin{corollary}[Semialgebraicity of the analytic set]
\label{cor:selected-semialgebraic}
Under the hypotheses of Theorem~\ref{thm:automatic-algebraicity},
assume additionally that \(M\) is a real-analytic subset of \(\T^n\).
Then \(M\) is semialgebraic in \((S^1)^n\).  More precisely, it is
the closure of a union of connected components of
\(Y_{\mathrm{reg}}\cap\T^n\), where
\(Y=\overline M^{\mathrm{Zar}}\), and \(M\cap Y_{\mathrm{reg}}\) is dense
in \(M\).
\end{corollary}

\begin{proof}
Reciprocal conjugation fixes \(M\), so it preserves \(Y\).
Its fixed locus on the smooth complex \(c\)-fold
\(Y_{\mathrm{reg}}\) is a smooth real \(c\)-manifold.
In particular, \(Y_{\mathrm{reg}}\cap\T^n\) is a
semialgebraic smooth real manifold of that dimension.

The intersection \(M\cap Y_{\mathrm{reg}}\) is dense in
\(M\). Each connected analytic piece in
\eqref{eq:alg-strata} is Zariski dense in its irreducible
\(c\)-dimensional component; its intersection with the singular
locus or another component is a proper real-analytic subset and has
empty interior. Their closures cover \(M\), proving density.

Fix \(p\in M\cap Y_{\mathrm{reg}}\), and take a connected small
chart of the real manifold \(Y_{\mathrm{reg}}\cap\T^n\) around
\(p\), where \(M\) has real-analytic defining functions.
Density supplies a nonempty open \(c\)-dimensional patch of \(M\)
in this chart. Each defining function vanishes on the patch, hence on
the connected chart. Thus \(M\cap Y_{\mathrm{reg}}\) is relatively
open and, by compactness of \(M\), relatively closed. It is a
union of connected components of \(Y_{\mathrm{reg}}\cap\T^n\).
There are finitely many such components, all semialgebraic. Their
union has semialgebraic closure, which equals the original set by density.
\end{proof}

\begin{corollary}[Algebraicity of the canonical set]
\label{cor:all-fq-algebraicity}
Let $\mu$ be an integer-mass FQ measure on $\R^d$,
$n=\operatorname{rank}\Gamma$, $c=n-d$, and let $M$ be its canonical
set. Its Zariski closure in the algebraic torus
\[
 Y=\overline M^{\mathrm{Zar}}\subset(\C^*)^n
\]
is reduced and pure of dimension $c$. Moreover $M$ is semialgebraic.
Near every $p\in M$,
\[
 (M^\C,p)\subseteq(Y^{\mathrm{an}},p),
\]
and equality holds whenever $Y$ is smooth at $p$.
\end{corollary}

\begin{proof}
If $c=0$, then $M$ is compact and zero-dimensional by
Theorem~\ref{thm:all-fq-toral-cycles}. It is finite and
semialgebraic, and so is its reduced Zariski closure.

For $c>0$, Lemma~\ref{lem:ac-transverse-strata} supplies the finitely
many connected transverse analytic pieces whose closures cover $M$.
By Proposition~\ref{prop:ac-flux-support}, the flux measure has support
$M$ and satisfies~\eqref{eq:HF-4.2}.
Its bounded normal-sheet model is \eqref{eq:HF-5.3}, as explained after
Definition~\ref{def:alg-sheets}.
Since $L^t$ is injective on $\Z^n$, local finiteness of the physical
spectrum gives
\[
 \#\{k\in\Z^n:\widehat m(k)\ne0,\ |L^tk|\le R\}<\infty
 \qquad(R<\infty).
\]
Theorem~\ref{thm:automatic-algebraicity} now proves that the reduced
Zariski closure is pure of dimension $c$.
Theorem~\ref{thm:all-fq-toral-cycles}(c) gives real analyticity of $M$
and purity of its complex germ. Together with
Corollary~\ref{cor:selected-semialgebraic}, this gives semialgebraicity.

In every dimension, Proposition~\ref{prop:ac-canonical-germ} gives the
finite local projection and real-fiber property required by
Lemma~\ref{lem:integer-real-sheets}. A Laurent polynomial vanishing on
$M$ therefore vanishes on every local branch of $M^\C$. This proves
the germ inclusion. At a smooth point of $Y$, its germ is irreducible
of dimension $c$. A nonempty pure $c$-dimensional analytic subgerm is the
whole germ, proving equality there.
\end{proof}

\section{Algebraic holonomicity of the canonical flux}
\label{sec:holonomic-flux}

We prove that the flux measure is holonomic by decomposing its support
into Nash graphs and expressing each graph measure as a derivative of
a bounded semialgebraic function. Oaku's results apply to these functions.
By Lemma~\ref{lem:fourier-holonomic} and
Theorem~\ref{thm:holonomic-height}, this yields the spectral bound of
Corollary~\ref{cor:universal-linear-height}.

We use algebraic holonomicity in the sense of
Section~\ref{sec:preliminaries}.

We recall the required distribution results. Put
$\mathrm H(t)=1$ for $t>0$ and $\mathrm H(t)=0$ for $t\le0$.
Finite products of functions $\mathrm H(P)$, with $P$ a real polynomial,
are algebraically holonomic
\cite[Propositions~1--2 and Lemma~2]{Oaku2013}. Integration along
coordinate fibers preserves holonomicity when the projection is proper
on the support of the distribution
\cite[Theorem~1]{Oaku2013}, with the properness hypothesis as in
\cite{OakuErrata}. Under a closed smooth algebraic embedding,
extension by zero preserves and reflects holonomicity for
tempered distributions \cite[Lemma~3.1(i)]{GSSauthor}.
We use the implication from the ambient distribution to its restriction.
For $\iota:X\hookrightarrow Z$ with $\iota_*u$ holonomic, the cyclic
module $\mathscr D_Z\cdot\iota_*u$ is supported on $X$. Kashiwara's
equivalence identifies it with $\iota_+N$, where the holonomic module
$N$ contains $\mathscr D_X\cdot u$, proving the implication.
Holonomicity is preserved by finite sums and subquotients
\cite[\S3]{GSSauthor}.
It is also preserved by algebraic differential operators, since the
resulting cyclic module is a submodule of the original one.

\subsection{Semialgebraic functions and graph measures}
We first integrate the indicator of a bounded subgraph to obtain a
function, then differentiate in the normal variables to obtain graph measures.

\begin{lemma}[Bounded semialgebraic functions]
\label{lem:bounded-semialgebraic-holonomic}
Every bounded complex-valued semialgebraic function on \(\R^N\),
regarded as a locally integrable distribution, is algebraically
holonomic.
\end{lemma}

\begin{proof}
For a semialgebraic set \(S\), remove identically zero polynomials
from a finite Boolean description by polynomial signs. The remaining
zero hypersurfaces partition their complement into strict sign cells.
The indicator of \(S\) agrees almost everywhere with a finite sum
of disjoint cell indicators, each a product of polynomial Heaviside
functions. The omitted hypersurfaces have Lebesgue measure zero, so
the equality holds as distributions. The Heaviside theorem and closure
under finite sums prove holonomicity of \(1_S\).

Splitting a bounded real semialgebraic function into positive and
negative parts reduces to \(0\le g\le C\). The bounded subgraph
\[
 S_g=\{(x,t)\in\R^N\times\R:0<t<g(x)\}
\]
is semialgebraic, so \(1_{S_g}\) is holonomic, and
\[
 g(x)=\int_\R 1_{S_g}(x,t)\,dt
\]
as functions and distributions. The support lies in
\(\R^N\times[0,C]\), so projection onto \(x\) is proper there.
Proper integration proves holonomicity of \(g\).
Taking real and imaginary parts treats the complex-valued case.
\end{proof}

\begin{lemma}[Graph measures]
\label{lem:semialgebraic-graph-measure}
Let \(D\subset\R^c\) be semialgebraic, let
\(\phi:D\to\R^r\) be semialgebraic, and let
\(h:D\to\C\) be bounded semialgebraic. Define
\[
 \nu(\psi)=\int_D h(x)\psi(x,\phi(x))\,dx
\]
for compactly supported test functions. Then \(\nu\) is a locally
finite measure and is algebraically holonomic on \(\R^{c+r}\).
\end{lemma}

\begin{proof}
For every compact $K\subset\R^{c+r}$,
$|\nu|(K)\le\|h\|_\infty\operatorname{Leb}(\operatorname{pr}K)$,
where $\operatorname{pr}$ is projection onto the first $c$ coordinates.
Thus the measure is locally finite. Define
\[
 G_{\phi,h}(x,y)=1_D(x)h(x)\prod_{j=1}^r \mathrm H(y_j-\phi_j(x)),
\]
with value zero outside \(D\). This bounded semialgebraic function is
holonomic by the preceding lemma, and distributional differentiation gives
\begin{equation}
 \nu=\partial_{y_1}\cdots\partial_{y_r}G_{\phi,h}.
 \label{eq:graph-measure-derivative}
\end{equation}
To verify the identity, test against a compactly supported smooth
function and integrate successively from \(\phi_j(x)\) to \(+\infty\).
The integration signs cancel those of distributional differentiation,
leaving the formula for \(\nu\). Differentiation occurs only in
\(y\), with zero contribution on \(\partial D\).
Holonomicity follows from preservation by differential operators.
The cases \(r=0\) and \(c=0\) give the preceding lemma and a Dirac
mass, respectively.
\end{proof}

\begin{lemma}[Bounded-slope graph partition]
\label{lem:bounded-slope-partition}
Let $E\subset\R^N$ be a compact semialgebraic set of pure dimension
$c$. There is a semialgebraic subset $E'\subset E$, with
$\dim E'<c$, such that $E\setminus E'$ is a disjoint union of finitely
many Nash graphs over coordinate $c$-planes. Their graph derivatives
have norm at most $\binom Nc^{1/2}$. On each graph, the coefficient
of the pullback of a polynomial $c$-form is bounded and semialgebraic.
\end{lemma}

\begin{proof}
If \(c=0\), each point of the finite set is a graph over \(\R^0\).
For \(c>0\), choose a finite Nash stratification of the top-dimensional part.
At a point of one such stratum choose an \(N\times c\) matrix
\(Q\) with orthonormal columns spanning its tangent plane.
Cauchy--Binet gives
\[
 \sum_{|I|=c}\det(Q_I)^2=1.
\]
For some coordinate projection \(\pi_I\), the tangent determinant
has absolute value at least \(\kappa=\binom{N}{c}^{-1/2}\).
Partition by a choice of \(I\), breaking ties in a fixed order, and
take a compatible Nash cell decomposition. On each retained
\(c\)-dimensional piece, \(\pi_I\) is a local diffeomorphism.
Its singular values are at most one, and their product is at least
\(\kappa\), so the smallest is at least \(\kappa\) and the inverse
tangent map has norm at most \(\kappa^{-1}\).

The fibers are discrete semialgebraic sets, hence finite. A cell
decomposition over the projected coordinates gives finitely many
graphs, whose functions are Nash by the implicit-function theorem.
Discard their lower-dimensional boundaries. The derivatives satisfy
the stated bound, while compactness bounds the domains and graph values.
The pullback of a polynomial \(c\)-form has a semialgebraic coefficient,
bounded by the bounds on the form coefficients and graph derivatives.
\end{proof}

\begin{proposition}[Polynomial-form measures on semialgebraic strata]
\label{prop:semialgebraic-flux-holonomic}
Let $E\subset\R^N$ be compact semialgebraic of pure dimension $c$,
and let $S_1,\ldots,S_q$ be disjoint top-dimensional Nash strata of
$E$. Let $\omega$ be a polynomial $c$-form, and let $a_j$ be bounded
locally constant semialgebraic weights on $S_j$. The measure
\[
 \psi\longmapsto\sum_{j=1}^q\int_{S_j}a_j\psi\,|\omega|,
\]
with zero mass on the remaining strata, is algebraically holonomic.
The same holds with $\omega$ in place of $|\omega|$ on oriented strata.
\end{proposition}

\begin{proof}
Apply Lemma~\ref{lem:bounded-slope-partition}. On each graph,
Lemma~\ref{lem:semialgebraic-graph-measure} applies to the bounded
semialgebraic coefficient obtained from the form, the orientation sign
or absolute value, and the stated weight. Summing proves holonomicity.
The omitted set has \(c\)-dimensional measure zero and receives zero
mass. When \(c=0\), we have a finite sum of Dirac measures, also
covered by Lemma~\ref{lem:semialgebraic-graph-measure}.
\end{proof}

\subsection{The flux density}
We identify the flux density on the smooth transverse part of the
canonical set $M$. To apply the preceding proposition, we realize
the logarithmic form as the restriction of a polynomial form.

Put $V=\operatorname{im}L$ and $c=n-d$, and choose an oriented
orthonormal basis $b_1,\ldots,b_c$ of $V^\perp$. The form is
\begin{equation}
 J=\sqrt{\det(L^tL)},\qquad
 \beta_L=J\bigwedge_{j=1}^{c}(b_j\cdot d\theta).
 \label{eq:transverse-flux-form}
\end{equation}
For $c=0$ the empty wedge is $1$.

Consider the closed smooth real algebraic variety
\[
 \mathcal T=\{(\mathsf u_1,\mathsf v_1,\ldots,\mathsf u_n,\mathsf v_n):
                   \mathsf u_j^2+\mathsf v_j^2=1\}\subset\R^{2n},
\]
and let $\iota:\mathcal T\hookrightarrow\R^{2n}$ be the inclusion.
Its real points are the compact torus, and its complexification is
$(\C^*)^n$ under
\[
 z_j=\mathsf u_j+i\mathsf v_j,\qquad
 \mathsf u_j=\frac{z_j+z_j^{-1}}2,\qquad
 \mathsf v_j=\frac{z_j-z_j^{-1}}{2i}.
\]
Since $d\theta_j=(\mathsf u_j\,d\mathsf v_j-
\mathsf v_j\,d\mathsf u_j)/(2\pi)$, the form $\beta_L$ is the
restriction of the polynomial form
\[
 J\bigwedge_{a=1}^c\left(
  \sum_{j=1}^n b_{aj}
  \frac{\mathsf u_j\,d\mathsf v_j-\mathsf v_j\,d\mathsf u_j}{2\pi}
 \right).
\]

\begin{lemma}[Density of the flux measure]\label{lem:flux-density}
For an integer-mass FQ measure, put
\[
 \mathcal S_{\mathrm{tr}}=M\cap(Y_{\mathrm{reg}}\setminus\mathrm{Crit}),
 \qquad
 \mathrm{Crit}=\{p\in Y_{\mathrm{reg}}:
 T^{\log}_pY\cap V_\C\ne0\}.
\]
Then $\mathcal S_{\mathrm{tr}}$ is a semialgebraic $c$-dimensional
submanifold, dense in $M$, and
\[
 \dim(M\setminus\mathcal S_{\mathrm{tr}})<c,
 \qquad m(M\setminus\mathcal S_{\mathrm{tr}})=0.
\]
The flux measure $m$ satisfies
\[
 dm=\varkappa\,|\beta_L|
\]
on $\mathcal S_{\mathrm{tr}}$, where $\varkappa(p)=\mu_p(\{0\})$ is
the coefficient of $\mathfrak M^\C$ on the branch through $p$.
The function $\varkappa$ is positive integer-valued, bounded, locally
constant, and semialgebraic, also on a finite compatible semialgebraic
stratification.
For $c=0$ the formula is $m=J\sum_{p\in M}\varkappa(p)\delta_p$,
where $\varkappa(p)=\mu_p(\{0\})$.
\end{lemma}

\begin{proof}
Suppose $c>0$. A real-analytic
stratum in the totally real torus has a complexification of the same
dimension. By the identity theorem, every complex analytic set
containing the stratum contains its complexification.

The set $\mathcal S_{\mathrm{tr}}$ is semialgebraic, since failure of
transversality is defined by Jacobian minors in torus coordinates.
By Corollary~\ref{cor:selected-semialgebraic}, $M\cap Y_{\mathrm{reg}}$
is dense in $M$, as is the transverse real locus by
Lemma~\ref{lem:ac-transverse-strata}. Both are relatively open in $M$,
so $\mathcal S_{\mathrm{tr}}$ is dense in $M$.
By Corollary~\ref{cor:all-fq-algebraicity}, its canonical complex germ
is the smooth germ of $Y$, with real locus a $c$-dimensional
submanifold. Every component of $Y$ contains such a point. Its
singular and critical loci therefore have dimension at most $c-1$;
the preceding dimension principle gives the same bound on their real loci.

At $p\in\mathcal S_{\mathrm{tr}}$, the germ equality in
Corollary~\ref{cor:all-fq-algebraicity} identifies the canonical germ
with one smooth branch of $Y$. Its coefficient is locally constant by
Theorem~\ref{thm:all-fq-toral-cycles}(c) and
Lemma~\ref{lem:integer-incidence-coefficients}. Since the projection
along $V$ has local intersection degree one,
Theorem~\ref{thm:all-fq-toral-cycles}(e) identifies the coefficient as
$\varkappa(p)=\mu_p(\{0\})$. It is bounded by the uniform local mass
bound $K$ of Lemma~\ref{lem:ac-translation-bound}. The semialgebraic
set $\mathcal S_{\mathrm{tr}}$ has finitely many connected components,
on each of which the coefficient is constant. Its graph is a finite
union of semialgebraic sets $S\times\{a\}$,
which proves semialgebraicity.

Put $B=[b_1\ \cdots\ b_c]$. On a normal chart $\theta(u)$ with
$B^t\theta(u)=u$, the pullback of
$\beta_L$ is $J\,du_1\wedge\cdots\wedge du_c$.
The normal-sheet formula~\eqref{eq:HF-5.2}, counting the local atom
$\varkappa$ times, gives $dm=\varkappa|\beta_L|$.
Normal projection is Lipschitz on bounded logarithmic charts, so the
exceptional set, of real dimension less than $c$, has null image.
The normal-sheet formula assigns it zero mass. Take a finite
compatible Nash refinement for
Proposition~\ref{prop:semialgebraic-flux-holonomic}.

For $c=0$, normal disintegration gives the stated finite sum with
the factor $J$.
\end{proof}

\begin{theorem}[Holonomicity of the flux measure]
\label{thm:holonomic-canonical-flux}
\label{thm:holonomic-flux}
The flux measure $m$ of an integer-mass FQ measure is an
algebraically holonomic distribution on the compact real torus,
with complexification $(\C^*)^n$.
\end{theorem}

\begin{proof}
For $c=0$, the assertion follows from holonomicity of Dirac masses.
For $c>0$, Lemma~\ref{lem:flux-density} expresses the measure using the
polynomial form $\beta_L$ and bounded locally constant semialgebraic
coefficients $\varkappa$. By
Proposition~\ref{prop:semialgebraic-flux-holonomic}, the extension
$\iota_*m$ is holonomic. Compact support implies temperedness.
Holonomicity of $\iota_*m$ therefore implies that of $m$ by the
closed-embedding equivalence \cite[Lemma~3.1(i)]{GSSauthor}, on
$\mathcal T_\C\simeq(\C^*)^n$.
\end{proof}

\begin{remark}\label{rem:flux-euler}
In these coordinates $z_j\partial_{z_j}=(2\pi i)^{-1}
\partial_{\theta_j}$. These invariant derivatives and multiplication
by $z_j^{-1}$ give the Mellin operators on Fourier coefficients,
as made explicit in Lemma~\ref{lem:fourier-holonomic} below.
Transposition with respect to densities identifies the right
distribution modules of \cite[\S3]{GSSauthor} with the left modules of
Section~\ref{sec:preliminaries}; holonomicity is unchanged.
\end{remark}

\section{A support-height theorem for holonomic sequences}
\label{sec:holonomic-height}

We prove the linear height bound of Theorem~\ref{thm:main-support}
for holonomic sequences whose support projects properly. Applied to
the Fourier coefficients of the flux measure, it gives the spectral
bound needed for reconstruction.

\subsection{Difference modules and affine restriction}
We use the difference algebra $\mathscr D_n$ and the Mellin identification
of Section~\ref{sec:preliminaries}. Localization will treat the cyclic
module $\mathscr D_nf$ and its $\C[s]$-torsion.

A lattice hyperplane is the nonempty set of integer points satisfying
one rational affine equation. Sabbah's localization theorem yields
first-order recurrences away from finitely many such hyperplanes.
We follow lattice paths toward the kernel of the projection. A path
either reaches a bounded projected region, where properness applies,
or approaches an exceptional hyperplane, where restriction reduces
the argument to induction on the dimension.

\begin{theorem}[Support height; Theorem~\ref{thm:main-support}]
\label{thm:holonomic-height}
Let \(f:\Z^n\to\C\) be algebraically holonomic and let
\(T:\R^n\to\R^d\) be a real linear map.  Suppose that
\begin{equation}
 \#\{k\in\Z^n:f(k)\ne0,\ \|Tk\|\le R\}<\infty
 \qquad(R<\infty).
 \label{eq:height-properness}
\end{equation}
Then there exists \(C<\infty\) such that
\begin{equation}
 \|k\|\le C(1+\|Tk\|)
 \qquad\text{for every }k\in\operatorname{supp}f.
 \label{eq:height-conclusion}
\end{equation}
\end{theorem}

\begin{remark}\label{rem:height-scope}
The theorem allows noninjective maps $T$ on $\Z^n$ and irregular
holonomic modules. In dimension $n=1$, holonomicity is unnecessary:
if $T=0$, the support is finite by properness;
otherwise $\|Tk\|=\|T(1)\||k|$.
In higher dimension, the indicator of
$\{(j,j^2):j\in\Z\}$ has proper support for $T(k_1,k_2)=k_1$, but
$\|(j,j^2)\|\ge j^2$ contradicts any bound $C(1+|j|)$.
Thus holonomicity is essential, and this indicator is not holonomic.
For the corresponding power series with nonnegative indices, see
\cite[Example~4.6]{BellChen}.
\end{remark}

We first apply Sabbah's restriction theorem to a rational affine
hyperplane.

\begin{lemma}[Rational affine slices]
\label{lem:holonomic-slice}
Let \(H\subset\R^n\) be a rational affine hyperplane with
\(H\cap\Z^n\ne\varnothing\), and suppose that \(f\) is algebraically
holonomic. Its restriction to \(H\cap\Z^n\) is algebraically
holonomic in \(n-1\) variables. If the hyperplane is $\ell(k)=a$,
with $\ell$ primitive integral and $a\in\Z$, its lattice points are
$k_0+\iota\Z^{n-1}$, where $\ell(k_0)=a$ and the columns of
$\iota$ form a basis of $\ker\ell\cap\Z^n$.
This also holds for the
restriction of every sequence in \(\mathscr D_n f\).
\end{lemma}

\begin{proof}
Write the hyperplane as \(\ell(k)=a\), where \(\ell\) is primitive
integral and \(a\in\Z\), and complete \(\ell\) to unimodular integer
coordinates. The resulting monomial automorphism of the torus
preserves holonomicity. By \cite[Lemma~3.4]{Sabbah1992}, the
cohomology of the derived restriction of an algebraic holonomic
Mellin module to \(\ell(s)=a\) is holonomic in the remaining
variables. Writing $\mathcal N=\mathscr D_nf$, this restriction is
represented, in degrees $-1$ and $0$, by
\[
 [\mathcal N\xrightarrow{\ell(s)-a}\mathcal N].
\]
The degree-zero cohomology is the ordinary quotient with the action
of tangential polynomial coordinates and tangential integer shifts.
Normal shifts generally do not preserve \((\ell(s)-a)\mathcal N\);
the \(\mathscr D_n\)-action is therefore restricted to this tangential algebra.

Evaluation on the hyperplane factors through
\[
 \mathscr D_n f\big/(\ell(s)-a)\mathscr D_n f
\]
and commutes with tangential polynomial coordinates and integer shifts.
The restricted sequence generates a subquotient of this holonomic
module. Each element of \(\mathscr D_n f\) also generates a holonomic
submodule, so the argument applies to every such sequence.
\end{proof}

\subsection{Propagation away from lattice hyperplanes}
Localization yields finitely many identities relating values at
neighboring lattice points. They imply that, away from the exceptional
hyperplanes, a nonzero vector remains nonzero along a lattice path.

\begin{proposition}[Propagation]\label{prop:height-propagation}
Let $f$ be algebraically holonomic. There are a multiplicative set
$\mathcal S\subset\C[s]$, an integer $r\ge0$,
sequences $g_1,\ldots,g_r\in\mathscr D_nf$, denominators
$b_i\in\mathcal S$, a finite union $\mathcal H$ of lattice
hyperplanes, and matrices $\mathsf A_j,\mathsf B_j\in
\operatorname{Mat}_r(\C[s]_{\mathcal S})$ with the following properties.
For $k\notin\mathcal H$, all $b_i(k)$ are nonzero, and the vector
$F(k)=(g_i(k)/b_i(k))_{i=1}^r$ is nonzero whenever $f(k)\ne0$.
Here a coordinate edge is $\{k,k+e_j\}$.
Along every coordinate edge whose endpoints avoid $\mathcal H$,
the matrices are defined and
\begin{equation}
 F(k+e_j)=\mathsf A_j(k)F(k),\qquad
 F(k)=\mathsf B_j(k+e_j)F(k+e_j).
 \label{eq:height-propagation}
\end{equation}
They satisfy $\mathsf B_j(s+e_j)\mathsf A_j(s)=I$.
For a linear map $T:\R^n\to\R^d$, if $\operatorname{supp}f$ is
$T$-proper, then so is
$\bigcup_i\operatorname{supp}g_i$. If $r=0$, the support of $f$
is contained in $\mathcal H$.
\end{proposition}

\begin{proof}
We localize, clear the resulting coordinate identities, and then
evaluate them on the lattice.
\emph{Localization.}
Put \(\mathcal N=\mathscr D_n f\).  By
\cite[Theorem~3.1]{Sabbah1992}, there is a multiplicative set
\(\mathcal S\subset\C[s]\), generated by integer translates of
finitely many affine factors
\begin{equation}
 \ell(s)-a,\qquad \ell\text{ primitive integral},\quad a\in\C,
 \label{eq:height-affine-factors}
\end{equation}
such that
\[
 \mathcal N_{\mathcal S}
 =\C[s]_{\mathcal S}\otimes_{\C[s]}\mathcal N
\]
is finite free over \(\C[s]_{\mathcal S}\).
Enlarge \(\mathcal S\) by all integer translates of these factors.
Freeness is preserved, and every \(\tau_j^{\pm1}\) acts on the localization.
The torsion of \(\mathcal N\) over \(\C[s]\) vanishes there by
\cite[Lemma~3.2]{Sabbah1992}.

Choose a basis \(v_1,\ldots,v_r\) of \(\mathcal N_{\mathcal S}\)
and write
\[
 v_i=\frac{g_i}{b_i},\qquad g_i\in\mathcal N,\quad b_i\in\mathcal S.
\]
Each \(g_i\) is a finite sum of polynomial multiples of translates
of \(f\). If $\operatorname{supp}f$ is $T$-proper, each translate
is $T$-proper, as is the finite union containing the supports of the $g_i$.

\emph{Clearing the identities.}
Express \(f\), \(\tau_jv_i\), and \(\tau_j^{-1}v_i\) in this
basis. With $v=(v_i)^t$, write
\[
 f=\sum_i c_i(s)v_i,\qquad
 \tau_jv=\mathsf A_j(s)v,\qquad
 \tau_j^{-1}v=\mathsf B_j(s)v.
\]
The coefficients belong to $\C[s]_{\mathcal S}$. Applying $\tau_j$
to the last identity gives
$v=\mathsf B_j(s+e_j)\mathsf A_j(s)v$, so the matrix product is $I$
by independence of the basis. Clear the finitely many denominators,
including \(b_i\) and their required shifts. For each resulting
identity, the difference \(\omega\in\mathcal N\) of its sides vanishes
in \(\mathcal N_{\mathcal S}\). Hence some \(h_\omega\in\mathcal S\)
satisfies \(h_\omega\omega=0\) in \(\mathcal N\). Multiplication by
\(h_\omega\) gives finitely many identities
\[
 h_\omega\omega=0\qquad\text{in }\mathcal N.
\]
They hold in the original sequence module and can be evaluated at every
$k\in\Z^n$.

\emph{Evaluation and propagation.}
Let \(\mathcal H\) be the union of the integer zero sets of all
affine factors in these multipliers and denominators, together with
the finitely many shifts needed at neighboring vertices. It is a finite
union of lattice hyperplanes
\begin{equation}
 \ell_j(k)=a_j,\qquad \ell_j\text{ primitive integral},\quad a_j\in\Z.
 \label{eq:height-exceptional-hyperplanes}
\end{equation}
Factors with no integer zeros impose no restriction.  For
\(k\notin\mathcal H\), put
\[
 F(k)=\bigl(g_1(k)/b_1(k),\ldots,g_r(k)/b_r(k)\bigr)^t.
\]
Evaluating the identities and dividing by the nonzero factors shows
that \(f(k)\ne0\) implies \(F(k)\ne0\). We also obtain
\eqref{eq:height-propagation} along every coordinate
edge whose endpoints avoid \(\mathcal H\).
The rational matrices obtained from the two inverse shifts satisfy
\(\mathsf B_j(s+e_j)\mathsf A_j(s)=I\), so the matrices in
\eqref{eq:height-propagation} are inverse at the indicated integer
points. If $\operatorname{supp}f$ is $T$-proper, so is the support of
$F$, since it is contained in $\bigcup_i\operatorname{supp}g_i$.
Nonzero values therefore propagate along every finite coordinate path
avoiding \(\mathcal H\).

If \(r=0\), clearing the identity \(f=0\) in
\(\mathcal N_{\mathcal S}\) gives \(hf=0\) in \(\mathcal N\) for
some \(h\in\mathcal S\).  The support of \(f\) is then contained
in a finite union of hyperplanes of the form
\eqref{eq:height-exceptional-hyperplanes}.
\end{proof}

\subsection{Descent and restriction}
We choose from finitely many bounded lattice steps to move toward
$\ker T$. Their projected decrease will control the number of steps
and hence the total displacement.

\begin{lemma}[Inward lattice steps]\label{lem:height-inward-steps}
Let $E\subset\R^n$ be a nonzero linear subspace and choose
$\alpha_0>2\sqrt n$. Put
\[
 \gamma=\alpha_0-\sqrt n/2,\quad
 R_{\rm step}=\alpha_0+\sqrt n/2,\quad
 R_*=R_{\rm step}^2/\gamma.
\]
There is a finite set of integer vectors from which, for every unit
$\upsilon\in E$, one can choose $w(\upsilon)$ satisfying
\begin{equation}
 \upsilon\cdot w(\upsilon)\le-\gamma,\qquad
 \|w(\upsilon)\|\le R_{\rm step}.
 \label{eq:height-inward-step}
\end{equation}
A lattice path with unit coordinate steps from $0$ to each step can be chosen with all
vertices of norm at most $R_{\rm path}=R_{\rm step}$. If
$\rho=\|P_Ek\|>R_*$ and $\upsilon=P_Ek/\rho$, then
$\|P_E(k+w(\upsilon))\|\le\rho-\gamma/2$.
\end{lemma}

\begin{proof}
Round $-\alpha_0\upsilon$ coordinatewise to an integer vector
$w(\upsilon)$. The rounding error has norm at most $\sqrt n/2$,
which proves \eqref{eq:height-inward-step} and leaves only finitely
many possible vectors. Along a path changing coordinates monotonically,
each coordinate is bounded in absolute value by that of the endpoint.
Every vertex therefore has norm at most
$R_{\rm step}$. Finally,
\begin{equation}
 \|P_E(k+w(\upsilon))\|^2
 \le\rho^2-2\gamma\rho+R_{\rm step}^2.
 \label{eq:height-decrease}
\end{equation}
For $\rho>R_*$ the right-hand side is at most
$(\rho-\gamma/2)^2$. Since $\rho-\gamma/2>0$, taking square roots
proves the decrease.
\end{proof}

\begin{proof}[Proof of Theorem~\ref{thm:holonomic-height}]
\emph{Step 1: reductions.}
We induct on \(n\), starting with the immediate case \(n=0\).
For \(T=0\), the support is finite by \eqref{eq:height-properness}.
Assume \(n\ge1\) and \(T\ne0\), and put \(E_T=(\ker T)^\perp\).
Injectivity of the linear map on this subspace implies that \(\|Tk\|\)
and \(\|P_{E_T}k\|\) bound one another by fixed factors. Apply
Proposition~\ref{prop:height-propagation} to obtain $F$, $g_i$ and
$\mathcal H$.

\emph{Step 2: thickening and long-step propagation.}
Use Lemma~\ref{lem:height-inward-steps} with $E=E_T$ and fix the
coordinate paths supplied there.

Thicken the exceptional hyperplanes to contain any starting point whose
chosen path might meet them:
\begin{equation}
 \mathcal H^+=\bigcup_j
 \left\{k\in\Z^n:|\ell_j(k)-a_j|
 \le\left\lceil\|\ell_j\|R_{\rm path}\right\rceil+1\right\}.
 \label{eq:height-thickening}
\end{equation}
This is a finite union of lattice hyperplanes.
If \(k\notin\mathcal H^+\), every vertex of the chosen path from
\(k\) to \(k+w(\upsilon)\) avoids \(\mathcal H\).  Therefore
\begin{equation}
 F(k)\ne0\quad\Longrightarrow\quad F(k+w(\upsilon))\ne0.
 \label{eq:height-long-step}
\end{equation}
\emph{Step 3: restriction to the thickening.}
Restrict each \(g_i\) to each hyperplane in \(\mathcal H^+\).
By Lemma~\ref{lem:holonomic-slice}, the restriction is holonomic,
with \(T\)-proper support. Write the
integer points of a nonempty slice as \(k_0+\iota\Z^{n-1}\), where
\(\iota:\R^{n-1}\to\R^n\) is the injective integer linear map of
Lemma~\ref{lem:holonomic-slice}.
If $\|T\iota j\|\le R$, then
$\|T(k_0+\iota j)\|\le R+\|Tk_0\|$. Since $\iota$ is injective,
the restricted support is proper for this projection.
We may therefore apply induction with projection \(T\iota\).
The translation \(Tk_0\) changes only the additive constant;
maximizing over the finitely many slices gives
\begin{equation}
 \|k\|\le C_*(1+\|Tk\|)
 \quad\text{on }\mathcal H^+\cap\bigcup_i\operatorname{supp}g_i.
 \label{eq:height-slice-bound}
\end{equation}
Applying the same argument to \(f\) on the original hyperplanes
bounds \(\mathcal H\cap\operatorname{supp}f\).  This proves the
result if $r=0$.

\emph{Step 4: descent.}
Let \(k\notin\mathcal H\) with \(f(k)\ne0\), so
\(F(k)\ne0\).  If \(k\in\mathcal H^+\), some \(g_i(k)\ne0\)
and \eqref{eq:height-slice-bound} applies.  Otherwise repeat
\eqref{eq:height-long-step} until the projected norm is at most
\(R_*\) or an endpoint enters \(\mathcal H^+\).  By
the last assertion of Lemma~\ref{lem:height-inward-steps}, this takes at most
\[
 2\|P_{E_T}k\|/\gamma+1
\]
steps. In the first case, properness of the support of \(F\) leaves
finitely many terminal indices, hence bounded norms. In the second,
some \(g_i\) is nonzero at the terminal index, whose norm is bounded
by \eqref{eq:height-slice-bound}. In both cases projection onto the
chosen subspace has no larger norm at the endpoint than initially.
Each step has norm at most $R_{\rm step}$. If
$k_{\rm end}$ is the terminal index, the triangle inequality gives
\[
 \begin{split}
 \|k\|&\le\|k_{\rm end}\|+
 R_{\rm step}\bigl(2\|P_{E_T}k\|/\gamma+1\bigr)\\
 &\le C'(1+\|P_{E_T}k\|)\le C''(1+\|Tk\|).
 \end{split}
\]
This completes the induction.
\end{proof}

\subsection{Fourier coefficients of toral distributions}
We identify algebraic torus operators with difference operators by
Fourier transformation, then apply the support-height theorem to the
flux measure.

\begin{lemma}[Fourier coefficients and holonomicity]
\label{lem:fourier-holonomic}
If $u$ is an algebraically holonomic distribution on $\T^n$, with
complexification $(\C^*)^n$, then the sequence $k\mapsto\widehat u(k)$
is algebraically holonomic.
\end{lemma}

\begin{proof}
Under the negative-exponent Fourier convention, multiplication and
distributional integration by parts give
\[
 \widehat{z_j^{-1}u}(k)=\widehat u(k+e_j),\qquad
 \widehat{z_j\partial_{z_j}u}(k)=k_j\widehat u(k).
\]
Thus Fourier coefficients intertwine these operators with $\tau_j$
and $s_j$, respectively, through the Mellin identification with
$t_j=z_j^{-1}$. The coefficient sequence generates an image of the
holonomic torus module generated by $u$, so it is holonomic. With
$T$-proper support, it satisfies the support-height theorem.
\end{proof}

\begin{corollary}[Linear spectral height; Theorem~\ref{thm:main-height}]
\label{cor:universal-linear-height}
Let $\mu$ be an integer-mass FQ measure on $\R^d$, with spectrum
$\Lambda'$. Its spectral group has finite rank $n\ge d$.
Choose a basis of \(\langle\Lambda^{\prime}\rangle_{\Z}\) and let its
elements be the rows of \(L\in\R^{n\times d}\).  For each \(\xi\in\Lambda^{\prime}\), write
\(\xi=L^tk(\xi)\), with \(k(\xi)\in\Z^n\).  There is a constant
\(C<\infty\) such that
\begin{equation}
 \|k(\xi)\|\le C(1+\|\xi\|)\qquad(\xi\in\Lambda^{\prime}).
 \label{eq:universal-height}
\end{equation}
Consequently
\begin{equation}
 \#\{\xi\in\Lambda^{\prime}:\|\xi\|\le R\}=O((1+R)^n).
 \label{eq:spectral-point-count}
\end{equation}
Moreover, taking $C>0$ and putting $V=L\R^d$, we have
\[
 \|P_Vk\|\ge\varepsilon\|k\|
\]
whenever $\widehat m(k)\ne0$ and $\|k\|\ge2C$, where
$\varepsilon=\min\{1,(2C\|L\|)^{-1}\}$.
\end{corollary}

\begin{proof}
By Theorem~\ref{thm:all-fq-toral-cycles}, the spectral group is finitely
generated and has the stated basis matrix. A nonzero real kernel of $L$
would, by covariance, leave the nonzero locally finite measure invariant
under a continuous line of translations, which is impossible.
Hence $L$ has rank $d$ and $n\ge d$. The flux measure $m$ is defined
in \eqref{eq:HF-4.1}, and \eqref{eq:HF-4.2} gives
$\widehat m(k)=c_{L^tk}$ at spectral indices and zero at the other
indices. By Theorem~\ref{thm:holonomic-flux}, $m$ is holonomic;
by Lemma~\ref{lem:fourier-holonomic}, so is $k\mapsto\widehat m(k)$.
The map $T=L^t$ is injective on $\Z^n$, since the rows
of $L$ form a basis of the spectral group. Hence local finiteness of
$\Lambda'$ gives \eqref{eq:height-properness}.
Theorem~\ref{thm:holonomic-height} proves \eqref{eq:universal-height}.
If $\|\xi\|\le R$, its integer coordinate belongs to the ball of
radius $C(1+R)$ in $\Z^n$. Counting lattice points in this ball proves
\eqref{eq:spectral-point-count}.

Choose the height constant positive. If $\widehat m(k)\ne0$ and
$\|k\|\ge2C$, the height estimate gives
\[
 \frac{\|k\|}{2C}\le\|L^tk\|
 =\|L^tP_Vk\|\le\|L\|\,\|P_Vk\|.
\]
This proves the stated cone bound and hence \eqref{eq:branched-cone}.
\end{proof}

\section{Reconstruction from linear height}
\label{sec:non-delone-lh}

We reconstruct the canonical algebraic cycle from the Fourier cone.
Rational projections near the physical plane give algebraic torus
quotients whose fiber moments are Laurent polynomials. The moment
matrices recover the fibers, while their growth along rays yields the
amoeba cone required for compression to a Lee--Yang cycle.

\begin{theorem}[Reconstruction from a Fourier cone]
\label{thm:branched-scalar-cone}
Let $M,m,L,V$ be the canonical set, flux measure, spectral basis
matrix, and physical plane of an integer-mass FQ measure $\mu$, and put
$c=n-d$ and $Y=\overline M^{\mathrm{Zar}}$.
Suppose that for some $0<\varepsilon\le1$ and $R<\infty$,
\begin{equation}
 \|P_Vk\|\ge\varepsilon\|k\|
 \quad\text{whenever }\widehat m(k)\ne0\text{ and }\|k\|\ge R.
 \label{eq:branched-cone}
\end{equation}
Then
\begin{equation}
 Y(\T)=M,\qquad
 \|P_{V^\perp}\Log z\|\ge\varepsilon\|\Log z\|\quad(z\in Y).
 \label{eq:branched-amoeba}
\end{equation}
Every rational $d$-plane $W$ sufficiently close to the physical plane gives a finite quotient
$\pi_W:Y\to(\C^*)^c$ with full inverse image of the compact base torus
exactly $M$.

At every $p\in M$, the reduced analytic germs satisfy
\[
 (Y^{\mathrm{an}},p)=(M^\C,p).
\]
There are unique positive integers $\varkappa_i$ on the irreducible
components $Y_i$ of $Y$ such that the effective algebraic cycle
\[
 \mathfrak Y=\sum_i\varkappa_iY_i
\]
has germ $\mathfrak M^\C$ at every point of $M$. It is invariant under
reciprocal conjugation, and its section at every real phase $t$ is
$\mu_t$, with weighted local intersection degrees.
\end{theorem}

\subsection{Rational projections and fiber moments}
\label{sec:rational-projections}
Assume the hypotheses of Theorem~\ref{thm:branched-scalar-cone}, with
$c=n-d$ and $Y=\overline M^{\mathrm{Zar}}$, and first suppose $c>0$.
By Corollary~\ref{cor:all-fq-algebraicity}, $Y$ is reduced and pure of
dimension $c$. Lemma~\ref{lem:ac-transverse-strata} shows that the
smooth $V$-transverse locus of $M$ is dense. The germ inclusion in
Corollary~\ref{cor:all-fq-algebraicity}, together with purity of
dimension $c$, identifies the complexifications of the
$c$-dimensional analytic pieces of the canonical set with local branches of $Y$.

For each irreducible component $Y_j$ of $Y$, choose
$p_j\in M\cap Y_j$ at which $Y$ is smooth and $M$ is smooth and
$V$-transverse. Such points exist: the smooth $V$-transverse locus
has Zariski closure $Y$, so its intersection with each $Y_j$ is
Zariski dense there, whereas $\operatorname{Sing}Y\cap Y_j$ is a
proper closed subset. Near $p_j$, the germ of $Y$ is the
complexification of the germ of $M$, since both have dimension $c$.
Thus $T_{p_j}Y=T_{p_j}M\otimes_\R\C$ in logarithmic coordinates.
Call a rational $d$-plane $W$ \emph{admissible} if
$T_{p_j}M\oplus W=\R^n$ for every $j$ and
\[
 \delta=\|P_W-P_V\|<\varepsilon,
 \qquad \eta=\varepsilon-\delta>0.
\]
Let $\mathcal L=W^\perp\cap\Z^n$, let
$Q_\C=\operatorname{Spec}\C[\mathcal L]\simeq(\C^*)^c$, and let
$\pi_W:(\C^*)^n\to Q_\C$ be the character quotient.
Write $Q$ for its compact real torus, with normalized Haar measure
$dw$. A real base point means a point of this compact torus.
Choose an integral splitting of the character lattice;
we write the resulting coordinates as $(\zeta,w)$ and characters as
$z^k=\zeta^a w^\ell$, for $k=(a,\ell)$.
Set $\mathcal R=\C[\mathcal L]$ and $K=\operatorname{Frac}\mathcal R$.

\begin{lemma}[Admissible projections]\label{lem:admissible-projections}
Every rational $d$-plane in a sufficiently small neighborhood of $V$
is admissible. For such a plane, $\pi_W|_Y$ is dominant and generically
finite on each component, and
\begin{equation}
 \|P_Wk\|\ge\eta\|k\|
 \quad\text{if }\widehat m(k)\ne0,\quad\|k\|\ge R.
 \label{eq:reconstruction-W-cone}
\end{equation}
\end{lemma}

\begin{proof}
Transversality at the chosen points is open in $W$ and holds at $V$;
rational planes are dense. The quotient differential therefore has
rank $c$ at each $p_j$. Since $p_j$ is smooth on $Y$, it belongs
to $Y_j$ alone; therefore $Y_j$ dominates the base with generically
zero-dimensional fibers. The cone inequality follows from
$\|P_Wk\|\ge\|P_Vk\|-\delta\|k\|$.
\end{proof}

\begin{lemma}[Polynomial fiber moments]\label{lem:polynomial-fiber-moments}
For every ambient character $k\in\Z^n$ there is $p_k\in\mathcal R$
such that $\pi_{W*}(z^k m)=p_k\,dw$.
In the chosen splitting, $p_{(a,\ell)}=w^\ell p_a$, where
$p_a=p_{(a,0)}$, and $p_{-k}=\overline{p_k}$ on $Q$.
There are positive finite conditional measures $m_w$ with
$m=\int_Qm_w\,dw$ and moments $p_a(w)$ for almost every $w$.
On a nonempty real base open set $\Omega$, for almost every
$w\in\Omega$ these measures have the form
\[
 m_w=\sum_{h=1}^r\varpi_h(w)\delta_{\zeta_h(w)},
 \qquad \varpi_h(w)>0,
\]
where the $r$ distinct points and their weights are real analytic.
\end{lemma}

\begin{proof}
We first use the cone to obtain polynomial moments, then describe the
conditional measures on regular fibers. For every $a\in\Z^d$, the
measure $\pi_{W*}(\zeta^a m)$ has finitely many nonzero Fourier
coefficients: at a base index $\ell$ its coefficient is
$\widehat m(-a,\ell)$, and \eqref{eq:reconstruction-W-cone} bounds
$\ell$ for fixed $a$. Fourier uniqueness gives
\begin{equation}
 \pi_{W*}(\zeta^a m)=p_a(w)\,dw,
 \qquad p_a\in \mathcal R.
 \label{eq:branched-moments}
\end{equation}
Since $p_0\ge0$, disintegration yields positive finite conditional
measures on the compact fiber torus, with moments $p_a(w)$
simultaneously for almost every $w$. Multiplication by $w^\ell$ and
positivity give the identities for ambient characters and conjugates.

Choose a nonempty open subset of the real base over which the
morphism is finite and unramified, avoiding also the image of the
locus where the logarithmic tangent space of $Y_{\rm reg}$ meets
$V_{\C}$ nontrivially. To justify this choice, each component has a
smooth $V$-transverse point, so that locus lies in a proper algebraic
subset. Over an open set where $\pi_W$ is finite, its image and the
singular image are proper algebraic subsets of the base. The
exceptional values of the generically finite morphism also lie in a
proper algebraic subset. None contains a nonempty open part of the
maximal totally real torus $Q$, including the image of a transverse
real open piece of $M$.
Over a remaining real base point the full algebraic fiber is finite,
regular, and $V$-transverse. At a point $y$ of $M$ in this fiber,
Corollary~\ref{cor:all-fq-algebraicity} gives
$(M^\C,y)=(Y^{\mathrm{an}},y)$, and
Proposition~\ref{prop:ac-canonical-germ} identifies the real part of
$M^\C$ with $M$. Conjugation invariance of the smooth germ in
logarithmic coordinates implies that $M$ is a real $c$-dimensional
manifold near $y$, with $T_yM\otimes_\R\C=T_yY$. Since
$\pi_W|_Y$ is \'etale at $y$, its differential maps $T_yM$
isomorphically onto the tangent space of $Q$. Thus $\pi_W|_M$ is a
local diffeomorphism onto a neighborhood in $Q$, and uniqueness of
the holomorphic inverse sends real base points into $M$.
The points of $M$ in this fiber persist as real local graphs. If additional
points of $M$ occurred over base points tending to the chosen point $w_0$,
compactness would give a limit at one of the points $y_1,\ldots,y_r$ in
the original fiber of $M$, contradicting the local graph descriptions.
Shrink to a real base open set $\Omega$ on which $M$ has a fixed
positive number $r$ of distinct fiber points. Then
\[
 p_a(w)=\sum_{h=1}^r\varpi_h(w)\zeta_h(w)^a,
 \qquad \varpi_h(w)>0.
\]
By Lemma~\ref{lem:flux-density}, the local flux measure is
$\varkappa_h|\beta_L|$, where $\varkappa_h$ is the constant positive
integer coefficient of the canonical analytic cycle on the graph.
The flux density is nonzero by $V$-transversality, and the projection
Jacobian is nonzero by $W$-transversality. The weights are their
absolute ratio times the coefficient: for a tangent basis $v_1,\ldots,v_c$
on the graph,
\[
 \varpi_h=\varkappa_h\frac{|\beta_L(v_1,\ldots,v_c)|}
 {|\pi_W^*dw(v_1,\ldots,v_c)|}.
\]
Here $\pi_W^*dw$ is the pullback of normalized Haar density on $Q$.
On a sufficiently small connected graph, both determinants are analytic
and nonzero. Their absolute values, and hence the positive weights,
are real analytic.
\end{proof}

\subsection{The finite moment algebra}
We recover each fiber as the joint spectrum of multiplication by
characters. A common denominator for all powers of these operators
will imply finiteness over the base.

Write $\mathbf f_1,\ldots,\mathbf f_d$ for the standard basis of
$\Z^d$, and $\mathbf u_1,\ldots,\mathbf u_r$ for that of $\C^r$.

\emph{Construction.}
Choose monomials $b_i=\zeta^{\beta_i}$, $1\le i\le r$, with $b_1=1$,
whose evaluation matrix $\mathsf E_{ih}=b_i(\zeta_h)$ is invertible at
one fiber over $\Omega$. Laurent monomials span the functions on a
finite set, so we may choose such a basis containing the constant.
Shrink $\Omega$ so that $\mathsf E$ remains invertible, and define
\begin{equation}
 \begin{aligned}
 H&=(p_{\beta_i-\beta_j})_{i,j},&
 H_a&=(p_{\beta_i-\beta_j+a})_{i,j},\\
 \mathsf U_a&=H^{-1}H_a,& \Delta&=\det H.
 \end{aligned}
 \label{eq:branched-hermitian}
\end{equation}

The constant monomial will be a cyclic vector for these matrices,
identifying the algebra with its representation away from the
determinant locus.

\begin{lemma}[The moment algebra]\label{lem:finite-moment-algebra}
The Laurent polynomial $\Delta=\det H$ is nonzero, the matrices
$\mathsf U_{\mathbf f_j}$ commute and are invertible, and
\begin{equation}
 \mathsf U_a=\prod_{j=1}^d \mathsf U_{\mathbf f_j}^{a_j},\qquad H\mathsf U_a=H_a
 \quad(a\in\Z^d)
 \label{eq:branched-all-powers}
\end{equation}
holds over $K$. Their joint $K$-algebra is \'etale of dimension $r$.
The algebra
$\mathcal B=\mathcal R[\mathsf U_{\mathbf f_1}^{\pm1},\ldots,
\mathsf U_{\mathbf f_d}^{\pm1}]$ is finite, reduced and torsion-free
of generic rank $r$. Its spectrum
$Z=\operatorname{Spec}\mathcal B$ is a closed subvariety of the ambient
torus, finite over $Q_\C$, and every component dominates the base.
The map $A\mapsto A\mathbf u_1$ identifies
$\mathcal B[\Delta^{-1}]$ with $\mathcal R[\Delta^{-1}]^r$ as modules
and its matrix representation remains faithful on every fiber with $\Delta\ne0$.
\end{lemma}

\begin{proof}
\emph{Step 1: moment matrices.}
On the real open set,
\[
 H=\mathsf E\operatorname{diag}(\varpi_h)\mathsf E^*,\qquad
 \mathsf U_a=\mathsf E^{-*}\operatorname{diag}(\zeta_h^a)\mathsf E^*.
\]
These formulas show that $\Delta\ne0$ and that the coordinate
matrices commute and are invertible. They also prove
\eqref{eq:branched-all-powers} over $K=\C(w_1,\ldots,w_c)$, since
identities on a real open part of the base torus are rational identities. Choose a generic linear
combination $X=\sum_j\lambda_j\mathsf U_{\mathbf f_j}$ with $r$
distinct eigenvalues at the chosen fiber. The discriminant of its
characteristic polynomial is therefore nonzero in $K$. Thus $K[X]$
is \'etale of dimension $r$ and equals its centralizer in
$\operatorname{Mat}_r(K)$. The joint algebra is commutative and contains
$X$, so it equals $K[X]$, of dimension $r$.

\emph{Step 2: integrality.}
The common denominator in \eqref{eq:branched-all-powers} is independent
of the multi-index. Consequently the module
\[
 \mathcal P=\sum_{a\in\Z^d}\mathsf U_a \mathcal R^r
 \subset\Delta^{-1}\mathcal R^r
\]
is finitely generated, since $\mathcal R$ is Noetherian. It contains
$\mathcal R^r$ and is preserved by each $\mathsf U_{\mathbf f_j}$
and its inverse. Express any such operator on generators of
$\mathcal P$ by a matrix over $\mathcal R$. By the adjugate identity,
its monic characteristic polynomial annihilates every generator and
therefore $K^r$. Each coordinate matrix and its inverse is thus
integral over $\mathcal R$, so
\[
 \mathcal B=\mathcal R[\mathsf U_{\mathbf f_1}^{\pm1},\ldots,\mathsf U_{\mathbf f_d}^{\pm1}]
\]
is a finite reduced torsion-free $\mathcal R$-algebra of generic degree
$r$, defining a closed subvariety $Z$ of the ambient torus. Every
component dominates the base: $\mathcal B$ embeds in a product of
finite extensions of $K$, and its components are closures of the
generic components. Equivalently, nonzero base elements are
non-zero-divisors in $\mathcal B$ and belong to no minimal prime.
It follows that $Z$ is pure of dimension $c$ and $Z\to Q_{\C}$ is finite.

The chosen character splitting gives the presentation
\[
 \mathcal R[\zeta_1^{\pm1},\ldots,\zeta_d^{\pm1}]
 \twoheadrightarrow\mathcal B,
 \qquad \zeta_j\longmapsto\mathsf U_{\mathbf f_j}.
\]

\emph{Step 3: a cyclic vector.}
For each $i$, the first column of $H_{-\beta_i}$ is the $i$th
column of $H$.  Hence
\begin{equation}
 H_{-\beta_i}\mathbf u_1=H\mathbf u_i,\qquad
 \mathsf U_{-\beta_i}\mathbf u_1=\mathbf u_i.
 \label{eq:branched-cyclic-vector}
\end{equation}
These identities hold wherever $\Delta\ne0$.  In particular $\mathbf u_1$
is a joint cyclic vector there.  A matrix commuting with all the
$\mathsf U_a$ and annihilating $\mathbf u_1$ annihilates every $\mathbf u_i$.
Thus $A\mapsto A\mathbf u_1$ identifies
$\mathcal B[\Delta^{-1}]$ with $\mathcal R[\Delta^{-1}]^r$ as modules:
injectivity follows from cyclicity, and surjectivity from
\eqref{eq:branched-cyclic-vector}.  Consequently this representation
remains faithful on every fiber with $\Delta\ne0$.
\end{proof}

\begin{lemma}[Identification of the fibers]\label{lem:moment-fiber-identification}
The variety constructed in Lemma~\ref{lem:finite-moment-algebra} satisfies
$Z=Y$ and $Y\cap\pi_W^{-1}(Q)=M$. In particular $Y(\T)=M$ and
$\pi_W|_Y$ is finite.
\end{lemma}

\begin{proof}
We identify regular fibers by positivity of the moment matrix and reach
exceptional fibers by local openness of a finite map.

\emph{Step 1: real fibers away from the discriminant.}
Outside the discriminant and the zero set of $\Delta$, almost
every real base point has $H$ as the Gram matrix of the monomials
$b_i$ for a positive conditional measure. It is therefore positive
semidefinite, and continuity extends this to all of $Q$; it is
positive definite wherever $\Delta(w)\ne0$.
Furthermore $H_a^*=H_{-a}$ on the real torus, and
\eqref{eq:branched-all-powers} gives
\[
 \mathsf U_{\mathbf f_j}^* H \mathsf U_{\mathbf f_j}=H.
\]
The commuting matrices are therefore jointly unitary for this
positive inner product. Their joint spectral measure at the vector
$\mathbf u_1$ is a positive finite atomic measure on $\T^d$ with moments
\[
 \mathbf u_1^*H\mathsf U_a \mathbf u_1=(H_a)_{11}=p_a.
\]
For almost every base point, Fourier uniqueness on $\T^d$ identifies
it with the conditional measure.  By \eqref{eq:branched-cyclic-vector}, the vector $\mathbf u_1$ is
jointly cyclic at every such point.  Each joint eigenspace therefore
occurs with positive mass. The faithful representation identifies the
joint spectral points with the full fiber of $Z$, consisting of $r$
points away from the branch locus
$\operatorname{Disc}(\mathcal B/\mathcal R)$, the image of the
non-\'etale locus of the finite map $Z\to Q_\C$. This image is closed
because the map is finite, and it is proper because
$\mathcal B\otimes K$ is \'etale.

The branch locus and the zero set of $\Delta$ have Haar measure zero,
being proper algebraic subsets of the base. Since
$\pi_{W*}m=p_0\,dw$, their inverse images have zero $m$ mass, so
$m$ is concentrated on the closed set $Z$. By full support
(Proposition~\ref{prop:ac-flux-support}), $M\subset Z$.
Conversely, each component of $Z$ dominates the base and contains a
holomorphic graph over a real base open set avoiding these loci.
Almost every point of that graph over the real base belongs to $M$;
closedness extends this to every such point. This maximal totally
real open piece is Zariski dense in the component, proving
\begin{equation}
 Z=\overline M^{\mathrm{Zar}}=Y.
 \label{eq:branched-identify}
\end{equation}
The same argument identifies the full fiber over every real base point
outside these two loci.

\emph{Step 2: exceptional fibers.}
Let $y\in Z$ lie over a real point $w_0$ of the discriminant or the
zero set of $\Delta$. A finite map from a pure $c$-dimensional
complex space to a smooth $c$-dimensional base is locally open.
Indeed, isolate the points above $w_0$ in disjoint neighborhoods.
Properness provides a base ball whose preimage lies in these
neighborhoods, each piece finite and proper over the ball. The image
of every local component through $y$ is analytic of dimension $c$,
since fibers are finite, and therefore contains a base neighborhood.
Every neighborhood of $y$ thus contains preimages of real base points
outside the two loci tending to $w_0$. These preimages belong to $M$; since $M$ is closed, $y\in M$. The full inverse image of
the base torus is therefore $M$.
\end{proof}

\subsection{The amoeba cone}
\label{sec:reconstruction-amoeba}
Let $C_V$ be the Zariski closure in $Y$ of the locus in $Y_{\rm reg}$
where the logarithmic tangent space meets $V_\C$ nontrivially, and put
\[
 Z_{\rm exc}=\{\Delta=0\}\cup\operatorname{Disc}(\mathcal B/\mathcal R)
 \cup\pi_W(\operatorname{Sing}Y)\cup\pi_W(C_V)\subset Q_\C.
\]
This algebraic subset is proper. By
Lemma~\ref{lem:moment-fiber-identification}, $\pi_W|_Y$ is finite,
so the two images are closed. Reducedness of $Y$ implies
$\dim\operatorname{Sing}Y<c$; also $C_V$ omits every $p_j$ and
therefore meets each $Y_j$ in a proper closed subset. Both images
have dimension less than $c$, while the other two sets are proper by
Lemma~\ref{lem:finite-moment-algebra}. We work over
$Q_\C\setminus Z_{\rm exc}$.

The Fourier cone bounds the Laurent support of the moments. Their
exponential growth along a ray will then bound each character on the
fiber; the shifted moments distinguish points with the same character
value and prevent cancellation.

\begin{lemma}[Moment growth and the amoeba]\label{lem:moment-amoeba}
For every admissible $W$,
\[
 \|P_{W^\perp}\Log z\|\ge\eta\|\Log z\|\qquad(z\in Y).
\]
Consequently $\|P_{V^\perp}\Log z\|\ge\varepsilon\|\Log z\|$ on $Y$.
\end{lemma}

\begin{proof}
\emph{Step 1: Laurent supports.}
For an ambient character $\chi_k(z)=z^k$, $k\in\Z^n$, denote its
polynomial fiber moment by $p_k$. Identifying the fiber character
lattice with its chosen integral complement in $\Z^n$ extends the
notation $p_a$. The Laurent expansion satisfies
\begin{equation}
 p_k(w)=\sum_{\ell\in\mathcal L}\widehat p_k(\ell)w^\ell,
 \qquad \widehat p_k(\ell)\ne0\Longrightarrow
 \ell-k\in\operatorname{supp}\widehat m,
 \qquad |\widehat p_k(\ell)|\le m(\T^n).
 \label{eq:branched-coefficients}
\end{equation}
The last inequality is the Fourier coefficient bound for a finite
positive measure, uniform in $k$.

The directional inequality is immediate when $k\in W^\perp$.
Write $k_W=P_Wk\ne0$, $k_\perp=P_{W^\perp}k$, and let $\widetilde\beta_i$ be the ambient
characters of $b_i$. For $p_{Nk+\widetilde\beta_i}$,
\eqref{eq:reconstruction-W-cone} gives, for every $\ell$ with
$\widehat p_{Nk+\widetilde\beta_i}(\ell)\ne0$ and all sufficiently
large positive integers $N$,
\[
 \|\ell-P_{W^\perp}(Nk+\widetilde\beta_i)\|
 \le C_\eta\|P_W(Nk+\widetilde\beta_i)\|,
 \qquad C_\eta=\sqrt{\eta^{-2}-1}.
\]
No exceptional spectral index contributes for large enough integers:
its $W$-projection would have norm at least $N\|k_W\|-O(1)$.
The rank-$c$ lattice $\mathcal L$ has $O((1+N)^c)$ indices in this
ball.

\emph{Step 2: spectral decomposition on complex fibers.}
For $w\in Q_\C\setminus Z_{\rm exc}$, let $y_h=(\zeta_h,w)$ denote the
$r$ ambient points and let $P_h$ be the joint spectral idempotents of
the matrices in Lemma~\ref{lem:finite-moment-algebra}. Since
$\Delta\ne0$, the $r$-dimensional fiber representation is faithful;
the fiber is \'etale because it lies outside the branch locus. Thus
$\mathsf U_{\mathbf f_j}(w)$ are simultaneously diagonalizable, with
rank-one idempotents corresponding to the $r$ fiber points, and
\[
 \mathsf U_a=\sum_h\zeta_h^aP_h,\qquad
 p_a=(H\mathsf U_a)_{11}
     =\sum_h\varpi_h\zeta_h^a,
 \qquad \varpi_h=(HP_h)_{11}.
\]
Writing $(\mathsf E_+)_{ih}=\zeta_h^{\beta_i}$ and
$(\mathsf E_-)_{ih}=\zeta_h^{-\beta_i}$ gives
\[
 H=\mathsf E_+\operatorname{diag}(\varpi_h)\mathsf E_-^t.
\]
On the initial real open set, $\mathsf E_-=\overline{\mathsf E_+}$
and $b_1=1$. Comparing the expansion
$p_a=\sum_h(HP_h)_{11}\zeta_h^a$ with the disintegration formula
$p_a=\sum_h\varpi_h\zeta_h^a$ on $\Omega$, linear independence of
the characters at distinct points identifies $(HP_h)_{11}$ with
the conditional weight $\varpi_h$. The determinant of the
factorization, being nonzero when $\Delta\ne0$, forces both evaluation
matrices to be invertible and every $\varpi_h$ to be nonzero. For
ambient characters, $p_k=\sum_h\varpi_h\chi_k(y_h)$.

\emph{Step 3: growth along a ray.}
For any common value $\gamma=\chi_k(y_h)$, the
vector
\[
 \left(\sum_{h:\chi_k(y_h)=\gamma}\varpi_h\chi_{\widetilde\beta_i}(y_h)\right)_{i=1}^r
\]
is nonzero by invertibility of the evaluation matrix. Thus at least
one of the generating functions
\[
 \sum_{N\ge0}p_{Nk+\widetilde\beta_i}(w)\tau^N
 =\sum_h\frac{\varpi_h\chi_{\widetilde\beta_i}(y_h)}
                  {1-\tau\chi_k(y_h)}
\]
has a pole at $1/\gamma$. The Cauchy--Hadamard formula yields
\[
 \limsup_{N\to\infty}\max_i|p_{Nk+\widetilde\beta_i}(w)|^{1/N}
 =\max_h|\chi_k(y_h)|.
\]
\emph{Step 4: the cone estimate.}
Let $u_h=P_W\Log y_h$ and $v=P_{W^\perp}\Log y_h$; the latter is
constant on the fiber. The coefficient and support bounds give
\[
 |p_{Nk+\widetilde\beta_i}(w)|
 \le C(1+N)^c
 \exp\bigl(N[k_\perp\cdot v+C_\eta\|k_W\|\|v\|]+O(1)\bigr).
\]
Fix $k$. Taking $N$th roots and cancelling $k_\perp\cdot v$ gives
$k_W\cdot u_h\le C_\eta\|k_W\|\|v\|$ at every point of every fiber
over $Q_\C\setminus Z_{\rm exc}$, that is,
$(P_Wk)\cdot P_W\Log z\le C_\eta\|P_Wk\|\,\|P_{W^\perp}\Log z\|$
for $z$ in a dense subset of $Y$: the fibers over
$Q_\C\setminus Z_{\rm exc}$ are dense since every component
dominates the base. Continuity on $Y$ extends the inequality to all
of $Y$. Varying $k$, density of $P_W\Q^n$ in $W$ implies that the
directions of nonzero $P_Wk$, $k\in\Z^n$, are dense in the unit
sphere of $W$. We obtain
$\|P_W\Log z\|\le C_\eta\|P_{W^\perp}\Log z\|$ on $Y$; equivalently
\[
 \|P_{W^\perp}\Log z\|\ge\eta\|\Log z\|\qquad(z\in Y).
\]
Let the admissible rational planes $W$ tend to $V$.
The variety in \eqref{eq:branched-identify} is independent of $W$,
and $\eta\to\varepsilon$, proving the inequality in
\eqref{eq:branched-amoeba}.
\end{proof}

\subsection{The canonical algebraic cycle}

\begin{lemma}[Equality of analytic germs]
\label{lem:integer-germ-equality}
Under the hypotheses of Theorem~\ref{thm:branched-scalar-cone},
$(Y^{\mathrm{an}},p)=(M^\C,p)$ at every $p\in M$.
\end{lemma}
\begin{proof}
If $c=0$, both germs are points and the assertion is immediate.
Assume $c>0$. The inclusion from $M^\C$ to $Y^{\mathrm{an}}$ was proved in
Corollary~\ref{cor:all-fq-algebraicity}. For the other one, fix an
admissible $W$. Since $p\in M$, we have $\pi_W(p)\in Q$.
Lemma~\ref{lem:moment-fiber-identification} says that $\pi_W$ is
finite and that its full inverse image of $Q$ is $M$. In local
logarithmic coordinates, Lemma~\ref{lem:integer-real-sheets} therefore
applies to $(Y^{\mathrm{an}},p)$ and $\pi_W$. It supplies smooth
unramified real graphs arbitrarily near $p$ on every local branch.

Choose holomorphic defining functions for $M^\C$ on a sufficiently
small representative. They vanish on the real graphs in $M$, so the
last assertion of Lemma~\ref{lem:integer-real-sheets} makes them
vanish on every local branch. Hence
$Y^{\mathrm{an}}\subseteq M^\C$ as germs.
\end{proof}

The local cycle coefficients must agree on different real pieces of
one algebraic component. We recover each conditional weight as a
rational function and divide it by the ratio of the flux and quotient
forms. The result is constant on a real open patch, hence on the
component by the identity principle.

\begin{lemma}[Rational weights and global component coefficients]
\label{lem:integer-global-coefficients}
Under the Fourier cone hypothesis there are unique positive integers
$\varkappa_i$ on the irreducible components $Y_i$ of $Y$ such that
\[
 \mathfrak Y=\sum_i\varkappa_iY_i
\]
has analytic germ $\mathfrak M^\C$ at every point of $M$.
Every $Y_i$ is preserved by reciprocal conjugation. In particular
$\mathfrak Y$ is invariant, and for every real phase $t$ and physical
root $x$,
\[
 \mu_t(\{x\})=\sum_i\varkappa_i\deg^{Y_i}_t(x).
\]
Moreover, for every real phase $t$,
$\{x\in\C^d:\exp(2\pi i(t+Lx))\in Y\}=\supp\mu_t$ and
$\nu^{\mathfrak Y}_t=\mu_t$.
\end{lemma}
\begin{proof}
If $c=0$, $Y=M$ is a finite set of torus points. Give its point $p$
the coefficient $\mu_p(\{0\})$. The weighted germ assertion follows
from Proposition~\ref{prop:ac-canonical-germ}; each point is fixed
by reciprocal conjugation. We now assume $c>0$ while identifying the
rational weights and branch coefficients.

\emph{Normalization of the two forms.}
Fix an admissible rational $W$ and the character splitting of
Section~\ref{sec:rational-projections}. Choose an ordered integral basis
$\ell_1,\ldots,\ell_c$ of $\mathcal L=W^\perp\cap\Z^n$, so the
quotient coordinates are $w_s=z^{\ell_s}$. Put
\[
 \omega_Q=\bigwedge_{s=1}^c\frac{dw_s}{2\pi i w_s},
 \qquad
 \beta_L^{\mathrm{hol}}
   =J\bigwedge_{s=1}^c
        \left(\sum_{j=1}^n (\mathbf b_s)_j\frac{dz_j}{2\pi i z_j}\right),
 \qquad J=\sqrt{\det(L^tL)}.
\]
The vectors $\mathbf b_s$ are the oriented orthonormal basis of $V^\perp$
used in~\eqref{eq:transverse-flux-form}. On $Q$, $\omega_Q$ is the
oriented volume form whose absolute density is the normalized Haar
measure $dw$. On $\T^n$, $\beta_L^{\mathrm{hol}}$ restricts to
$\beta_L$, since $dz_j/(2\pi iz_j)=d\theta_j$.

On each irreducible $Y_i$, the rational top form $\pi_W^*\omega_Q$
is nonzero because $\pi_W|_Y$ is \'etale off the proper branch locus
$\operatorname{Disc}(\mathcal B/\mathcal R)$. Likewise,
$\beta_L^{\mathrm{hol}}$ is nonzero because $Y_i$ contains a real
$V$-transverse point. Their ratio defines a nonzero rational function
\[
 \mathcal G=\frac{\beta_L^{\mathrm{hol}}|_Y}
                  {\pi_W^*\omega_Q|_Y}
\]
on each component, equivalently a tuple in the total ring of rational
functions of the reduced variety $Y$. At a smooth real point where
both projections are transverse, evaluation on any real tangent frame
gives a quotient of two nonzero real numbers. Hence $\mathcal G$ is
real and nonzero there, and
\[
 |\mathcal G|
 =\frac{|\beta_L(v_1,\ldots,v_c)|}
        {|\pi_W^*dw(v_1,\ldots,v_c)|}.
\]
Changing an orientation changes only the sign of $\mathcal G$,
leaving its absolute value and the recovered coefficients unchanged.

\emph{A rational conditional-weight function.}
Call $y\in Y$ good if $\pi_W(y)\notin Z_{\rm exc}$, where
$Z_{\rm exc}$ is defined in Section~\ref{sec:reconstruction-amoeba}.
Keep the monomials and moment matrices of
Lemma~\ref{lem:finite-moment-algebra}. For $y=(\zeta,w)\in Y$ set
\[
 \mathbf v_+(y)=(\zeta^{\beta_i})_{i=1}^r,
 \qquad \mathbf v_-(y)=(\zeta^{-\beta_i})_{i=1}^r,
 \qquad
 \varpi(y)=
 \left(\mathbf v_-(y)^tH(w)^{-1}\mathbf v_+(y)\right)^{-1}.
\]
The expression inside the reciprocal is regular on
$Y\setminus\pi_W^{-1}\{\Delta=0\}$. On a good complex fiber, the
factorization from
Lemma~\ref{lem:moment-amoeba} gives
\[
 \mathsf E_-^tH^{-1}\mathsf E_+
       =\operatorname{diag}(1/\varpi_h).
\]
At $y_h$ the expression inside the reciprocal is therefore
$1/\varpi_h\ne0$ over $Q_\C\setminus Z_{\rm exc}$, so the rational
function has value $\varpi_h$. Every component dominates the base and
meets this open set; the reciprocal is consequently a nonzero rational
function on each component.

\emph{Identification on good real graphs.}
Good real graphs occur arbitrarily near every point of $M$ on each
local irreducible branch. Indeed, by
Lemma~\ref{lem:moment-fiber-identification}, the projection is finite
and its full inverse image of the real base torus is $M$. Each local
branch maps onto a base neighborhood. As in
Lemma~\ref{lem:integer-real-sheets}, choose real base points outside
$Z_{\rm exc}$ approaching the projection of the given point. Their preimages on the
branch give smooth real graphs transverse to both $V$ and $W$.
These graphs are maximal totally real, and may also be chosen away
from the other global components, whose intersections are proper
algebraic subsets. In particular, every $Y_i$ contains such a graph.

Over a sufficiently small real base open set outside $Z_{\rm exc}$,
all fiber points lie on these graphs.
The weight calculation in Lemma~\ref{lem:polynomial-fiber-moments}
applies on every sufficiently small real open set
$\Omega'\subset Q\setminus Z_{\rm exc}$, since the flux and quotient
Jacobians are nonzero there. The conditional density is therefore
$\varkappa(y)|\mathcal G(y)|$. At a real base point with $\Delta(w)\ne0$,
$H$ is positive definite by continuity, and $\mathsf U_a$ are unitary
for the $H$ inner product by Lemma~\ref{lem:moment-fiber-identification}.
Their joint spectral idempotents $P_h$ are self-adjoint for this inner
product.
The spectral weight at $y_h$ is
\[
 \|P_h\mathbf u_1\|_H^2=(HP_h)_{11}=\varpi(y_h),
\]
where the last equality follows from Step~2 of
Lemma~\ref{lem:moment-amoeba}. The spectral measure identification in
Lemma~\ref{lem:moment-fiber-identification} shows that $\varpi$ is
the conditional weight for almost every base point. On a sufficiently
small connected real graph, $\varkappa$ is constant and $\mathcal G$
is nonzero of constant sign. Real analyticity extends the almost
everywhere equality to every point:
\begin{equation}
 \varpi(y)=\varkappa(y)|\mathcal G(y)|.
 \label{eq:integer-weight-identification}
\end{equation}
The identity holds on every good real graph.

\emph{Constancy on a global algebraic component.}
Choose a connected good real graph on $Y_i$. There
\eqref{eq:integer-weight-identification} gives
\[
 \frac{\varpi}{\mathcal G}=\pm\varkappa_i
\]
for a positive integer $\varkappa_i$. A maximal totally real open
patch in the regular locus of an irreducible $c$-dimensional variety
is Zariski dense: a Laurent polynomial vanishing there vanishes on
the complex graph by the identity theorem, and then on the component.
Clearing the rational denominator and applying this to the numerator
proves the identity on all of $Y_i$.

On every other good real patch of $Y_i$, \eqref{eq:integer-weight-identification} and this identity imply
$\varkappa(y)=\varkappa_i$, by taking absolute values.

\emph{Equality of cycles along the torus.}
Fix $p\in M$ and a local irreducible branch of $Y$ there. By
Lemma~\ref{lem:integer-germ-equality} it is also a branch of $M^\C$.
It belongs to one global $Y_i$ and, by the branchwise density proved
above, contains good real patches away from the other components.
Its analytic-cycle coefficient is locally
constant and equals $\varkappa_i$ on such a patch. This identifies
every local branch coefficient, including those at singular points,
so
$(\mathfrak Y^{\mathrm{an}},p)=(\mathfrak M^\C,p)$.
Uniqueness follows by evaluating on a good patch of each $Y_i$.

\emph{Sections at real phases.}
We now include the case $c=0$ again.
For every real phase $t\in\R^n$,
\[
 \{x\in\C^d:\exp(2\pi i(t+Lx))\in Y\}=\supp\mu_t.
\]
Indeed, if $z=\exp(2\pi i(t+Lx))\in Y$, then
$\Log z=-2\pi L\operatorname{Im}x\in V$.
For $c>0$, Lemma~\ref{lem:moment-amoeba} forces $\Log z=0$;
for $c=0$, the same conclusion follows from $Y\subset\T^n$.
Since $L$ has rank $d$, $\operatorname{Im}x=0$.
Lemma~\ref{lem:moment-fiber-identification} gives $Y(\T)=M$
(the equality is immediate when $c=0$), so $t+Lx\in M$.
Covariance gives $\mu_{t+Lx}=\tau_{-x}\mu_t$, hence
$x\in\supp\mu_t$. Conversely, this covariance shows that
$x\in\supp\mu_t$ implies $t+Lx\in M=Y(\T)$.
By Theorem~\ref{thm:all-fq-toral-cycles}(e), the weighted local
intersection degree of $\mathfrak M^\C$ at a physical root is its
mass. The equality of the weighted germs of $\mathfrak Y$ and
$\mathfrak M^\C$ at $t+Lx$ therefore gives the claimed mass identity
at every real phase and root.
Combining the set and mass identities yields
$\nu^{\mathfrak Y}_t=\mu_t$ for every real phase $t$.

Reciprocal conjugation $\sigma(z)=1/\bar z$ fixes $M$ and sends
Zariski closed sets to Zariski closed sets. Hence $\sigma(Y)=Y$,
and $\sigma$ permutes the components $Y_i$. Each $Y_i$ contains a
Zariski-dense patch of torus points, fixed pointwise by reciprocal
conjugation. Its reciprocal-conjugate component contains that same
patch and therefore equals $Y_i$. For $c=0$ this follows from the
pointwise invariance already observed. The cycle is accordingly
invariant.
\end{proof}

\subsection{Smoothness and real equations}
We next use the description of the fibers along $M$ to prove
smoothness for uniformly discrete supports. Amoeba avoidance then
provides real equations with frequencies in $\Gamma$.

\begin{lemma}[Smoothness for uniformly discrete support]
\label{lem:integer-ud-smoothness}
For an integer-mass FQ measure with uniformly discrete support, $M^\C$
is smooth at every point $p\in M$, with
$T_pM^\C\oplus V_\C=\C^n$. Consequently,
$Y$ is smooth along $Y(\T)=M$.
\end{lemma}
\begin{proof}
Let the original support have separation at least $a_0>0$. The supports
of all hull measures have the same separation: two distinct limit atoms
at distance less than $a_0$ would, by positive test functions in disjoint
small balls, force two distinct atoms at such a distance in nearby
translates.

Fix $p\in M$ and isolate zero in a ball of diameter less than $a_0$.
If its mass is $s_0$, Lemma~\ref{lem:ac-splitting} yields a
holomorphic local cycle of degree $s_0$. At every nearby real phase
it is nonempty with at most one support point, hence equals
$s_0[\mathfrak x(t)]$. Its first moment divided by $s_0$ defines a
holomorphic $\mathfrak x(\theta)$ on the complex phase neighborhood. Every moment of the cycle agrees with the corresponding
moment of $s_0[\mathfrak x(\theta)]$ on the real neighborhood; the
identity theorem makes these equalities hold for complex phases too.
Since moment coordinates separate cycles, the local cycle has that form
throughout the complex neighborhood.

Covariance gives
$\mathfrak x(\theta+Lu)=\mathfrak x(\theta)-u$, so
$D\mathfrak x(\theta)L=-I_d$. Hence its zero locus $M^\C$ is smooth
at $p$ with $T_pM^\C\oplus V_\C=\C^n$.
By Corollary~\ref{cor:universal-linear-height}, the Fourier cone
hypothesis of Lemma~\ref{lem:integer-germ-equality} holds. That
lemma identifies this germ with $Y$ at every torus point.
\end{proof}

\begin{proposition}[Real equations in the spectral group]\label{cor:canonical-real-equations}
Let $\mu$ be an integer-mass FQ measure, $\Lambda=\supp\mu$, and
$Y=\overline M^{\mathrm{Zar}}$ for its canonical set. Suppose that
$Y(\T)=M$ and $\cA(Y)\cap V\subset\{0\}$.
Then $E_L^{-1}(Y)=\Lambda$ as subsets of $\C^d$. Let $I(Y)$ denote
the vanishing ideal of $Y$ in the Laurent polynomial ring. This ideal
has finitely many Laurent generators real-valued on $\T^n$.
Their compositions with $E_L$ are real-valued on $\R^d$, have
frequencies in $\Gamma$, and have common complex zero set $\Lambda$.
\end{proposition}

\begin{proof}
The assumption $Y(\T)=M$ and the definition of $M$ give the first
equality below; covariance in
Theorem~\ref{thm:all-fq-toral-cycles}(a) gives the second:
\[
 \{x\in\R^d:E_L(x)\in Y\}
 =\{x:0\in\supp\mu_{Lx}\}
 =\{x:0\in\supp\tau_{-x}\mu\}=\Lambda.
\]
If $E_L(x)\in Y$ for $x\in\C^d$, then
$\Log E_L(x)=-2\pi L\operatorname{Im}x\in V$.
Amoeba avoidance and the rank of $L$ force $\operatorname{Im}x=0$,
so the section identity holds over $\C^d$.

Reciprocal conjugation fixes $M$ and hence preserves $Y$. Choose
Laurent generators $P_1,\ldots,P_q$ of $I(Y)$ and put
\[
 P^\dagger(z)=\sum_k\overline{a_k}z^{-k}
 \quad\text{when }P(z)=\sum_k a_kz^k.
\]
Then $P_j^\dagger\in I(Y)$, and
\[
 P_j^{\mathrm{re}}=\tfrac12(P_j+P_j^\dagger),\qquad
 P_j^{\mathrm{im}}=\tfrac1{2i}(P_j-P_j^\dagger)
\]
are real-valued on $\T^n$. They vanish on $Y$ and generate the same
ideal, since they belong to $I(Y)$ and
$P_j=P_j^{\mathrm{re}}+iP_j^{\mathrm{im}}$ recovers each original
generator. Composition with $E_L$ gives frequencies in
$L^t\Z^n=\Gamma$ and the asserted zero set.
\end{proof}

\begin{proof}[Proof of Theorem~\ref{thm:branched-scalar-cone}]
For $c>0$, all sufficiently close rational planes are admissible by
Lemma~\ref{lem:admissible-projections}.
Lemmas~\ref{lem:moment-fiber-identification} and~\ref{lem:moment-amoeba}
prove the finite quotient and amoeba assertions, with $Y(\T)=M$ by
Lemma~\ref{lem:moment-fiber-identification}. The reduced germs agree
by Lemma~\ref{lem:integer-germ-equality}, and the canonical effective
cycle and its mass identities at every real phase follow from
Lemma~\ref{lem:integer-global-coefficients}.
If $c=0$, $M$ is finite and lies in the compact torus, so its Zariski
closure is $M$ itself. The quotient to a point is finite and its
amoeba is $\{0\}$. Give each point $p$ the coefficient
$\mu_p(\{0\})$; Proposition~\ref{prop:ac-canonical-germ}
gives the same germ and mass assertions directly.
Lemma~\ref{lem:integer-global-coefficients} also gives invariance
and the section identity at every real phase in this case.

\end{proof}

\section{Compression to Lee--Yang varieties}
\label{sec:compression}

A basis matrix of the spectral group need not have positive maximal
minors. We take an isogeny preimage for which the matrix has positive
maximal minors and the amoeba avoids every nonnegative plane,
equivalently the ordinary sign-variation condition by
\cite[Lemma~3.2]{AKKV}. The construction follows \cite[\S8, especially
Lemmas~8.1 and~8.3]{AKKV}; taking the full inverse image preserves
the sections and their local intersection degrees.

We say that $Y$ has \emph{robust avoidance at $V$} if there is an open
neighborhood $U$ of $V$ in $\Gr(d,n)$ such that
$\cA(Y)\cap W\subseteq\{0\}$ for every $W\in U$.

\begin{lemma}[Rational compression]
\label{lem:compression}
Let $L\in\R^{n\times d}$ have rank $d$, and let $U$ be an open
neighborhood of $V=L\R^d$ in $\Gr(d,n)$. There is a nonsingular
integer matrix $S$ such that $S^{-1}L$ has positive maximal minors and
$SW\in U$ for every $W\in\Gr_{\ge0}(d,n)$.
\end{lemma}
\begin{proof}
Choose an $n\times d$ positive Vandermonde matrix and write its QR
factorization with positive diagonal in the upper triangular factor.
The orthonormal factor $Q_+$ has positive maximal minors because the
right change of basis has positive determinant. Put
$V_+=Q_+\R^d$ and choose an invertible real linear map $G$ with
$GL=Q_+$. If $W\in\Gr_{\ge0}(d,n)$ has basis matrix $F$ with
nonnegative maximal minors, Cauchy--Binet gives
\[
 \det(Q_+^tF)=\sum_{|I|=d}\det((Q_+)_I)\det(F_I)>0.
\]
Thus $W$ is the graph of a map $D_W:V_+\to V_+^\perp$,
and compactness of the nonnegative Grassmannian bounds these maps
uniformly. Let $C_t$ be the identity on $V_+$ and multiplication by
$t$ on $V_+^\perp$. Then
$C_tW=\operatorname{graph}(tD_W)\to V_+$ uniformly as $t\downarrow0$.
Consequently $G_t=C_t^{-1}G$ satisfies $G_tL=Q_+$ and, for small $t$,
\[
 G_t^{-1}\Gr_{\ge0}(d,n)\subset U.
\]
By compactness of the nonnegative Grassmannian and openness of $U$,
this containment persists under small perturbations of $G_t$, as does
positivity of the finitely many maximal minors of $G_tL$. Choose a nonsingular rational matrix $H$
in this neighborhood and a positive integer $h$ for which
$S=hH^{-1}$ is integral. Then $S^{-1}L=h^{-1}HL$ has positive maximal
minors, and $SW\in U$ for every nonnegative plane $W$.
\end{proof}

\begin{proposition}[Robust avoidance and Lee--Yang representation]
\label{prop:compression}
Let $Y\subset(\C^*)^n$ be reduced and pure of dimension $n-d$,
invariant under $z\mapsto1/\overline z$, and let $L$ have rank $d$.
Suppose that $Y$ has robust avoidance at $L\R^d$. There are a
nonsingular integer matrix $S$, a Lee--Yang variety
$X\subset(\PP^1)^n$, and a matrix $\widetilde L=S^{-1}L$ with positive
maximal minors. Here $X$ is the closure of $f_S^{-1}(Y)$ in $(\PP^1)^n$, where
$f_S(w)_i=\prod_jw_j^{S_{ij}}$. For every
$\theta'\in\R^n$,
\begin{align*}
 &\{x\in\C^d:\exp(2\pi i(S\theta'+Lx))\in Y\}\\
 &\qquad=\{x\in\C^d:\exp(2\pi i(\theta'+\widetilde Lx))\in X\},
\end{align*}
with equal local intersection degrees.
Moreover, $X$ satisfies $\varn(\log|z|)\ge d$ for every
$z\in X\setminus\T^n$. If $Y$ is smooth at every point of $Y(\T)$,
then $X$ is strict. At each physical root the finite torus isogeny
preserves the local intersection degree, smoothness and transversality.

If an effective invariant cycle $\mathfrak Y=\sum_i\varkappa_iY_i$
is given on $Y$, its component coefficients satisfy
$\varkappa_i\in\Z_{>0}$, and invariance means that components
interchanged by reciprocal conjugation have equal coefficients.
Give each irreducible component of $f_S^{-1}(Y_i)$ the coefficient
$\varkappa_i$, and take its projective closure. The
resulting cycle $\mathfrak X$ is effective Lee--Yang, and the section
identities hold at every phase with weighted local intersection degrees.
If $Y$ is smooth on $Y(\T)$, the reduced support of the resulting cycle
is strict.
\end{proposition}
\begin{proof}
\emph{The pullback and its amoeba.}
Choose $S$ from Lemma~\ref{lem:compression} and put
\[
 f_S(w)_i=\prod_jw_j^{S_{ij}},\qquad
 Y'=f_S^{-1}(Y),\qquad \widetilde L=S^{-1}L.
\]
Smith normal form shows that $f_S$ is finite, surjective and
\'etale, so $Y'$ is reduced, pure of dimension $n-d$, and invariant
under reciprocal conjugation. The identities
\[
 f_S(E_{\widetilde L}(x))=E_L(x),\qquad
 \cA(Y')=S^{-1}\cA(Y)
\]
follow from the monomial formula. For the reverse amoeba inclusion,
lift a point of $Y$ under $f_S$. Real invertibility of $S$ determines
the logarithmic moduli of the lift, so $Y'$ avoids every plane in
$U'=\{W:SW\in U\}$, an open set containing $\Gr_{\ge0}(d,n)$.

\emph{Points in the algebraic torus.}
For a nonzero vector $v$ with $\varn(v)<d$,
\cite[Lemma~3.2]{AKKV} gives a nonnegative plane containing $v$.
It follows that every point of $Y'\setminus\T^n$ satisfies
$\varn(\Log z)\ge d$. If $Y$ is smooth along $Y(\T)$, then $Y'$ is
smooth along $Y'(\T)$: a torus point maps to a torus point, and the
\'etale map is a local analytic isomorphism.

\emph{The projective boundary.}
Let $X$ be the closure of $Y'$ in $(\PP^1)^n$. Suppose a boundary
point violates the ordinary sign-variation condition, and choose
$z_j\in Y'$ tending to it. Then $\|\Log z_j\|\to\infty$. After
passing to a subsequence, the unit vectors
$a_j=\Log z_j/\|\Log z_j\|$ tend to a unit vector $a$.
Every nonzero coordinate of $a$ has the sign of an infinite coordinate
of the limiting logarithmic modulus; deleting the other coordinates
cannot increase sign variation. Hence $\varn(a)<d$.
Choose a nonnegative plane containing $a$, complete $a$ to a basis,
and replace $a$ in that basis by $a_j$. For large $j$ the resulting
planes lie in $U'$ and contain $\Log z_j$, contrary to avoidance.

\emph{The closure and the sections.}
Closure preserves purity and involution invariance and adds no points
in $(\C^*)^n$, where $Y'$ is closed. Hence $X$ is Lee--Yang and has
the same local germs as $Y'$ along the torus; it is strict when $Y$
is smooth along $Y(\T)$. For every real phase $\theta'$,
the monomial formula gives
\[
 f_S\bigl(\exp(2\pi i(\theta'+\widetilde Lx))\bigr)
 =\exp(2\pi i(S\theta'+Lx)),
\]
proving equality of the complex sections. Near each physical root,
$f_S$ is a biholomorphism carrying the two local varieties and their
section parametrizations to one another. It therefore preserves local
intersection degree, smoothness and transversality.

\emph{Component coefficients.}
Each irreducible preimage component lies over a unique $Y_i$ and
receives its coefficient $\varkappa_i$. The map $f_S$ is finite,
surjective and \'etale, so these coefficients are unchanged. Reciprocal
conjugation commutes with $f_S$ and preserves them. The reduced
preimage and its closure satisfy the geometric properties above,
including ordinary sign variation at the boundary. Closure leaves the
germs at torus points unchanged. Near a physical root, the local
biholomorphism identifies each branch over $Y_i$ and preserves its
coefficient and local intersection degree. Summing over branches proves
the weighted identities at every real phase; preservation of smoothness
along the torus proves strictness of the support.
\end{proof}

\begin{remark}\label{rem:singular-isogeny}
The inverse image is needed for equality of sections. Take $n=d=1$,
$Y=\{1\}$, $S=(2)$, $L=(2)$ and $\widetilde L=(1)$.
The section of $Y$ by $L$ is $\frac12\Z$. Although
$f_S(\{1\})=Y$, the section of $\{1\}$ by $\widetilde L$ is $\Z$.
The full inverse image $f_S^{-1}(Y)=\{\pm1\}$ has section
$\frac12\Z$ under $\widetilde L$.
The compression follows \cite[\S8]{AKKV}; taking the full inverse
image preserves the physical section.
\end{remark}

\section{Classification}
\label{sec:classification}

We combine reconstruction and compression to obtain the representation
and real equations of Theorem~\ref{thm:main-equations}. The same
construction proves the representation in
Theorem~\ref{thm:main-delone}; the converse follows from
Proposition~\ref{prop:integer-forward} and \cite[Theorem~2.3]{AKKV}.
We begin in the torus determined by the spectral group.

\begin{theorem}[Classification by effective Lee--Yang cycles]
\label{thm:universal-classification}
An integer-mass measure $\mu$ on $\R^d$ is a Fourier quasicrystal measure
if and only if there are an effective Lee--Yang cycle $\mathfrak X$ of
codimension $d$ in $(\PP^1)^n$ for some $n\ge d$ and
$\widetilde L\in\R^{n\times d}$
with positive maximal minors such that
\begin{equation}
 \mu=\nu^{\mathfrak X,\widetilde L}_0.
 \label{eq:universal-ly-section}
\end{equation}
In the inverse direction the cycle can be chosen with all components
active and with reduced support satisfying
\[
 \varn(\log|z|)\ge d
 \qquad(z\in|\mathfrak X|\setminus\T^n),
\]
including at the projective boundary.
If $\mu$ is an FQ measure, choose a basis matrix $L$ for its spectral
group $\Gamma$, let $M$ be the canonical set, and put
$Y=\overline M^{\mathrm{Zar}}$. In these spectral coordinates,
$E_L^{-1}(Y)=\supp\mu$ in $\C^d$, and $\supp\mu$ is the common
complex zero set of finitely many real-valued trigonometric polynomials
with frequencies in $\Gamma$.
\end{theorem}

\begin{proof}
\emph{The canonical variety.}
Suppose first that $\mu$ is an integer-mass FQ measure and put
$\Lambda=\supp\mu$. By Theorem~\ref{thm:all-fq-toral-cycles}, the
spectral group is finitely generated and the hull has the torus
parametrization described there. Choose a basis matrix $L$ and put
$V=L\R^d$. Then
\[
 M=\{\theta\in\T^n:0\in\supp\mu_\theta\},\qquad
 Y=\overline M^{\mathrm{Zar}}\subset(\C^*)^n.
\]
The matrix $L$ has rank $d$. The flux measure $m$ has support $M$,
and its Fourier coefficients are the spectral coefficients in their
unique integer coordinates. By
Corollary~\ref{cor:all-fq-algebraicity}, $Y$ is reduced and pure of
dimension $n-d$.

\emph{Height and reconstruction.}
Put $E=\{k:\widehat m(k)\ne0\}$. By \eqref{eq:HF-4.2},
$\widehat m(k)\ne0$ if and only if $L^tk\in\Lambda'$; hence
Corollary~\ref{cor:universal-linear-height} gives
\eqref{eq:classification-height}:
\begin{equation}
 \|k\|\le C(1+\|L^tk\|)\qquad(k\in E).
 \label{eq:classification-height}
\end{equation}

Since $L^t$ vanishes on $V^\perp$ and is invertible on $V$,
\eqref{eq:classification-height} implies the Fourier cone hypothesis
of Theorem~\ref{thm:branched-scalar-cone}. That theorem yields
\begin{equation}
 Y(\T)=M,\qquad
 \|P_{V^\perp}\Log z\|\ge\varepsilon\|\Log z\|
 \quad(z\in Y)
 \label{eq:universal-envelope}
\end{equation}
for some $\varepsilon>0$.

Put $U=\{W\in\Gr(d,n):\|P_W-P_V\|<\varepsilon\}$. If
$0\ne a\in\cA(Y)\cap W$ for $W\in U$, then
\[
 \|P_{V^\perp}a\|=\|(P_W-P_V)a\|<\varepsilon\|a\|,
\]
contrary to \eqref{eq:branched-amoeba}. Thus $Y$ has robust avoidance
at $V$. Reciprocal conjugation fixes $M$ and therefore preserves $Y$.
Proposition~\ref{cor:canonical-real-equations} gives
$E_L^{-1}(Y)=\Lambda$ over $\C^d$.

\emph{Local degrees and compression.}
Lemma~\ref{lem:integer-germ-equality} identifies the reduced analytic
germ of $Y$ with $M^\C$ at every point of $M$.
By Lemma~\ref{lem:integer-global-coefficients}, unique positive
integer coefficients $\varkappa_i$ define a cycle
$\mathfrak Y=\sum_i\varkappa_iY_i$ with germ $\mathfrak M^\C$ there.
Its section at every real phase is the hull measure, and reciprocal
conjugation preserves each component. Reduced purity and the
neighborhood $U$ now verify the hypotheses of
Proposition~\ref{prop:compression}. The isogeny preimage produces
$\mathfrak X$ and $\widetilde L$ in \eqref{eq:universal-ly-section},
preserving weighted local intersection degrees and satisfying ordinary
sign variation through the projective boundary.
Delete any inactive components; Lemma~\ref{lem:integer-active-components}
shows that sections and masses are unchanged. If $\Lambda$ is
uniformly discrete, then $Y$ is smooth along $Y(\T)$ by
Lemma~\ref{lem:integer-ud-smoothness}, so
Proposition~\ref{prop:compression} yields strict reduced support.

\emph{Real equations.}
By Proposition~\ref{cor:canonical-real-equations}, the real-valued
trigonometric equations have frequencies in $\Gamma$. The Laurent
generators constructed there have the same complex zero set as $I(Y)$,
and their compositions have zero set $E_L^{-1}(Y)=\Lambda$.

For the converse, Proposition~\ref{prop:integer-forward} shows that
the section of an effective Lee--Yang cycle is an integer-mass FQ
measure whenever it is nonzero.
\end{proof}

\section{Singular Lee--Yang sections}
\label{sec:singular}

We extend \cite[Theorem~2.3]{AKKV} to singular Lee--Yang varieties
by counting roots with their local intersection degrees. Conservation
of number makes the section measures continuous in the phase, while
a holomorphic trace replaces the unramified homological pullback of
\cite[\S4.5]{AKKV} in the Fourier argument. We first prove the summation
formula under ordinary sign variation and then apply
Proposition~\ref{prop:compression} to the general case. Together with
the inverse construction, this proves the classification for effective
cycles; the final example explains why the degrees are necessary.

\begin{theorem}[Singular Lee--Yang sections]
\label{thm:singular-section}
Let $1\le d\le n$, let $c=n-d$, let $X\subset(\PP^1)^n$ be a
Lee--Yang variety of codimension $d$, and let $L\in\R^{n\times d}$
have positive maximal minors. For $\theta\in\T^n$, put
\[
 \Lambda_\theta=\{x\in\C^d:
       \exp(2\pi i(\theta+Lx))\in X\}.
\]
Then the following assertions hold.
\begin{enumerate}
\item The set $\Lambda_\theta$ is real and locally finite. Its roots
have positive local intersection degrees $\deg_\theta(x)$, equal to
one whenever the image of the root is a smooth point of $X$.
\item The measures
\[
 \nu_\theta=\sum_{x\in\Lambda_\theta}\deg_\theta(x)\delta_x
\]
form a vaguely continuous family.
The degrees are those of Definition~\ref{def:local-intersection-degree};
their relation to local ring length is given in
Remark~\ref{rem:degree-length}.

The family is uniformly translation bounded. In particular, for a
constant $C$ independent of $\theta,a,R$,
\[
 \nu_\theta(a+B_R)\le C(1+R)^d
 \qquad(a\in\R^d,\ R>0).
\]
\item Each $\nu_\theta$ is a Fourier quasicrystal measure. Its spectrum
is contained in $L^t\Z^n$, and its spectral total variation in $B_R$
is $O((1+R)^n)$.
\end{enumerate}

If $X$ satisfies the stronger condition
\begin{equation}
\label{eq:strong-var-singular}
 z\in X\setminus\T^n
 \quad\Longrightarrow\quad
 \varn(\log|z|)\ge d,
\end{equation}
there is a finite positive measure $\sigma_{X,L}$ on $X(\T)$ such that
\begin{equation}
\label{eq:singular-summation}
 \sum_{x\in\Lambda_\theta}\deg_\theta(x)\widehat f(x)
 =
 \sum_{\substack{k\in\Z^n\\\varn(k)<d}}
 \widehat \sigma_{X,L}(k)e^{2\pi i\langle k,\theta\rangle}f(L^tk)
 \qquad(f\in\mathcal S(\R^d)).
\end{equation}
Both sums converge absolutely.  The spectral total variation in
$B_R$ is $O((1+R)^n)$.
\end{theorem}

\begin{remark}\label{rem:singular-degree-role}
Local intersection degrees obey conservation of number and therefore
continue through collisions of roots. They can differ from local ring
lengths, as explained in Remark~\ref{rem:degree-length}.
Counting each root once can also destroy the Fourier quasicrystal
property; see Example~\ref{ex:two-circles} and
Proposition~\ref{prop:singular-example}.
\end{remark}

\subsection{Local degree and conservation of number}

We realize the local degrees as degrees of finite incidence maps.
Conservation of number will then imply continuity in the phase, which
is needed to recover singular sections from the averaging formula.
\begin{lemma}[Local degrees and continuity]
\label{lem:singular-degree}
Under the hypotheses of Theorem~\ref{thm:singular-section}, all the
complex sections $\Lambda_\theta$, for real $\theta$, are real and
locally finite.  They have positive local intersection degrees
$\deg_\theta(x)$ for which $\theta\mapsto\nu_\theta$ is vaguely continuous
and uniformly translation bounded.
\end{lemma}

\begin{proof}
\emph{Reality and discreteness.}
Put $V=L\R^d$ and write
\[
 \mathcal Y=\{y\in\C^n:\exp(2\pi iy)\in X\}.
\]
For a real lift of $\theta$,
\[
 \Log\exp(2\pi i(\theta+Lx))=-2\pi L\operatorname{Im}x.
\]
A nonzero vector in $L\R^d$ has upper variation at most $d-1$ by
\cite[Lemma~3.2]{AKKV}. The Lee--Yang condition forces
$L\operatorname{Im}x=0$; since $L$ has full rank, $x\in\R^d$
(cf.\ \cite[Proposition~4.4]{AKKV}).

A positive-dimensional analytic germ contains a nonconstant
holomorphic curve, whereas a holomorphic map with real-valued
coordinates is constant. An analytic subset of a totally real linear
space must therefore be zero-dimensional. Each $\Lambda_\theta$ is
closed and analytic in $\C^d$, hence locally finite.

\emph{Finite incidence maps and local degrees.}
Fix $(\theta_0,x_0)$ with $\theta_0$ real and $x_0\in\Lambda_{\theta_0}$.
Consider the local incidence variety
\[
 \mathcal F=
 \{(\theta,x)\in\C^n\times\C^d:\theta+Lx\in \mathcal Y\},
 \qquad p(\theta,x)=\theta.
\]
The map $(y,x)\mapsto(y-Lx,x)$ identifies $\mathcal F$ with
$\mathcal Y\times\C^d$.  Hence $\mathcal F$ is reduced and of pure dimension
$c+d=n$.  The fiber of $p$ is isolated at $(\theta_0,x_0)$.
The local finite-mapping theorem provides a small isolating
neighborhood $U_x$ of $x_0$ and a connected neighborhood $U$ of
$\theta_0$ such that
\[
 p:\mathcal F\cap(U\times U_x)\longrightarrow U
\]
is proper and finite. Choose $U_x$ with boundary disjoint from the
central fiber and shrink $U$. Over the resulting neighborhood $U$,
the incidence misses $U\times\partial U_x$, so no root can enter
or leave $U_x$ as the phase varies in $U$. This also proves properness
directly.
Every nonempty component dominates $U$. By Remmert's theorem
\cite[II, Theorem~8.8]{Demailly}, its image is analytic; finite fibers
preserve its dimension $n$, so the image contains a neighborhood in
the smooth base.

We define $\deg_{\theta_0}(x_0)$ as the positive generic number of
points of this finite map. To identify this degree with the local
intersection degree, choose a real complement $W$ of $L\R^d$ and
complexify the decomposition.  If
$y=y_V+y_W$, the source coordinate change
$u=y_V-Lx$ changes $p$ to
\[
 (u,y)\longmapsto(u,y_W).
\]
The local degree is therefore that of the projection
$\mathcal Y\to W_\C$ with kernel $L\C^d$, as in
Definition~\ref{def:local-intersection-degree}. Two choices of complement
give equivalent target coordinates and hence the same degree.

Let $y_0\in\mathcal Y$ be real with $\exp(2\pi iy_0)$ a smooth point
of $X$, and suppose $0\ne w\in T_{y_0}\mathcal Y\cap L\C^d$.
Choose $s_0\in\{1,i\}$ with $v=\operatorname{Im}(s_0w)\ne0$; then
$v\in L\R^d$. A holomorphic arc $\gamma(s)=y_0+sw+O(s^2)$ in
$\mathcal Y$ gives points
$z_\varepsilon=\exp(2\pi i\gamma(\varepsilon s_0))\in X$ with
$\Log|z_\varepsilon|=-2\pi(\varepsilon v+O(\varepsilon^2))\ne0$.
Since $\vbar$ is upper semicontinuous and invariant under nonzero
scalings, and $\vbar(v)\le d-1$ by \cite[Lemma~3.2]{AKKV}, we get
$\vbar(\Log|z_\varepsilon|)\le d-1$ for small $\varepsilon>0$,
contradicting the Lee--Yang condition. Hence
$T_{y_0}\mathcal Y\cap L\C^d=0$, and the inverse function theorem gives
degree one.

\emph{Conservation of number.}
Apply Lemma~\ref{lem:local-degree-general} to the finite projection
$p:\mathcal F\cap(U\times U_x)\to U$. If its degree is $q_0$, then
\[
 \sum_{x\in\Lambda_\theta\cap U_x}\deg_\theta(x)=q_0
\]
for every nearby real phase.

\emph{Continuity and translation bounds.}
Let $\phi$ be a continuous function supported in $B_R$, and let
$x_1,\ldots,x_s$ be the points of
$\Lambda_{\theta_0}\cap\overline{B_{R+1}}$. Choose pairwise disjoint
isolating balls $U_j$ around these points. The compact set
\[
 K=\{x:|x|\le R+1\}\setminus\bigcup_j U_j
\]
misses the fiber of $\mathcal F$ over $\theta_0$. Since $\mathcal F$
is closed, it also misses every sufficiently nearby fiber. For real
phases all roots remain real, and the roots in each $U_j$ converge to
$x_j$. Conservation of number keeps their total local intersection
degree fixed. Uniform continuity of $\phi$ therefore gives
$\nu_\theta(\phi)\to\nu_{\theta_0}(\phi)$, proving vague continuity.

The incidence set with $\theta\in\R^n/\Z^n$ and
$x\in\overline{B_1}$ is compact.  Finitely many of the preceding
finite-map neighborhoods bound the total multiplicity of its fibers
uniformly.  Thus
\[
 \sup_\theta\nu_\theta(B_1)<\infty.
\]
The identity
\[
 \nu_\theta(a+E)=\nu_{\theta+La}(E)
\]
holds with multiplicities, since translation in $x$ preserves the
incidence germ. The unit-ball estimate now implies uniform translation
boundedness and the stated $O((1+R)^d)$ bound. Continuity of
$\theta\mapsto\nu_\theta(h)$ for every Schwartz function $h$ follows
by compactly supported approximation, with uniformly summable tails.
\end{proof}

By Lemma~\ref{lem:singular-degree}, if $X(\T)$ is empty, every
real-phase section is empty. All assertions then hold with the zero
measure and empty spectrum. We may therefore assume
$X(\T)\ne\varnothing$ in the arguments using amoeba avoidance.

\begin{remark}[The zero-dimensional case]
\label{rem:singular-c-zero}
If $c=0$, upper variation is at most $n-1$ on every nonzero vector,
so $X$ is a finite subset of $\T^n$.  The matrix $L$ is square
and invertible, and the section is a finite union of distinct
cosets of $L^{-1}\Z^n$.  Every local intersection degree is one.
Ordinary Poisson summation proves
Theorem~\ref{thm:singular-section} in this case.  The remaining
argument assumes $c\ge1$.
\end{remark}

\subsection{Torus measures and a holomorphic trace}

The phase average is computed from the Jacobian of
$(t,x)\mapsto t-Lx$ on the smooth real locus. We express this Jacobian
as a positive measure and prove Fourier vanishing by taking the trace
of a holomorphic form over a solid torus. Throughout this subsection
we assume \eqref{eq:strong-var-singular} and put $N=X_{\rm sm}(\T)$.
For $I=\{i_1<\cdots<i_c\}\subset[n]$, write
\[
 d\theta_I=d\theta_{i_1}\wedge\cdots\wedge d\theta_{i_c},
 \qquad
 s(I)=(-1)^{\,i_1+\cdots+i_c-c(c+1)/2}.
\]
We use angular coordinates $z_j=e^{2\pi i\theta_j}$ and
write $\pi_I:X\to(\PP^1)^I$ for coordinate projection.

\begin{lemma}[Positive torus measures]
\label{lem:singular-torus-measure}
Assume \eqref{eq:strong-var-singular}. The forms $s(I)d\theta_I$
induce the same orientation on $N$.
With this orientation, the measures
\[
 \sigma_I(E)=\int_{E\cap N}s(I)\,d\theta_I
\]
are finite and positive on $X(\T)$ and give zero mass to
$X_{\rm sing}(\T)$. They satisfy
$\sigma_I(N)=\deg(\pi_I)$, where the degree is summed over the
dominant components of the coordinate projection. The measure
\[
 \sigma_{X,L}=\sum_{\substack{I\subset[n]\\|I|=c}}L_{I^c}\sigma_I
\]
has, on an oriented tangent frame $F_N$, the positive density
$\det[F_N,L]$. Here $L_J$ is the maximal minor on rows $J$.
\end{lemma}

\begin{proof}
For the coordinate projection $\pi_I:X\to(\PP^1)^I$, prescribing $c$
moduli to be one leaves at most $d$ nonzero logarithmic moduli, with
at most $d-1$ sign changes. The strong condition therefore implies
\begin{equation}
 \pi_I^{-1}(\T^I)=X(\T).
 \label{eq:singular-real-fiber}
\end{equation}
This is the ordinary-sign-variation argument of
\cite[Lemma~4.7]{AKKV}. The fibers over $\T^I$ are compact analytic
subsets of a totally real torus, and are therefore finite. The
projection is consequently finite over a Zariski open neighborhood
$U_I$ of the real torus. The locus of finite fibers is open and
invariant under the involution, so we choose $U_I$ invariant as well.
Each component of $X$ meeting $\pi_I^{-1}(U_I)$ has finite fibers
over $U_I$ and hence maps onto $U_I$.

After coordinatewise Cayley transformations,
$\pi_I:\pi_I^{-1}(U_I)\to U_I$ is a finite surjective real-fibered
morphism as defined in \cite[Definition~2.1]{KSreal}. At each point
of $N$ in $X(\T)$, source and target are smooth, so the differential
is an isomorphism by \cite[Theorem~2.19]{KSreal}; compare
\cite[Lemma~4.18]{AKKV}. It follows that every $d\theta_I$ is
nonvanishing on $N$.

By \cite[Lemma~4.17]{AKKV}, the conormal plane at each point of $N$
belongs to the nonnegative Grassmannian.  The complementary Pl\"ucker
coordinates give, for an oriented frame $F_N$ with
$\det((F_N)_{[c]})>0$,
\[
 \operatorname{sgn}\det((F_N)_I)=s(I).
\]
Thus all the forms $s(I)d\theta_I$ are positive, as in
\cite[Remark~4.20]{AKKV}.

We compute the mass of $\sigma_I$ by change of variables under
$\pi_I$. Off the proper algebraic set of branch and singular values,
each fiber has the finite generic degree of $\pi_I$, and all points
over the real torus are real. The torus $\T^I$ is maximal totally real,
so a nonzero polynomial restricts to a nonzero real-analytic function
on this connected torus. Its zero set has Lebesgue measure zero.
The exceptional set therefore has measure zero in $\T^I$, and its
preimage in $N$ has measure zero because the projection is a local
diffeomorphism there.  Thus
$\sigma_I(N)=\deg(\pi_I)$, summing the degrees of dominant components.
A nondominant component has no smooth real point, since the
projection is unramified at every such point.  This proves
finiteness.

Finally, $L_{I^c}>0$.  On an oriented tangent frame $F_N$ of $N$,
expansion by the first $c$ columns gives
\[
 \det[F_N,L]
 =
 \sum_{|I|=c}s(I)\det((F_N)_I)L_{I^c}>0.
\]
The density of $\sigma_{X,L}$ is thus the absolute Jacobian of
$(t,x)\mapsto t-Lx$, with respect to the tangent frame and the
physical coordinates.
\end{proof}

The trace of a top holomorphic form extends across critical values,
allowing the Fourier argument to include singular torus loci.

\begin{lemma}[Trace of a top holomorphic form]
\label{lem:singular-trace}
Let $\varpi:\widetilde Z\to U$ be a proper holomorphic map between
smooth complex manifolds of the same dimension $c$.  If $\omega$
is a holomorphic $c$-form on $\widetilde Z$, its proper pushforward
$\operatorname{Tr}_{\varpi}\omega=\varpi_*\omega$ is a holomorphic
$c$-form on $U$. Outside the critical values of $\varpi$, it is the
sum of the pullbacks by the local inverse branches.
\end{lemma}

\begin{proof}
For a compactly supported smooth test form $\eta$ of type $(0,c)$,
define
\[
 \langle \varpi_*\omega,\eta\rangle
 =
 \int_{\widetilde Z}\omega\wedge \varpi^*\eta.
\]
The integral has compact support by properness and defines a current
of type $(c,0)$ \cite[I, Theorem~2.14(b)]{Demailly}. Moreover,
\[
 \overline\partial(\varpi_*\omega)
 =
 \varpi_*(\overline\partial\omega)=0.
\]
Its scalar distributional coefficient in local coordinates satisfies
the Cauchy--Riemann equations and is therefore holomorphic
\cite[I, Lemma~3.29(a)]{Demailly}.
The inverse-branch formula follows by change of variables where
$\varpi$ is a covering map.  Components with lower-dimensional image
contribute zero, while properness defines the current pushforward
also across positive-dimensional exceptional fibers.
\end{proof}

\begin{lemma}[Vanishing Fourier coefficients]
\label{lem:singular-vanishing}
Under \eqref{eq:strong-var-singular},
\[
 \widehat \sigma_{X,L}(k)=0
 \qquad\text{for every }k\in\Z^n\text{ with }\varn(k)\ge d.
\]
\end{lemma}

\begin{proof}
We trace the integrand to a holomorphic top form on a solid torus.
Stokes' theorem then gives the required vanishing.

\emph{The solid torus and the sign pattern.}
Fix such $k$. Choose
$I_0=\{j_0<\cdots<j_d\}$ on which the entries of $k$ are nonzero
and alternate in sign.  Choose $i_*\in I_0$ with $k_{i_*}<0$, and set
\[
 J_0=([n]\setminus I_0)\cup\{i_*\},\qquad |J_0|=c.
\]
In the target of $\pi_{J_0}$ let
\[
 C_{J_0}=
 \{z\in(\PP^1)^{J_0}: |z_i|=1\ (i\ne i_*),\ |z_{i_*}|\le1\}.
\]
This is a solid torus with boundary $\T^{J_0}$.

Suppose $z\in\pi_{J_0}^{-1}(C_{J_0})$.  If $|z_{i_*}|=1$, the
real-fibered property \eqref{eq:singular-real-fiber} gives
$z\in\T^n$. Otherwise, the
strong sign condition forces the $d+1$ coordinates
$(\log|z_j|)_{j\in I_0}$ to be nonzero and alternate in sign;
the other coordinates are zero.  Since
$\log|z_{i_*}|<0$ and $k_{i_*}<0$, these signs are exactly the
signs of $(k_j)_{j\in I_0}$. We have used the sign argument of
\cite[Lemma~4.30]{AKKV}, with $J,K,m$ there corresponding to
$I_0,J_0,i_*$ here.

\emph{Regularity of the form.}
For every $c$-subset $I$ the rational ambient form
\[
 \omega_{k,I}
 =
 z^{-k}\bigwedge_{i\in I}\frac{dz_i}{z_i}
\]
is therefore regular in a neighborhood of
$\pi_{J_0}^{-1}(C_{J_0})$.  For $i\in I$ with $k_i<0$, the factor
$z_i^{-k_i}\,dz_i/z_i=z_i^{|k_i|-1}\,dz_i$ can have a pole only at
$z_i=\infty$, and the preceding sign condition gives $|z_i|\le1$;
for $k_i>0$ its only possible pole is at $z_i=0$, and $|z_i|\ge1$.
Coordinates outside $I_0$ have modulus one, which also handles
$k_i=0$.  The same statement holds for the function factors
$z_j^{-k_j}$ with $j\notin I$.

\emph{Trace and Stokes' theorem.}
Let $\rho:\widetilde X\to X$ be a resolution of singularities,
an isomorphism over $X_{\rm sm}$. Let $P$ be the intersection of $X$
with the polar divisor of the rational form $\omega_{k,I}$ on
$(\PP^1)^n$. By the pole calculation,
$\pi_{J_0}^{-1}(C_{J_0})\cap P=\varnothing$. Properness of $\pi_{J_0}$
implies that $\pi_{J_0}(P)$ is closed and algebraic. Its complement
in $(\PP^1)^{J_0}$ is an algebraic open set $U$ containing $C_{J_0}$, and
$\rho^*\omega_{k,I}$ is holomorphic on
$(\pi_{J_0}\rho)^{-1}(U)$.  The map $\varpi=\pi_{J_0}\rho$ is proper over
$U$.  Lemma~\ref{lem:singular-trace} shows that
\[
 \alpha_{k,I}=\operatorname{Tr}_{\varpi}(\rho^*\omega_{k,I})
\]
is a holomorphic $c$-form on $U$. Being of top holomorphic degree,
it is closed, so Stokes' theorem gives
\[
 \int_{\T^{J_0}}\alpha_{k,I}=\int_{\partial C_{J_0}}\alpha_{k,I}=0.
\]

\emph{Comparison with the integral on the smooth locus.}
Remove the proper algebraic set consisting of the branch and
exceptional values of $\varpi$ and $\pi_{J_0}(X_{\mathrm{sing}})$.
Since $\T^{J_0}$ is maximal totally real, a nonzero polynomial
restricts to a nonzero real-analytic function on this connected torus,
whose zero set has Lebesgue measure zero. Thus the discarded set has
real $c$-dimensional measure zero in $\T^{J_0}$. Over its complement,
the inverse branches lie in $N$ and have the common orientation sign
$s(J_0)$.
The omitted preimage in $N$ has measure zero because $\pi_{J_0}$
is a local diffeomorphism there. Moreover, $|z^{-k}|=1$ on $N$, so
$|\omega_{k,I}|$ is a constant multiple of the finite positive measure
$\sigma_I$. The integral is therefore absolutely convergent, and
change of variables across the discarded null set gives
\[
 \int_N\omega_{k,I}
 =s(J_0)\int_{\T^{J_0}}\alpha_{k,I}=0.
\]

Finally,
$d\theta_I=(2\pi i)^{-c}\bigwedge_{i\in I}dz_i/z_i$.
The preceding vanishing for every $I$, followed by the finite
linear combination defining $\sigma_{X,L}$, proves the result.
\end{proof}

\subsection{Phase averaging and Fourier inversion}

We identify the torus Fourier coefficients by applying the area
formula to the phase map. Continuity will then recover the section
measure at every phase.

\begin{lemma}[Averaged Fourier coefficients]
\label{lem:singular-phase}
For $h\in\mathcal S(\R^d)$ define the continuous torus function
$F_h(\theta)=\nu_\theta(h)$.  Under
\eqref{eq:strong-var-singular},
\[
 \widehat F_h(k)
 =
 \widehat \sigma_{X,L}(k)\widehat h(-L^tk)
 \qquad(k\in\Z^n).
\]
\end{lemma}

\begin{proof}
We first take $h$ with compact support. For almost every phase,
no root in that support maps into $X_{\rm sing}(\T)$: the real singular
set has dimension at most $c-1$, and the exceptional phases lie in
the image of
\[
 X_{\rm sing}(\T)\times\operatorname{supp}(h)
 \longrightarrow\T^n,\qquad(t,x)\longmapsto t-Lx.
\]
Cover the compact real singular locus by finitely many bounded torus
coordinate charts. In these charts the phase map is Lipschitz on
the product with the compact support of $h$. Its source has Hausdorff
dimension at most $c-1+d=n-1$, so its image has Haar measure zero.

At a phase with no singular root in the support of $h$, each root
contributing to $\nu_\theta(h)$ is transverse and has local intersection
degree one, by Lemma~\ref{lem:singular-degree}.
The map
\[
 \Phi:N\times\R^d\longrightarrow\T^n,\qquad
 \Phi(t,x)=t-Lx
\]
is a local diffeomorphism, and its inverse fiber is
\[
 \Phi^{-1}(\theta)
 =\{(\theta+Lx,x):x\in\Lambda_\theta,
              \theta+Lx\in N\}.
\]
The absolute Jacobian of $\Phi$ is the density of $\sigma_{X,L}$,
by Lemma~\ref{lem:singular-torus-measure}.
Since each inverse image has degree one, the area formula gives
\begin{align*}
 \widehat F_h(k)
 &=
 \int_{\T^n}\nu_\theta(h)e^{-2\pi i\langle k,\theta\rangle}
       \,d\theta\\
 &=
 \int_N\int_{\R^d}
 h(x)e^{-2\pi i\langle k,t-Lx\rangle}\,dx\,d\sigma_{X,L}(t)\\
 &=
 \widehat \sigma_{X,L}(k)\widehat h(-L^tk).
\end{align*}
Uniform translation boundedness from
Lemma~\ref{lem:singular-degree} permits passage from compact
support to Schwartz functions.
\end{proof}

\begin{proof}[Proof of Theorem~\ref{thm:singular-section}
under the strong condition]
The cone $\{u:\varn(u)\le d-1\}$ is closed because $\varn$ is
lower semicontinuous. By \cite[Lemma~3.5]{AKKV}, it meets
$\ker L^t=(L\R^d)^\perp$ only at zero. Compactness of its intersection
with the unit sphere, followed by homogeneity, gives $a>0$ such that
\begin{equation}
\label{eq:singular-proper-cone}
 \|L^tu\|\ge a\|u\|
 \qquad\text{whenever }\varn(u)<d.
\end{equation}
It follows that $L^t$ is proper on the closed sign-variation cone.
The positive measure $\sigma_{X,L}$ satisfies
$|\widehat \sigma_{X,L}(k)|\le \sigma_{X,L}(\T^n)$.
By Lemmas~\ref{lem:singular-vanishing} and
\ref{lem:singular-phase}, the Fourier coefficients of $F_h$
vanish off that cone and are absolutely summable on it.
The Fourier series converges uniformly and equals $F_h$ at every
phase by uniqueness of Fourier coefficients for continuous torus
functions.

Take $h=\widehat f$.  With our Fourier convention,
$\widehat h(-L^tk)=f(L^tk)$, which gives
\eqref{eq:singular-summation}.
The left-hand side is absolutely convergent by uniform
translation boundedness.
On the right, \eqref{eq:singular-proper-cone} shows that
there are only $O((1+R)^n)$ integer indices with
$\varn(k)<d$ and $\|L^tk\|\le R$.  In particular, the
spectral set is locally finite, and each frequency has finitely many
preimages among these indices. After grouping equal frequencies and
discarding zero coefficients, the bound on $\widehat \sigma_{X,L}$
and the triangle inequality give spectral total variation
$O((1+R)^n)$. The remaining FQ conditions follow from
Lemma~\ref{lem:singular-degree}.
\end{proof}

\subsection{Reduction to ordinary sign variation}

We apply the compression proposition to obtain the ordinary
sign-variation condition. The kernel of the isogeny supplies periods
of the phase functions, which put their frequencies back in the
original frequency module.

\begin{proof}[Proof of Theorem~\ref{thm:singular-section}
in general]
Set $X^*=X\cap(\C^*)^n$ and take $U=\Gr_+(d,n)$. Since $L$ has
positive maximal minors and $\Gr_+(d,n)$ is open, $U$ is an open
neighborhood of $L\R^d$. The variety $X^*$ is reduced, pure of
codimension $d$, and invariant under reciprocal conjugation. By
\cite[Lemma~3.2]{AKKV}, a nonzero vector in a positive $d$-plane has
upper variation less than $d$. Thus the Lee--Yang condition gives
$\cA(X^*)\cap W\subseteq\{0\}$ for every $W\in U$.
By Proposition~\ref{prop:compression}, there are a nonsingular integer
matrix $S$, a matrix $\widetilde L=S^{-1}L$ with positive maximal
minors, and the preimage
\[
 Z=f_S^{-1}(X^*),\qquad f_S(z)_i=\prod_j z_j^{S_{ij}},
\]
whose projective closure satisfies
\eqref{eq:strong-var-singular}.

The phase identity in Proposition~\ref{prop:compression} gives, at
phase zero,
\[
 E_{\widetilde L}(x)\in Z
 \quad\Longleftrightarrow\quad
 f_S(E_{\widetilde L}(x))=E_L(x)\in X^*.
\]
The local intersection degrees also agree at every real phase by
the same proposition. To see this in logarithmic coordinates, let
$\mathcal Y_Z=\{y\in\C^n:\exp(2\pi iy)\in Z\}$ and
$\mathcal Y_X=\{y\in\C^n:\exp(2\pi iy)\in X\}$ be the logarithmic
lifts. Then $\mathcal Y_Z=S^{-1}\mathcal Y_X$, and $S$ identifies
the local varieties and the physical planes $\widetilde L\C^d$ and $L\C^d$.
The strong case applied to $(\overline Z,\widetilde L)$ shows that
$\nu_\theta$ is an FQ measure with spectrum in $\widetilde L^t\Z^n$;
we now show that the spectrum lies in $L^t\Z^n$.
For an arbitrary phase, choose a lift
$\theta\in\R^n$ and put $\theta'=S^{-1}\theta$ before passing to
the new torus.

To obtain the spectral inclusion in $L^t\Z^n$, fix a Schwartz
function $h$. The phase functions satisfy
\[
 F_h^{\rm new}(\theta')=F_h^{\rm old}(S\theta').
\]
It follows that $F_h^{\rm new}$ has all periods in
$S^{-1}\Z^n/\Z^n$; its Fourier coefficient at $k$ therefore vanishes
unless $k\in S^t\Z^n$.  The phase formula of
Lemma~\ref{lem:singular-phase}, applied with a Gaussian
$h$ whose Fourier transform never vanishes, gives
\[
 \widehat \sigma_{\overline Z,\widetilde L}(k)=0\qquad(k\notin S^t\Z^n).
\]
Writing $k=S^t\ell$, the surviving frequencies are
$\widetilde L^tk=L^t\ell$. More explicitly, the spectrum lies in
\begin{equation}
\label{eq:singular-weak-spectrum}
 \Sigma_{L,S}
 =
 \{L^t\ell:\ell\in\Z^n,\ \varn(S^t\ell)<d\}.
\end{equation}
It is locally finite by the proper-cone bound for $\widetilde L$,
since $S^t$ is invertible.  The same counting argument
gives spectral total variation $O((1+R)^n)$.
\end{proof}

\subsection{Effective cycles}

We now sum over the components of an effective cycle, keeping their
integer coefficients. This gives the forward construction for the
classification theorem; we then identify the components that meet
real-phase sections.

\begin{proposition}[Fourier quasicrystals from effective cycles]
\label{prop:integer-forward}
Let $\mathfrak X$ be an effective Lee--Yang cycle of codimension $d$
in $(\PP^1)^n$, and let $L\in\R^{n\times d}$ have positive maximal
minors. For every real phase $\theta$, its complex section is real
and locally finite. The measures $\nu^{\mathfrak X,L}_\theta$ form
a vaguely continuous, uniformly translation bounded family of FQ
measures, with spectra in $L^t\Z^n$ and spectral total variation
$O((1+R)^n)$ in $B_R$. In particular, every nonzero section is an
integer-mass FQ measure.
\end{proposition}
\begin{proof}
Partition the irreducible components into reciprocal-conjugation orbits.
The reduced union of one orbit is invariant, pure of codimension $d$,
and satisfies the sign inequality of $|\mathfrak X|$, so it is
Lee--Yang. By definition of the cycle, all its components have the
same coefficient $\varkappa_i$.

Apply Theorem~\ref{thm:singular-section} to this reduced orbit union,
and multiply its section measure by $\varkappa_i$. The reduced local
degree of the orbit union is the sum of the reduced local degrees of
its components: choose a common generic value for the finitely many
local finite maps and count their distinct sheets away from all component
intersections. Thus summing these orbit measures gives exactly
$\nu^{\mathfrak X}_\theta$ as defined above.

The finite sum is vaguely continuous and uniformly translation
bounded, and its Fourier transform is the corresponding sum of the
spectral measures in Theorem~\ref{thm:singular-section}. Their finite
union is locally finite, and the triangle inequality preserves the
spectral total-variation bound $O((1+R)^n)$. Grouping equal frequencies
and discarding zero coefficients leaves the spectrum in $L^t\Z^n$.
At every root, the mass is a sum of positive integers and is therefore
positive, which proves the assertion at every real phase.
\end{proof}

\begin{lemma}[Active components]
\label{lem:integer-active-components}
Let $X$ be Lee--Yang of pure dimension $c=n-d$. If an irreducible
component $X_i$ meets $\T^n$, then $X_i(\T)$ has real dimension $c$,
is Zariski dense in $X_i$, and $X_i$ is individually preserved by
reciprocal conjugation. Moreover, every analytic branch of $X_i$
at a torus point contains smooth real $c$-dimensional patches arbitrarily
near that point. Components with empty torus locus contribute zero to
all real-phase sections.
\end{lemma}
\begin{proof}
Choose a positive $d$-plane and a torus point of $X_i$. By
Lemma~\ref{lem:singular-degree}, the local logarithmic projection
along this plane is finite, and its entire fiber over each real
point of a complementary coordinate space is real. This follows
from the sign inequality on each local branch of $X_i$, as in the reality
argument for every real phase. Lemma~\ref{lem:integer-real-sheets}
then yields the asserted real patches. We may choose them away from
the singular and component-intersection loci, which have smaller
dimension. The identity principle and irreducibility imply that
these patches are Zariski dense in the component. Since the torus
is totally real, the real dimension is at most $c$ and hence equals $c$.

Reciprocal conjugation sends $X_i$ to an irreducible component containing
its dense real patch, so the two components coincide. Finally a real-phase
root is real by Lemma~\ref{lem:singular-degree}, and its exponential image
lies in $\T^n$. It cannot lie in a component with empty torus locus.
Deleting all such components preserves purity, reciprocal invariance,
and the required sign inequality of the remaining union.
\end{proof}

\subsection{Unweighted singular sections}

Two transverse branches can contribute to the same root, and counting
that root once changes the measure. We give a two-dimensional example
satisfying ordinary sign variation, so that no compression is needed.

\begin{proposition}[The local intersection degrees cannot be discarded]
\label{prop:singular-example}
There is a reduced Lee--Yang variety $X\subset(\PP^1)^3$
satisfying \eqref{eq:strong-var-singular} and a matrix
$L\in\R^{3\times2}$ with positive maximal minors for which the
section with its local intersection degrees is an FQ measure, but
its unweighted counting measure is not. One may take $X$ to be the
reduced union of the projective closures of
$t\mapsto(t,t^{-1},t)$ and $t\mapsto(t,t^{-2},t^3)$, and
\[
 L=\begin{pmatrix}1&0\\0&1\\-\sqrt2&\sqrt3\end{pmatrix}.
\]
\end{proposition}

\begin{proof}
Let $X_1$ and $X_2$ be the projective closures of
\[
 t\longmapsto(t,t^{-1},t),
 \qquad
 t\longmapsto(t,t^{-2},t^3),
 \qquad t\in\C^*,
\]
and let $X=X_1\cup X_2$.
Each map has first coordinate $t$, so it is an embedding
of $\PP^1$.  Off the torus, the logarithmic modulus
vectors are nonzero scalar multiples of $(1,-1,1)$
and $(1,-2,3)$, respectively.  They have two sign
changes; the same holds at the two boundary points,
using infinite logarithmic moduli.  Each component is
invariant under $z\mapsto1/\overline z$.
Both components are therefore strict Lee--Yang curves, and their
reduced union satisfies \eqref{eq:strong-var-singular}.  Their torus intersection is exactly
$(1,1,1)$, where their tangent directions are distinct. Their
projective closures also meet at $(0,\infty,0)$ and $(\infty,0,\infty)$;
these boundary intersections do not affect the torus sections.

The maximal minors of the displayed matrix are $1,\sqrt3,\sqrt2$.
The two sections are lattices
$\Lambda_j=D_j^{-1}\Z^2$, where
\[
 D_1=
 \begin{pmatrix}
 1&1\\
 -(1+\sqrt2)&\sqrt3
 \end{pmatrix},
 \qquad
 D_2=
 \begin{pmatrix}
 2&1\\
 -(3+\sqrt2)&\sqrt3
 \end{pmatrix}.
\]
In the algebraic torus, the first curve is defined by $z_1z_2=1$
and $z_3=z_1$, and the second by $z_1^2z_2=1$ and $z_3=z_1^3$.
Pulling back these equations gives $D_jx\in\Z^2$.

If $x\in\Lambda_1\cap\Lambda_2$, its torus image is
$(1,1,1)$.  Hence $x_1,x_2$ are integers and
$-\sqrt2x_1+\sqrt3x_2$ is an integer.  The rational
independence of $1,\sqrt2,\sqrt3$ gives $x=0$.
Each component meets the section transversely: the tangent directions
$(1,-1,1)$ and $(1,-2,3)$ both have $\varn=2=d$, so neither lies in
$L\R^2$ by \cite[Lemma~3.2]{AKKV}. Equivalently, $\det D_j\ne0$.
Additivity of the local intersection degree over the two branches
therefore gives the section
measure
\[
 \nu=\sum_{x\in\Lambda_1}\delta_x+
      \sum_{x\in\Lambda_2}\delta_x,
\]
with weight two at the origin and weight one elsewhere.
Poisson summation gives the exact identity
\[
 \widehat\nu
 =
 (1+\sqrt2+\sqrt3)
       \sum_{k\in\Z^2}\delta_{D_1^tk}
 +
 (3+\sqrt2+2\sqrt3)
       \sum_{k\in\Z^2}\delta_{D_2^tk}.
\]
This is a locally finite measure of polynomial growth.

The unweighted counting measure of
$\Lambda_1\cup\Lambda_2$ is $\nu-\delta_0$.
Its Fourier transform is $\widehat\nu-\mathcal L^2$,
where $\mathcal L^2$ denotes Lebesgue measure.
Outside the two dual lattices the Fourier transform equals
$-\mathcal L^2$, so it cannot have locally finite discrete support.
The unweighted section is therefore not an FQ-set.
\end{proof}

\begin{remark}[The collision at the node]\label{rem:singular-node}
In local logarithmic coordinates $(u,v,w)$ centered at $(1,1,1)$,
the two branches are $(s,-s,s)$ and $(s,-2s,3s)$. The coordinates
\[
 a=2u+v,\qquad b=-u-v,\qquad h=u+2v+w
\]
have change-of-basis determinant $-1$ and identify their union with
$\{h=0,ab=0\}$. Thus $X$ has an ordinary node there and is not strict.
Each branch contributes one root near the origin under a small phase
change. The roots can agree only if $\theta+Lx=0$ in the local logarithmic
chart, since the two branches intersect only at the origin there.
This requires $\theta\in L\R^2$, a proper plane in $\R^3$. Thus
the two roots are distinct at a generic nearby phase and coalesce
at phase zero. The measure counted with local intersection degrees
retains mass two, whereas the unweighted measure has mass one at the
limit and is not vaguely continuous there. Continuity of $\nu_\theta$
is what extends the almost-everywhere phase identity to every phase.
\end{remark}

\section{Equations of the support}
\label{sec:q4-criterion}

We determine all trigonometric polynomials with arbitrary real frequencies
that vanish on the support. We first identify the closure of the root
phases in an enlarged torus, then separate characters on the kernel of
its projection. These equations also explain the distinction between
local intersection degrees and colengths.

Let $\mu$ be an integer-mass FQ measure with support $\Lambda=\supp\mu$
and basis matrix $L$ for $\Gamma$. Write $M$ and
$Y=\overline M^{\mathrm{Zar}}\subset(\C^*)^n$ for its canonical set
and variety. The rows of $L$ identify $\C[\Gamma]$ with the Laurent
polynomial ring of $(\C^*)^n$. If $\Gamma'\supset\Gamma$ is finitely
generated and $L'$ is a basis matrix for it, then
$L=AL'$ for an integer matrix $A\in\Z^{n\times n'}$ of rank $n$.
The dual map is
\[
 \rho:\T^{n'}\longrightarrow\T^n,\qquad \rho(\theta')=A\theta'.
\]
Smith normal form expresses it as coordinate power maps followed by
projection, so it is surjective. Writing $H=\ker\rho$, we have
\[
 H\cong\mathbb U_{\varsigma_1}\times\cdots\times\mathbb U_{\varsigma_n}
                    \times\T^{n'-n},
\]
where $\varsigma_1,\ldots,\varsigma_n$ are the nonzero Smith invariant
factors and $\mathbb U_j$ is the group of $j$th roots of unity.
The same notation $\rho$ denotes the monomial extension to the complex
algebraic tori, $\rho(z')_i=\prod_j(z'_j)^{A_{ij}}$. Then
$\rho\circ E_{L'}=E_L$, and $\rho^*$ is the inclusion
$\C[\Gamma]\subset\C[\Gamma']$.

For a finitely generated subgroup $\Gamma'\subset\R^d$ with basis
matrix $L'$, write an element of its group algebra as
$P=\sum_{\gamma\in\Gamma'}a_\gamma u^\gamma$. The associated
trigonometric polynomial is
\[
 \mathcal E_P(x)=\sum_\gamma a_\gamma e^{2\pi i\gamma\cdot x},
 \qquad x\in\C^d.
\]
Under the Laurent identification, $\mathcal E_P=P\circ E_{L'}$.
The association is injective: on a real line separating the finitely
many frequencies, linear independence follows from the Vandermonde
determinant.

\begin{proposition}[The vanishing ideal of the support]
\label{prop:canonical-ideal}
For every finitely generated subgroup $\Gamma'\subset\R^d$
containing $\Gamma$,
\begin{equation}
 \{P\in\C[\Gamma']:\mathcal E_P|_\Lambda=0\}
       =I(Y)\C[\Gamma'].
 \label{eq:q4-ideal-extension}
\end{equation}
Consequently the common complex zero set of all trigonometric
polynomials with real frequencies vanishing on $\Lambda$ is
$E_L^{-1}(Y)$.
\end{proposition}

Theorem~\ref{thm:universal-classification} already gives
$E_L^{-1}(Y)=\Lambda$ and finite real-valued equations in $\Gamma$.
The proposition identifies the entire ideal after adjoining any
finitely many real frequencies.

\begin{lemma}[Density in an enlarged torus]\label{lem:extended-phase-density}
Let $L'$ be a basis matrix for a finitely generated group
$\Gamma'\supset\Gamma$, and let $\rho:\T^{n'}\to\T^n$ be dual to the
inclusion, where $n'=\operatorname{rank}\Gamma'$. Then
\begin{equation}
 \overline{\{L'\lambda\bmod\Z^{n'}:\lambda\in\Lambda\}}
       =\rho^{-1}(M).
 \label{eq:q4-phase-density}
\end{equation}
\end{lemma}
\begin{proof}
By Theorem~\ref{thm:all-fq-toral-cycles}(c), the canonical set $M$ is
compact. For $\lambda\in\Lambda$, the measure
$\mu_{L\lambda}=\tau_{-\lambda}\mu$ has an atom at zero, which proves
one inclusion in \eqref{eq:q4-phase-density}.

Conversely, let $\theta'\in\rho^{-1}(M)$. The flow
$x\mapsto L'x\bmod\Z^{n'}$ is dense because the basis vectors of
$\Gamma'$ are integer-independent in $\R^d$. We can therefore choose
$x_j$ with $L'x_j\to\theta'$. By
Definition~\ref{def:ac-hull-measures},
\[
 \tau_{-x_j}\mu=\mu_{Lx_j}\longrightarrow\mu_{\rho\theta'}
\]
vaguely. The limit has an atom at zero, so a nonnegative smooth test
function supported in any ball about zero and positive at zero has positive
limiting integral. It follows that $\operatorname{dist}(0,\Lambda-x_j)\to0$.
Choosing $y_j\in\Lambda-x_j$ tending to zero and setting $\lambda_j=x_j+y_j$,
we obtain $\lambda_j\in\Lambda$ and
\[
 L'\lambda_j=L'x_j+L'y_j\longrightarrow\theta'\pmod{\Z^{n'}}.
\]
\end{proof}

\begin{lemma}[Extension of the vanishing ideal]\label{lem:torus-ideal-extension}
Let $\Gamma\subset\Gamma'\subset\R^d$ be finitely generated groups
with basis matrices $L$ and $L'$, and let
$\rho:\T^{n'}\to\T^n$ be the dual surjection. For any subset
$M\subset\T^n$, let $Y$ be its Zariski closure in $(\C^*)^n$.
The Laurent polynomials in $\C[\Gamma']$ vanishing on $\rho^{-1}(M)$
form the ideal $I(Y)\C[\Gamma']$.
\end{lemma}
\begin{proof}
Pullback proves one inclusion. For the reverse inclusion, we group the
exponents of a Laurent polynomial $P$ vanishing on $\rho^{-1}(M)$ by
cosets in $\Gamma'/\Gamma$:
\begin{equation}
 P=\sum_{\beta}u^\beta\rho^*P_\beta,
 \qquad P_\beta\in\C[\Gamma].
 \label{eq:q4-coset-expansion}
\end{equation}
Once the distinct coset representatives are fixed, this finite expansion
is unique. In multiplicative notation, $u^\beta$ is a character of the
compact torus. Put $H=\ker\rho$. The characters $u^\beta|_{H}$ are
distinct because the annihilator of $H$ in $\Gamma'$ is $\Gamma$.
Indeed, in Smith normal form, triviality on the finite factors requires
exponents divisible by $\varsigma_i$, while triviality on the remaining
torus factors requires zero exponents. These are precisely the characters
from $\Gamma$; the finite factors account for torsion in $\Gamma'/\Gamma$.

For $z'\in\rho^{-1}(M)$ and $h\in H$,
\[
 0=P(z'h)=
 \sum_\beta u^\beta(z')u^\beta(h)P_\beta(\rho z').
\]
Integrating against the conjugate of each character over
$H$, with normalized Haar measure, gives
$u^\beta(z')P_\beta(\rho z')=0$ for every $\beta$.
Monomials never vanish on the torus, and $\rho$ is surjective, so each
$P_\beta$ vanishes on $M$ and hence on $Y$. This proves
$P_\beta\in I(Y)$, and \eqref{eq:q4-coset-expansion} concludes the proof.
\end{proof}

\begin{proof}[Proof of Proposition~\ref{prop:canonical-ideal}]
Every element of $I(Y)\C[\Gamma']$ vanishes on $\Lambda$.
Conversely, Lemma~\ref{lem:extended-phase-density} shows that a
polynomial vanishing on $\Lambda$ vanishes on $\rho^{-1}(M)$, and
Lemma~\ref{lem:torus-ideal-extension} gives
\eqref{eq:q4-ideal-extension}.

For a trigonometric polynomial with real frequencies vanishing on
$\Lambda$, adjoin its frequencies to $\Gamma$ and write it as
$\mathcal E_P$ in the resulting group algebra. The ideal identity implies
$P\in I(Y)\C[\Gamma']$, and the complex torus morphism $\rho$ satisfies
\[
 \rho(E_{L'}(x))=E_L(x),\qquad x\in\C^d,
\]
so $\mathcal E_P$ vanishes on $E_L^{-1}(Y)$. Conversely, the finite
real-valued generators in Proposition~\ref{cor:canonical-real-equations}
already have common complex zero set $E_L^{-1}(Y)$ after composition
with $E_L$. Allowing additional frequencies therefore leaves this
common zero set unchanged.
\end{proof}

\subsection{Masses and equation multiplicities}
\label{sec:equation-multiplicities}

Equations determine the support, but their local colengths need not
recover the masses. The following example uses two transverse surfaces
and a dense physical flow to exhibit this discrepancy.

\begin{proposition}[Two transverse surfaces and equation colengths]
\label{prop:integer-equation-obstruction}
Put
\[
 \mathsf q(a)=(1,-a,a^2,-a^3)^t,\qquad
 W=\operatorname{span}_{\R}\{\mathsf q(1),\mathsf q(2)\},\qquad
 W'=\operatorname{span}_{\R}\{\mathsf q(3),\mathsf q(5)\},
\]
and let $T_A,T_{A'}\subset(\C^*)^4$ be the connected subtori with
these logarithmic tangent spaces. Their projective closures are strict
Lee--Yang of codimension $2$. Their reduced union $X$ is Lee--Yang,
satisfies ordinary sign variation off the torus, and is not strict.
For
\[
 L=L_{\mathrm{irr}}=
 \begin{pmatrix}1&0\\1&1\\1&\sqrt2\\1&\sqrt3\end{pmatrix},
\]
all maximal minors are positive and the physical flow is dense in $\T^4$.
Its section measure $\mu=\nu^{X,L}_0$ is an integer-mass FQ measure with
$\mu(\{0\})=2$. Every finite family of trigonometric polynomials
$F_1,\ldots,F_q$, with arbitrary real frequencies, vanishing on $\supp\mu$
satisfies
\[
 \dim_\C\mathcal O_{\C^2,0}/(F_1,\ldots,F_q)
       \ge3>2=\mu(\{0\}).
\]
The inequality allows the left side to be infinite. Thus weighted local
intersection degrees cannot in general be realized as equation colengths,
already for $c=2$.
\end{proposition}
\begin{proof}
We construct two transverse subtori and use density of the physical flow
to force every equation of the support into their common ideal.

\emph{Sign variation on the two planes.}
For $0<a<b$ write a nonzero vector in their span as
$A\mathsf q(a)+B\mathsf q(b)$. Its coordinate of index $j=0,1,2,3$ is
$(-1)^js_j$, where
\[
 s_j=a^j\bigl(A+B(b/a)^j\bigr).
\]
The nonzero $s_j$ change sign at most once, and at most one
$s_j$ vanishes. With no zero entry, alternating the signs leaves at least
two sign changes. If a zero is at an endpoint, the other three $s_j$
have the same sign, so the remaining coordinates alternate. At an interior
zero, the $s_j$ on either side have opposite signs, and deleting the zero
again leaves three alternating coordinates. Every nonzero vector in
$W$ or $W'$ therefore has ordinary variation at least $2$.

Since the amoeba of a subtorus is its real logarithmic tangent plane,
this proves the inequality in the algebraic torus. For a sequence
converging to a projective boundary point, the logarithmic norms tend to
infinity, and a subsequence of the normalized logarithmic vectors tends
to a nonzero vector in $W$ or $W'$. Every nonzero coordinate of this
limit has the sign of an infinite logarithmic modulus at the boundary
point. These signs form a subsequence of the boundary sign vector;
inserting the other coordinates cannot decrease its variation, which is
therefore at least $2$.

Each subtorus is smooth and closed in the algebraic torus, so its
projective closure remains smooth along the real torus. The defining
characters have integer exponents and value $1$, which makes both
subtori and their closures invariant under reciprocal conjugation.
Together with the boundary inequality, this proves that both closures
are strict Lee--Yang by Definition~\ref{def:lee-yang}.

\emph{Transversality and the section masses.}
The alternating Vandermonde determinant is
\[
 \det[\mathsf q(1)\ \mathsf q(2)\ \mathsf q(3)\ \mathsf q(5)]=48.
\]
It follows that $W\oplus W'=\R^4$, so the two subtori meet in a finite
algebraic subgroup and are transverse at every intersection point.
Their reduced union is singular at the identity, where it has two
branches; hence $X$ is Lee--Yang but not strict.

Every nonzero vector in a positive $2$-plane has upper variation less
than $2$ by \cite[Lemma~3.2]{AKKV}. The preceding ordinary-variation
bound implies $L\R^2\cap W=L\R^2\cap W'=\{0\}$.
Each branch has local intersection degree one at the identity, so the
union has degree two there. By Theorem~\ref{thm:singular-section}, its
section measure is an FQ measure.

For explicit equations, we use the saturated integer annihilator matrices
\[
 \mathsf C_{12}=
 \begin{pmatrix}2&3&1&0\\-6&-7&0&1\end{pmatrix},
 \qquad
 \mathsf C_{35}=
 \begin{pmatrix}15&8&1&0\\-120&-49&0&1\end{pmatrix}.
\]
Multiplication yields
$\mathsf C_{12}\mathsf q(1)=\mathsf C_{12}\mathsf q(2)=0$ and
$\mathsf C_{35}\mathsf q(3)=\mathsf C_{35}\mathsf q(5)=0$.
Since the last two columns of either matrix form the identity, the
character equations $z^{\mathsf C_{12,j}}=1$ and $z^{\mathsf C_{35,j}}=1$,
with exponents given by the rows, define connected subtori with the
required dimensions and tangent spaces. Their sections are the lattices
\[
 \Lambda_1=(\mathsf C_{12}L)^{-1}\Z^2,
 \qquad
 \Lambda_2=(\mathsf C_{35}L)^{-1}\Z^2,
 \qquad
 \mu=\sum_{x\in\Lambda_1}\delta_x+\sum_{x\in\Lambda_2}\delta_x.
\]
Both matrices are invertible because the physical plane is transverse
to each subtorus.

\emph{The dense physical flow.}
For $L=L_{\mathrm{irr}}$, the slopes $0<1<\sqrt2<\sqrt3$ imply that
all maximal minors are positive. If $L^tk=0$ for $k\in\Z^4$, the
second coordinate reads $k_2+\sqrt2k_3+\sqrt3k_4=0$.
The rational independence of $1,\sqrt2,\sqrt3$ implies $k_2=k_3=k_4=0$,
and the first coordinate then implies $k_1=0$. The rows thus form an
integer-independent basis of $\Gamma_{\mathrm{amb}}=L^t\Z^4$, and the
physical flow is dense in $\T^4$.

Write $Y_{\mathrm{amb}}=T_A\cup T_{A'}$. Since Poisson summation for
each lattice has positive coefficients, the spectral group of $\mu$ is
\[
 L^t(\mathsf C_{12}^t\Z^2+\mathsf C_{35}^t\Z^2).
\]
Its index in $\Gamma_{\mathrm{amb}}=L^t\Z^4$ is
\[
 \left|\det\begin{pmatrix}\mathsf C_{12}\\\mathsf C_{35}\end{pmatrix}\right|=24.
\]
The ambient torus $\T^4$ is therefore a $24$-fold cover of the canonical
torus, and we work in this cover.
For any finitely generated $\Gamma'\supset\Gamma_{\mathrm{amb}}$, take
a basis matrix $L'$ and the dual surjection
$\rho:\T^{n'}\to\T^4$. The section measures in
Theorem~\ref{thm:singular-section} satisfy
$\nu_{Lx}=\tau_{-x}\nu_0$ and are vaguely continuous. Therefore
\[
 \overline{\{L'\lambda:\lambda\in\supp\mu\}}
       =\rho^{-1}\bigl(Y_{\mathrm{amb}}(\T)\bigr).
\]
The section identity and closedness prove the forward inclusion.
Conversely, approximate a point on the right by $L'x_j$.
The measures $\tau_{-x_j}\mu$ converge to a real-phase section measure
having a positive atom at zero. Positive test functions give support
points $y_j$ of these translates with $y_j\to0$, so
$L'(x_j+y_j)$ tends to the chosen point. This proves the other inclusion.

The torus locus is Zariski dense in each of the two subtori: pull a
Laurent polynomial back to $(\C^*)^2$ and use Fourier uniqueness on
$(S^1)^2$. Applying Lemma~\ref{lem:torus-ideal-extension} to this
ambient group and the preceding phase-density equality gives
\[
 \{P\in\C[\Gamma']:\mathcal E_P|_{\supp\mu}=0\}
     =I(Y_{\mathrm{amb}})\C[\Gamma'].
\]
Adjoining its finitely many frequencies therefore places every
polynomial in the proposition in such an extended ideal.

\emph{The local ideal and colength.}
In logarithmic coordinates at the identity, the germ of
$Y_{\mathrm{amb}}$ is $W_\C\cup W'_\C$. The two planes are complementary,
so a linear change of coordinates writes their ideals as
$(u_1,u_2)$ and $(v_1,v_2)$. Their intersection in the convergent
power-series ring equals their product: a series in the intersection
has no monomial involving only $u$ or only $v$, and grouping the remaining
monomials by their first $u$ and first $v$ factors expresses it as a sum
of $u_iv_j$ times convergent series. Hence the reduced germ ideal is
\[
 I(W_\C)\cap I(W'_\C)=I(W_\C)I(W'_\C).
\]
On the physical plane $L\C^2$, each linear ideal restricts to the
maximal ideal $\mathfrak m_0$: the two normal linear forms remain
independent because $L\C^2$ is complementary to the branch. The product
ideal therefore restricts to $\mathfrak m_0^2$.

Every Laurent function vanishing on $Y_{\mathrm{amb}}$ has its germ in
this product ideal. An added character restricts to a holomorphic unit,
so the extended ideal identity places each $F_i$ in $\mathfrak m_0^2$.
We obtain a surjection
\[
 \mathcal O_{\C^2,0}/(F_1,\ldots,F_q)
   \longrightarrow\mathcal O_{\C^2,0}/\mathfrak m_0^2.
\]
The target has basis $1,x_1,x_2$ and dimension $3$, proving the
inequality. The two-branch projection has local degree $2$, whereas the
restricted reduced ideal $\mathfrak m_0^2$ has colength $3$,
in accord with the non-Cohen--Macaulay case of
Remark~\ref{rem:degree-length}.
\end{proof}

\begin{remark}
For the rational matrix
\[
 L_{\mathrm{rat}}=
 \begin{pmatrix}1&0\\1&1\\1&2\\1&3\end{pmatrix},
\]
all maximal minors are positive, since $0<1<2<3$. The determinants
$\det(\mathsf C_{12}L_{\mathrm{rat}})=36$ and
$\det(\mathsf C_{35}L_{\mathrm{rat}})=576$ show that the two section
lattices lie in $(1/576)\Z^2$, so
$F_1(x)=\sin(2\pi\,576x_1)$ and $F_2(x)=\sin(2\pi\,576x_2)$ vanish on
the support and have independent derivatives at zero, giving local
colength $1$ although the mass there is $2$.
Here the physical flow is not dense in $\T^4$; the spectral group has
rank $2$, the canonical torus is two-dimensional, and the canonical
variety is zero-dimensional.
\end{remark}

\section{Proofs of the main theorems}
\label{sec:main-proofs}

The general classification is proved in
Theorem~\ref{thm:universal-classification}. We give its Delone and
unit-mass specializations here; the latter uses density of the root
phases to determine every component coefficient.
Theorem~\ref{thm:main-height} follows from
Corollary~\ref{cor:universal-linear-height}.

\begin{proof}[Proof of Theorem~\ref{thm:main-delone}]
By Lemma~\ref{lem:ac-translation-bound}, the support is relatively
dense, so uniform discreteness and the Delone property are equivalent
for an integer-mass FQ measure. For uniformly discrete support, the
representing cycle of
Theorem~\ref{thm:universal-classification} has strict support by
Lemma~\ref{lem:integer-ud-smoothness} and
Proposition~\ref{prop:compression}. Conversely, a section of a cycle
with strict support is FQ by Proposition~\ref{prop:integer-forward}.
Its support is the section of the reduced variety and is Delone by
\cite[Theorem~2.3]{AKKV}.
For the counting measure of a Delone FQ-set,
$\Lambda=\supp\nu^{\mathfrak X,\widetilde L}_0
=\Lambda(|\mathfrak X|,\widetilde L)$, with $|\mathfrak X|$ strict.
This answers Question~(1) of AKKV.
\end{proof}

\begin{proof}[Proof of Theorem~\ref{thm:main-equations}]
The integer-mass assertions follow from
Theorem~\ref{thm:universal-classification}. For the unit-mass
specialization, suppose first that $c=0$. The set $M$ is finite, and
density of $L\Lambda$ in $M$ implies that every point of $M$ is the phase
of a root of the original measure. That root has mass one, so its component
coefficient is one. If $c>0$, every $Y_i$ contains a smooth real patch
disjoint from the other components, with local intersection degree one.
By Lemma~\ref{lem:ac-transverse-strata}, the root phases $L\Lambda$ are
dense in $M$, so the patch contains such a phase and the mass identity
implies $\varkappa_i=1$. All component coefficients are therefore one.
At every root of the original measure, the local intersection degree
equals its mass, which is one; by Lemma~\ref{lem:local-degree-one}, the
variety is smooth and the section is transverse there.
Compression preserves these properties and the reduced coefficients.
The unit-mass specialization follows. As
Remark~\ref{rem:smooth-cycle-coefficients} shows, smoothness alone does
not force the cycle coefficients to be one.
\end{proof}

\begin{corollary}[Geometric classification]
\label{cor:geometric-classification}
The nonempty Delone Fourier quasicrystals in $\R^d$ are exactly the
sections $\Lambda(X,L)$ of strict Lee--Yang varieties by matrices with
positive maximal minors.

The nonempty Fourier quasicrystals with unit masses are exactly the
nonempty sections $\Lambda(X,L)$ of Lee--Yang varieties by such matrices for
which every physical root has local intersection degree one.
Equivalently, $X$ is smooth at $E_L(x)$ for every physical root $x$;
the section is then transverse there by
Theorem~\ref{thm:singular-section}(1) and
Lemma~\ref{lem:local-degree-one}. The representing variety may be required
to satisfy ordinary sign variation off the torus.
\end{corollary}

\begin{proof}
Necessity follows from the unit-mass specializations of
Theorems~\ref{thm:main-delone} and~\ref{thm:main-equations}.
For strict varieties, sufficiency is \cite[Theorem~2.3]{AKKV}.
For a general Lee--Yang variety, Theorem~\ref{thm:singular-section}
produces the FQ measure weighted by local intersection degrees, which
is the counting measure when all these degrees are one.

A local projection of degree one onto a smooth base is a local
isomorphism by Lemma~\ref{lem:local-degree-one}, so the variety is smooth
and the section is transverse. The converse is the local inverse
function theorem.
\end{proof}

\begin{proof}[Proof of Theorem~\ref{thm:main-singular}]
Proposition~\ref{prop:integer-forward} proves the forward assertion,
and Theorem~\ref{thm:universal-classification} proves the converse.
Example~\ref{ex:two-circles} and Proposition~\ref{prop:singular-example}
show that discarding degrees can destroy the FQ property.
\end{proof}

\section{Further questions}\label{sec:further-questions}

The first question concerns when local colength can recover the masses
supplied by local intersection degree.

\begin{question}[Masses and colengths]\label{qu:mass-colengths}
Let $\mu$ be an integer-mass FQ measure, and suppose that the local
rings of its canonical variety $Y$ at points of $M$ are Cohen--Macaulay.
Is there a finite family of trigonometric polynomials whose common complex
zero set is $\supp\mu$ and whose local colength at every point of this
support equals its mass?
\end{question}

The local rings are Cohen--Macaulay when $c=1$.
Remark~\ref{rem:degree-length} compares degree with colength;
Proposition~\ref{prop:integer-equation-obstruction} shows that, in
general, equations cannot realize the masses.

\begin{question}[The number of equations]\label{qu:number-equations}
For an integer-mass FQ measure on $\R^d$ with spectral group $\Gamma$,
can the number of real-valued trigonometric equations with frequencies
in $\Gamma$ in Theorem~\ref{thm:main-equations} always be taken equal
to $d$?
\end{question}

For $d=1$, the canonical variety $Y$ is a reduced hypersurface in
$(\C^*)^n$, so its Laurent ideal is principal, say $I(Y)=(P)$.
The involution of Proposition~\ref{cor:canonical-real-equations}
preserves $I(Y)$, so $P^\dagger$ differs from $P$ by a Laurent unit.
The identity $\mathcal E_{PP^\dagger}=|\mathcal E_P|^2$ on $\R$
therefore gives one real-valued trigonometric polynomial with frequencies
in $\Gamma$ and complex zero set $E_L^{-1}(Y)=\Lambda$.

\begin{question}[Effective bounds]\label{qu:effective-bounds}
For an integer-mass FQ measure, can the rank of the spectral group be
bounded effectively in terms of the data in \eqref{eq:fq-measure-growth}?
Can one choose a basis of the spectral group for which the constant
$C_L$ in Theorem~\ref{thm:main-height} is also bounded effectively
in terms of these data?
\end{question}

The basis cannot be chosen arbitrarily: for the Dirac comb of $\Z^2$,
which is its own Fourier transform, the bases with rows $(1,0)$ and
$(q,1)$, $q\in\N$, give the spectral point $(0,1)$ the coordinates
$(-q,1)$.

\begin{question}[An intrinsic representative]\label{qu:intrinsic-representative}
Can the effective Lee--Yang cycle representing an integer-mass FQ measure
be chosen intrinsically, independently of the basis $L$ and the
compression matrix $S$?
\end{question}

By Proposition~\ref{prop:canonical-ideal}, the variety $Y$ in
$\operatorname{Hom}(\Gamma,\C^*)$ is intrinsic. Its coordinates in
$(\C^*)^n$ depend on $L$, while the Lee--Yang representative also
depends on $S$.

Question~(5) of \cite[\S2.4]{AKKV} asks for a Lee--Yang representation
of the zero set of an arbitrary real-rooted system of trigonometric
equations. Our inverse theorem starts from an integer-mass FQ measure
on that zero set; this additional hypothesis leaves their question open.

\section*{Use of AI}

The principal mathematical ideas, constructions and proofs were generated
by GPT, including the comparison of character counts in automatic
algebraicity, the support-height theorem for holonomic sequences, and
the reconstruction of branched fibers from their moments. Claude
reviewed the mathematics and proposed the extension from unit masses
to positive integer masses; GPT generated its statements and proofs,
which Claude reviewed. Claude also proposed the example in
Proposition~\ref{prop:integer-equation-obstruction}; GPT identified the
need for a dense physical flow and supplied the correction. Codex
wrote and revised the exposition, typeset the paper, compared cited
statements with their sources, and carried out computational checks.
The author selected the questions, organized the reviews and revisions,
and takes responsibility for the entire content of the paper.

\end{document}